\documentclass{amsart}

\usepackage{amsthm}
\usepackage{amsmath}
\usepackage{amssymb}

\usepackage{geometry}
\usepackage{dsfont}
\usepackage{MnSymbol}
\usepackage{url}
\usepackage[hidelinks]{hyperref}
\usepackage{xcolor}
\usepackage{enumitem}
\usepackage{lipsum}

\usepackage{indentfirst}

\newcommand{\N}{\mathbb{N}}
\newcommand{\Z}{\mathbb{Z}}
\newcommand{\Q}{\mathbb{Q}}
\newcommand{\R}{\mathbb{R}}
\newcommand{\C}{\mathbb{C}}
\newcommand{\HH}{\mathbb{H}}
\newcommand{\F}{\mathbb{F}}

\newtheorem{Theorem}{Theorem}[section]
\newtheorem{Corollary}[Theorem]{Corollary}
\newtheorem{Proposition}[Theorem]{Proposition}
\newtheorem{Lemma}[Theorem]{Lemma}

\newtheorem{Algorithm}{Algorithm}

\theoremstyle{remark}
\newtheorem{Remark}[Theorem]{Remark}

\theoremstyle{definition}
\newtheorem{Definition}{Definition}[section]

\usepackage[
backend=bibtex,        
sorting=nyt,          
citestyle=alphabetic,
bibstyle=alphabetic,  
isbn=false, 
url=false, 
doi=false, 
eprint=false,
]{biblatex}
\AtEveryBibitem{%
	\clearfield{note}%
}

\let\oldtocsection=\tocsection
\let\oldtocsubsection=\tocsubsection
\let\oldtocsubsubsection=\tocsubsubsection

\renewcommand{\tocsection}[2]{\hspace{0em}\oldtocsection{#1}{#2}}
\renewcommand{\tocsubsection}[2]{\hspace{1em}\oldtocsubsection{#1}{#2}}
\renewcommand{\tocsubsubsection}[2]{\hspace{2em}\oldtocsubsubsection{#1}{#2}}

\title[Bowditch representations in Gromov-hyperbolic spaces]{\textsc{Bowditch representations in Gromov-hyperbolic spaces : characterizations, dynamics of $\mathrm{Out}(\F_2)$ and recognition}}
\author{Suzanne Schlich }

\begin{document}
	
	\maketitle
	
	\begin{abstract}
		We study a generalization of the $BQ$-conditions, introduced by Bowditch and further developed by Tan-Wong-Zhang, for representations of the free group of rank two into isometry groups of Gromov-hyperbolic spaces. We show the existence of an explicit constant 
	~$K_\delta$, depending only on the hyperbolicity constant $\delta$ of the space, such that the hyperbolicity of the images of primitive elements together with the finiteness of the set of (conjugacy classes of) primitive elements whose images have lengths bounded by $K_\delta$ imply a linear growth of the lengths with respect to the word length on primitive elements. We give several characterizations of Bowditch representations, and the framework developed allows us to prove that they form an open domain of discontinuity in the character variety. 
	As a corollary, we also obtain a new characterization of primitive-stable representations, introduced by Minsky. 
	Finally, we explain how our results can be used to obtain finite certificate for the recognition of Bowditch representations. 
	\end{abstract}
	
	\tableofcontents
	
	\section{Introduction}
	
	Let $\Gamma$ be a finitely generated group and $\mathcal{X}$ be a metric space. The group of isometries $\mathrm{Isom}(\mathcal{X})$ of $\mathcal{X}$ endowed with the compact-open topology is a topological group and acts by conjugation on the space of representations $\mathrm{Hom}(\Gamma,\mathrm{Isom}(\mathcal{X}))$ from $\Gamma$ to $\mathrm{Isom}(\mathcal{X})$. On may consider the Hausdorffization of the quotient of  $\mathrm{Hom}(\Gamma,\mathrm{Isom}(\mathcal{X}))$ by the action of $\mathrm{Isom}(\mathcal{X})$ (that is, its largest Hausdorff quotient). It is commonly referred to as the character variety and denoted $\chi(\F_2,\mathrm{Isom}(\mathcal{X}))$. Character varieties have been extensively studied, from many different points of view including their geometric, dynamical, algebraic (when $G$ is an algebraic group) or analytical aspects. Numerous dynamical questions on character varieties still remain widely open. In this paper, we will study some specific subspaces of the character varieties $\chi(\F_2,\mathrm{Isom}(\mathcal{X}))$, where $\F_2$ is the free group on two generators and $\mathcal{X}$ is a $\delta$-hyperbolic space, consisting of representations with  good geometric behavior and which will give informations about the dynamic of the outer automorphism group $\mathrm{Out}(\F_2)$ of $\F_2$ on $\chi(\F_2,\mathrm{Isom}(\mathcal{X}))$. 
	
	An important class of representations from $\Gamma$ to $\mathrm{Isom}(\mathcal{X})$ is the set of  \emph{convex-cocompact} representations, which we define in this setting to be the set of representations $\rho$ so that the map from $\Gamma$ to $\mathcal{X}$ which sends $\gamma \in \Gamma$ to $\rho(\gamma)x_0 \in \mathcal{X}$ is a quasi-isometric embedding, for $x_0$ any basepoint in the space $\mathcal{X}$. The finitely generated group $\Gamma$ here is endowed with any word metric, and note that this definition depends neither on the choice of the word metric on $\Gamma$ nor on the choice of the basepoint $x_0$ in $\mathcal{X}$. For $\gamma \in \Gamma$, denote $\Vert \gamma \Vert$ its cyclically reduced word length (for a chosen set of generators), and $l(\rho(\gamma))=\underset{x\in \mathcal{X}}{\inf}d(\rho(\gamma)x,x)$ the displacement of $\rho(\gamma)$.	Delzant-Guichard-Labourie-Mozes (\cite{delzant_displacing_2011}) proved that $\rho$ is convex-cocompact if and only if there exist two constants $C>0$ and $D\geq 0$ such that for all $\gamma \in \Gamma$, $l(\rho(\gamma)) \geq \frac{1}{C} \Vert \gamma \Vert -D$. A fundamental example arises when $\mathcal{X}=\HH^2$ and $\Gamma$ is the fundamental group $\pi_1(S)$ of a closed oriented surface $S$ with genus at least 2. In this setting, the set of convex-cocompact representations coincide with two copies of the Teichmüller space of $S$ in the character varitety $\chi(\pi_1(S), \mathrm{PSL}(2,\R))$. Another rich example appears when $\mathcal{X}=\HH^3$ : the set of convex-cocompact representations agrees with the set of quasi-fuchsian representations from $\pi_1(S)$ to $\mathrm{PSL}(2,\C)$. The set of (conjugacy classes of) convex-cocompact representations is known to be an open domain of discontinuity for the action of $\mathrm{Out}(\pi_1(S))$ on the character variety $\chi(\pi_1(S), \mathrm{Isom}(\mathcal{X}))$. Whether or not it is the largest remains in general open. In the two examples above, Goldman conjectured that the action is ergodic on the complement of the set of convex-cocompact representations~\cite{farb_mapping_2006}.\\
	
	The situation is different when $\Gamma=\F_2$ and $\mathcal{X}=\HH^3$. Indeed, Bowditch \cite{bowditch_markoff_1998}, followed by Tan-Wong-Zhang \cite{tan_generalized_2008}, found a subspace in the character variety $\chi(\F_2,\mathrm{SL}(2,\C))$  \emph{strictly} larger than the set of convex-cocompact representations and which form an open domain of discontinuity. This subspace is defined using some conditions on the images of the primitive element in $\F_2$ by the representation. We say that $u \in \F_2$ is \emph{primitive} if there exists a free basis of $\F_2$ containing $u$. The set of primitive elements in $\F_2$ is invariant under conjugation and inversion and we denote $\mathcal{P}(\F_2)$ the set of conjugacy and inversion classes of primitive elements. We denote $[\gamma]$ the class of the primitive element $\gamma \in \F_2$ in $\mathcal{P}(\F_2)$. \\
	
	\begin{Definition}[Bowditch \cite{bowditch_markoff_1998}, Tan-Wong-Zhang \cite{tan_generalized_2008}] \label{def:BQ SL(2,C)}~ \\
		Let $\rho : \F_2\to \mathrm{SL}(2,\C)$ be a representation. We say that $\rho$ satisfies the \emph{$BQ$-conditions} if the two following conditions hold : 
		\begin{itemize}
			\item For all $[\gamma] \in \mathcal{P}(\F_2)$, $\mathrm{Tr}(\rho(\gamma)) \in \C \backslash [-2,2]$.
			\item The set $\{ [\gamma] \in \mathcal{P}(\F_2) \; | \; |\mathrm{Tr}(\rho(\gamma)) | \leq 2\}$ is finite.
		\end{itemize}
	Note that since these two conditions are invariant under  inversion of the trace, the $BQ$-conditions are well-defined for representations with values in $\mathrm{PSL}(2,\C)$. 
	\end{Definition} 
	
	The work of Tan-Wong-Zhang \cite{tan_generalized_2008}, generalizing that of Bowditch \cite{bowditch_markoff_1998}, gives a characterization of the $BQ$-conditions above in term of the growth of the logarithm of the traces of the images of primitive elements with respect to the (cyclically reduced) word length. Precisely, they proved that a representation $\rho$ satisfies the $BQ$-conditions if and only if there exists a constant $C>0$ such that for all $[\gamma] \in \mathcal{P}(\F_2)$, $\log^+ |\mathrm{Tr}(\rho(\gamma))|:=\max\left\{\log |\mathrm{Tr}(\rho(\gamma))| , 0\right\}	\geq \frac{1}{C}\Vert \gamma \Vert$. They proved that the set of representations satisfying the $BQ$-conditions is open, and that the action of $\mathrm{Out}(\F_2)$ on it is properly discontinuous. Their work makes essential use of the trace formulas which hold in $\mathrm{SL}(2,\mathbb{C})$. \\
	
	\subsection{Characterizations of Bowditch representations in Gromov-hyperbolic spaces}~\\ 
	The aim of this paper is to define such a set of representations in the more general setting of Gromov-hyperbolic spaces and to investigate its properties. We will assume that the space is \emph{geodesic}, which means that there exists a geodesic segment between any two distinct points in $\mathcal{X}$ and that it is a \emph{visibility space}, which means that there exists a geodesic between any two distinct points in the boundary of $\mathcal{X}$. Since traces of isometries do not exist in this broader context, we will replace them by the \emph{stable-norm} $l_S(A)$ of an isometry $A$ of $\mathcal{X}$, defined as follows~: $ \displaystyle l_S(A) = \underset{n \to \N}{\lim} \frac{1}{n} d(A^nx_0,x_0)$, for $x_0 \in \mathcal{X}$ being any basepoint in $\mathcal{X}$ (see section \ref{subsec:hyperbolic isometries} for more informations on the stable-norm). Note that in $\HH^n$, the stable-norm and the usual displacement $l(A) := \underset{x \in \HH^n}{\inf} d(Ax,x)$ coincide. 
	 Let us also recall that in $\mathrm{PSL}(2,\R)$, which is the isometry group of the hyperbolic plane $\HH^2$, the stable-norm and the traces are linked by the following formula : $ \mathrm{Tr}(A)=2\cosh ( \frac{l_S(A)}{2} ) $, and that in $\mathrm{PSL}(2,\C)$, we have the following inequality~: $ \displaystyle e^{\frac{l_S(A)}{2}}-1 \leq | \mathrm{Tr}(A)| \leq e^{\frac{l_S(A)}{2}}+1$. We will work in our setting with the stable-norm which will be more convenient that the displacement. We prove that the following conditions are equivalent~:
	 
	 	\begin{Theorem} \label{thm:equivalence} Let $\mathcal{X}$ be a geodesic, $\delta$-hyperbolic, visibility space. Let $K_\delta=329\delta$. \\
	 	Let $\rho : \F_2\to \mathrm{Isom}(\mathcal{X})$ be a representation. 
	 	The following conditions are equivalent : \vspace{0.1cm}
	 	\begin{enumerate}[label=(BQ\alph*)]
	 		\item \label{cond:thm:Cdelta} \begin{itemize}
	 			\item For all $[\gamma] \in \mathcal{P}(\F_2)$, $\rho(\gamma)$ is hyperbolic.
	 			\item The set $\{[\gamma] \in \mathcal{P}(\F_2) \, \vert \, l_S(\rho(\gamma))\leq K_\delta)\}$ is finite. 
	 		\end{itemize}
	 		\vspace{0.2cm}
	 		\item \label{cond:thm:K>Cdelta} 
	 		\begin{itemize}
	 			\item For all $[\gamma] \in \mathcal{P}(\F_2)$, $\rho(\gamma)$ is hyperbolic. 
	 			\item For all $K \geq 0$, the set $\{[\gamma] \in \mathcal{P}(\F_2) \, \vert \, l_S(\rho(\gamma))\leq K)\}$ is finite. 
	 		\end{itemize}
	 		\vspace{0.2cm}
	 		\item \label{cond:thm:tree} 
	 		\begin{itemize}
	 			\item For all $K \geq K_\delta$, the tree $T_\rho(K)$ (defined in section \ref{subsec:treeJ}) is finite.
	 		\end{itemize} \vspace{0.2cm}
	 		\item \label{cond:thm:displacing} 
	 		\begin{itemize}
	 			\item There exists $C>0$ and $D\geq 0$ such that for all $[\gamma] \in \mathcal{P}(\F_2)$, $l_S(\rho(\gamma)) \geq \frac{1}{C}\Vert \gamma \Vert -D$.
	 		\end{itemize}
	 	\end{enumerate}
	 When one (hence all) of these conditions is satisfied, we say that $\rho$ is a \emph{Bowditch representation}. 
	 \end{Theorem}
	
	\begin{Remark} \label{rem:thm}
		\begin{enumerate}[label=\arabic*.]
			\item In particular, this theorem covers the case where $\mathcal{X}$ is a rank one symmetric space~: the real hyperbolic $n$-space (with isometry group $\mathrm{Isom}(\mathcal{X})= \mathrm{SO}_0(n, 1)$), the complex hyperbolic $n$-space ($\mathrm{Isom}(\mathcal{X})=\mathrm{SU}(n, 1)$), the quaternionic hyperbolic $n$-space ($\mathrm{Isom}(\mathcal{X})=\mathrm{Sp}(n, 1)$) and the hyperbolic plane over the Cayley numbers ($\mathrm{Isom}(\mathcal{X})=\mathrm{F}_{4(-20)} $). See \cite{knapp_lie_1996} for the classification of real Lie groups. This theorem also allows non proper examples, such as the infinite dimentional hyperbolic space $\HH^\infty$. 
			\item We stated the theorem using the stable-norm $l_S$, but since the stable-norm and the displacement are equal up to an additive constant $16\delta$ (see section \ref{subsec:hyperbolic isometries} or Proposition 6.4 in \cite{coornaert_geometrie_1990}), Theorem \ref{thm:equivalence} remains true when replacing the stable-norm $l_S(\rho(\gamma))$ by the displacement $l(\rho(\gamma))$ and $K_\delta$ by $K_\delta+16\delta$.  
			\item The tree $T_\rho(K)$ will be defined in section \ref{subsec:treeJ} using previous results of the paper. It is a subtree of the Farey tree, which may be thought of as the dual of the Farey tessellation of the hyperbolic plane by ideal triangles (see section \ref{subsec:notations}).  
			\item 	\label{rem:constants} In condition \ref{cond:thm:displacing}, the additive constant $D$ plays no role. We will actually prove that condition \ref{cond:thm:Cdelta} implies condition \ref{cond:thm:displacing} with $D=0$. Moreover, we will also explain in the proof of Theorem \ref{thm:equivalence a,b,d} that assuming that condition \ref{cond:thm:displacing} holds with constants $C$ and $D$ implies that the same inequality holds with $D=0$ without altering the multiplicative constant $C$.  
			\item Note that implication \ref{cond:thm:K>Cdelta} $\implies$ \ref{cond:thm:Cdelta} is trivial, and so is \ref{cond:thm:displacing} $\implies$ \ref{cond:thm:K>Cdelta} once we assume that $D=0$ (see Remark \ref{rem:thm}. \ref{rem:constants}). 
			\item \label{rem:thm:hyp} The constant $K_\delta=329\delta$ here works for all geodesic, $\delta$-hyperbolic, visibility space. It could be improved (that is, reduced) by assuming further conditions on the space $\mathcal{X}$. For example, if the space $\mathcal{X}$ satisfies furthermore the two following conditions : 
			\begin{enumerate}[label=(H\alph*)]
				\item \label{hyp:unique-visibility} there exists a \emph{unique} geodesic between two distinct points on the boundary of  $\mathcal{X}$,
				\item \label{hyp:strong-asymptot} any two asymptotic geodesic rays $\gamma_1, \gamma_2 : [0,+\infty) \to \mathcal{X}$ are \textit{strongly asymptotic}, that is there exists $t_0 \in \R$ such that $d(\gamma_1(t),\gamma_2(t+t_0)) \underset{t \to +\infty}{\longrightarrow} 0$,
			\end{enumerate} then Theorem \ref{thm:equivalence} remains true when replacing $K_\delta$ by $K_\delta=217\delta$ (see Remark~\ref{rem:irreducible-CAT(-1)} and Remark~\ref{rem:ineg-CAT(-1)}).
			Note that every CAT(-1) space satisfies \ref{hyp:unique-visibility} and \ref{hyp:strong-asymptot}.
		\end{enumerate}
	\end{Remark}
	
	Lawton, Maloni and Palesi studied in \cite{lawton_dynamics_2025} the relative $\mathrm{SU}(2,1)$-character varieties of the one-holed torus, that is the level sets $\left\{[\rho ] \in \chi(\F_2,\mathrm{SU}(2,1))  \; | \; \mathrm{Tr}(\rho([a,b]))=c \right\}$ for $c\in \C$, of the character variety $\chi(\F_2,\mathrm{SU}(2,1))$, with $\F_2=\langle a,b \rangle $.
	Using an explicit parametrization of the character variety $\chi(\F_2,\mathrm{SU}(2,1))$ in term of traces, they proved similar characterizations of Bowditch representations in $\mathrm{SU}(2,1)$ (Theorem B in \cite{lawton_dynamics_2025}). Their conditions on the representations are defined using the traces of the isometries as follows : For all $[\gamma] \in \mathcal{P}(\F_2)$, $\rho(\gamma)$ is hyperbolic and the set $\{[\gamma] \in \mathcal{P}(\F_2) \; \vert \; |\mathrm{Tr}(\rho(\gamma))| \leq M(c)\}$ is finite, where $M(c)$ is a constant depending on $c=\mathrm{Tr}(\rho([a,b]))$. Note that the main difference between their characterization of Bowditch representations and ours is that in our work the constant $K_\delta$ used in characterization \ref{cond:thm:Cdelta} is independent of $c$. Hence our work applied to $\mathrm{SU}(2,1)$ gives a new characterization of Bowditch representations with a constant independent of the value of the image of the commutator. 
	
    In another direction, Maloni, Palesi and Tan studied in \cite{maloni_character_2015} the relative $\mathrm{SL}(2,\C)$-character variety of the four-punctured sphere group and Maloni and Palesi in \cite{maloni_character_2020} the $\mathrm{SL}(2,\C)$-character variety of the three-holed projective plane group. In both cases, they defined a set of Bowditch representations using similar characterizations to the one given in Theorem \ref{thm:equivalence}. In their works, the value of the constant used to define the condition analogous to condition \ref{cond:thm:Cdelta} of Theorem \ref{thm:equivalence} involves the traces of the images of the boundary components. 
	~\\
	
	When $\Gamma=\F_n$ is the free group on $n$ generators and $G=\mathrm{PSL}(2,\C)$, Minsky introduced in \cite{minsky_dynamics_2013} the set of \emph{primitive-stable} representations in $\mathrm{Hom}(\F_n,\mathrm{PSL}(2,\C))$. Although Minksy defined primitive-stable representations in $\mathrm{PSL}(2,\C)=\mathrm{Isom}^+(\HH^3)$, the definition makes sense when replacing $\HH^3$ by any Gromov-hyperbolic space $\mathrm{Isom}(\mathcal{X})$ (or more generally any metric space) so we give the definition in this setting. 
	Let $S$ be a free generating set for $\F_n$ and $\mathcal{C}(\F_n)$ the Cayley graph of $\F_n$ with respect to the generating set $S\cup S^{-1}$. Let $x_0 \in \mathcal{X}$. The \emph{orbit map} $\tau_\rho$ of a representation $\rho : \F_n \to \mathrm{Isom}(\mathcal{X})$ is the unique map $\tau_\rho : \mathcal{C}(\F_n) \to \mathcal{X}$ such that $\tau_\rho(id)=x_0$, $\tau_\rho(\gamma)=\rho(\gamma)x_0$ for all $\gamma \in \F_n$ and such that $\tau_\rho$ sends the edges of the graph to geodesic segments in $\mathcal{X}$. Every (primitive) element $\gamma \in \F_n$ acts on the Cayley graph $\mathcal{C}(\F_n)$  and we denote by $L_\gamma$ its \emph{axis} in $\mathcal{C}(\F_n)$, which is the unique $\gamma$-invariant geodesic in $\mathcal{C}(\F_n)$. A representation $\rho : \F_n \to \mathrm{Isom}(\mathcal{X})$ is \emph{primitive-stable} if there exist two constants $C>0$ and $D\geq0$ such that for all $[\gamma] \in \mathcal{P}(\F_n)$, the image of $L_\gamma$ by the orbit map $\tau_\rho$ is a $(C,D)$-quasi-geodesic in~$\HH^3$. 
	Minsky proved in $\mathrm{PSL}(2,\C)$ that the subspace in the character variety $\chi(\F_n,\mathrm{PSL}(2,\C))$ corresponding to primitive-stable representations is an open domain of discontinuity strictly containing the set of convex-cocompact representations. When $n=2$, Lee-Xu \cite{lee_bowditchs_2019} on the one hand and Series \cite{series_primitive_2019,series_primitive_2020} on the other hand proved in $\mathrm{PSL}(2,\C)$ that the $BQ$-conditions (see Definition \ref{def:BQ SL(2,C)}) are equivalent to primitive-stability. In \cite{schlich_equivalence_2024}, the author generalized this equivalence in any Gromov-hyperbolic, geodesic, visibility space using the \ref{cond:thm:displacing} condition, a priori stronger than the \ref{cond:thm:Cdelta} condition. In view of this previous work and Theorem \ref{thm:equivalence} in this article, we obtain the following Corollary : 
	
	\begin{Corollary}\label{thm:primitive-stable}
		Let $\mathcal{X}$ be a geodesic, $\delta$-hyperbolic, visibility space.
		Let $\rho : \F_2\to \mathrm{Isom}(\mathcal{X})$ be a representation. Then $\rho$ satisfies \ref{cond:thm:Cdelta} if and only if $\rho$ is primitive-stable. 
	\end{Corollary}
	When $\mathcal{X}$ is the real hyperbolic $n$-space $\HH^n$,  Fléchelles \cite{flechelles_primitive_2025} also generalized the $BQ$-conditions of Definition \ref{def:BQ SL(2,C)} for representations of $\F_2$ into $\mathrm{Isom}(\HH^n)$. For $\lambda >0$, he defined the $Q_\lambda$-conditions on a representation $\rho : \F_2 \to \mathrm{Isom}(\HH^n)$ as follows : for every $[\gamma] \in \mathcal{P}(\F_2)$, $\rho(\gamma)$ is hyperbolic and the set of (conjugacy classes of) primitive elements $\gamma$ such that $l(\rho(\gamma)) \leq \lambda$ is finite. Hence, since the displacement $l(\rho(\gamma))$ and the stable-norm $l_S(\rho(\gamma))$ agrees in $\HH^n$, the $Q_{K_\delta}$-conditions, for $K_\delta$ the constant defined in Theorem \ref{thm:equivalence}, agree with our conditions \ref{def:BQa}. Fléchelles proved that for all \emph{Coxeter extensible} representations $\rho : \F_2 \to \mathrm{Isom}(\HH^n)$, there exists a constant $\lambda>0$, such that $\rho$ satisfies the $Q_\lambda$-conditions if and only if $\rho$ is primitive-stable (Theorem 1.6 in \cite{flechelles_primitive_2025}). Note that in fact, Fléchelles proved this result for an a priori larger class of representations from $\F_2$ to $\mathrm{Isom}(\HH^n)$, those satisfying a property he called the \emph{half-length} property (see section 3 in \cite{flechelles_primitive_2025}). Let us also observe that in this theorem, the constant $\lambda$ depends on the representation $\rho$. More precisely, Fléchelles shows that this constant can be taken as a strictly decreasing function of the \emph{primitive systole} of $\rho$, which is the infimum of the displacement $l(\rho(\gamma))$ among all primitive element $\gamma$ in $\F_2$. 
	
	Note that when $\F_2$ is replaced by $\pi_1(S)$, the fundamental group of a surface $S$, one can define the notion of simple-stable representations in analogy with Minsky's notion of primitive-stability by replacing primitive elements in the definition by simple closed curves. When $S$ is a four-punctured sphere, the author proved in \cite{schlich_simple-stable_2025} the equivalence between simple-stable representations and Bowditch representations in this setting using a definition analogous to \ref{cond:thm:displacing}. \\

\subsection{Dynamics of the outer automorphism group}~\\ 
Let us recall that the \emph{outer automorphism group} $\mathrm{Out}(\Gamma)$ of a group $\Gamma$ is the quotient of the automorphism group $\mathrm{Aut}(\Gamma)$ of $\Gamma$ by the set of inner automorphisms of $\Gamma$. There is a natural action of $\mathrm{Aut}(\Gamma)$ on the space of representations $\mathrm{Hom}(\Gamma, \mathrm{Isom}(\mathcal{X}))$ by precomposition : for $\rho \in \mathrm{Hom}(\Gamma,\mathrm{Isom}(\mathcal{X}))$ and $\Phi \in \mathrm{Aut}(\F_2)$, we set $\Phi.\rho=\rho \circ \Phi^{-1}$. This action induces an action of the outer automorphism group $\mathrm{Out}(\F_2)$ on the character variety $\chi(\F_2,\mathrm{Isom}(\mathcal{X}))$. We say that the action of a group $\Gamma$ on a topological space $\mathcal{T}$ is \emph{properly discontinuous} if for every compact subset $K$ of $\mathcal{T}$, 
the set $\left\{ \Phi \in \Gamma \; | \; \Phi(K) \cap K \neq \emptyset \right\}$ is finite. The techniques developed in this article allow us to study the action of the outer automorphism group $\mathrm{Out}(\F_2)$ of $\F_2$ on the set of Bowditch representations. 

\begin{Theorem} \label{thm:open and dynamics}
	Let $\mathcal{X}$ be a geodesic, $\delta$-hyperbolic, visibility space. The set of (conjugacy classes of) representations satisfying \ref{cond:thm:Cdelta} is open in the character variety $\chi(\F_2,\mathrm{Isom}(\mathcal{X}))$ and the action of the outer automorphism group $\mathrm{Out}(\F_2)$ is properly discontinuous on it. 
\end{Theorem}
\begin{Remark}
	It is not hard to see that the set of primitive-stable representations is open and that the action of the outer automorphism group on it is properly discontinuous. Minsky \cite{minsky_dynamics_2013} proved it for $\mathrm{PSL}(2,\C)$-representations but no difficulty is added when passing to $\mathrm{Isom}(\mathcal{X})$ for $\mathcal{X}$ Gromov-hyperbolic (see also Proposition 4.5 in \cite{schlich_equivalence_2024} for a proof of the openness). Therefore, we could also obtain Theorem \ref{thm:open and dynamics} using Corollary \ref{thm:primitive-stable}. However, we will explain how the framework developed for proving Theorem \ref{thm:equivalence} can be used to obtain a direct proof of this result which does not rely on the equivalence with primitive-stable representations. \\
\end{Remark}

\subsection{Recognition}~\\
	Finally, we will discuss how our work enables us to obtain finite certificates for Bowditch representations. More precisely, we will write an explicite algorithm which ends in finite time if $\rho$ is a Bowditch representation. 
	
	\begin{Theorem}\label{thm:recognition}
		Let $\mathcal{X}$ be a geodesic, $\delta$-hyperbolic, visibility space. Let $\rho : \F_2 \to \mathrm{Isom}(\mathcal{X})$ be a representation. Then, there exists an algorithm, described in Algorithm \ref{alg:main} in section \ref{subsec:recognition}, such that
		\begin{itemize}
			\item If $\rho$ is a Bowditch representation, then the algorithm ends in finite time. 
			\item If the algorithm ends in finite time, then it decides whether $\rho$ is a Bowditch representation. \\
		\end{itemize}
	\end{Theorem}

\subsection{An ingredient of the proof}~\\ 
Part of this work is very much inspired by the work of Bowditch \cite{bowditch_markoff_1998} and its generalization by Tan-Wong-Zhang \cite{tan_generalized_2008}. However, their work in 
$\mathrm{PSL}(2,\C)$ crucially relies on the existence of trace relations (as in Lawton-Maloni-Palesi’s work in $\mathrm{SU}(2,1)$ \cite{lawton_dynamics_2025}), and in particular on the two relations :
\begin{align}
&	 \mathrm{Tr}(AB)+\mathrm{Tr}(AB^{-1})=\mathrm{Tr}(A)\mathrm{Tr}(B) \label{eq:edge relation}\\ 
& \mathrm{Tr}(A)^2+\mathrm{Tr}(B)^2+\mathrm{Tr}(AB)^2=\mathrm{Tr}(A)\mathrm{Tr}(B)\mathrm{Tr}(AB)+\mathrm{Tr}([A,B])+2.
\end{align}
These two equalities, which can be interpreted as an \emph{edge relation} and a \emph{vertex relation} in their framework, are extensively used in their papers. In order to prove Theorem \ref{thm:equivalence} without such relations, we will need to replace them by new large-scale geometry arguments.
In particular, we will prove the following inequality on the length of a product of two isometries. 
\begin{Proposition}[Proposition \ref{prop:ineq-l(AB)} in section \ref{sec:l(AB)}] \label{prop:l(AB) intro} Let $C_\delta = 256\delta$. \\Let $A$ and $B$ be two hyperbolic isometries such that $\{A^+,A^-\}\cap \{B^+,B^-\}$. Assume that $l_S(A) >C_\delta$ and $l_S(B)>C_\delta$. Then 
	\begin{equation}\label{eq:l(AB)intro}
		\max \{ l_S(AB),l_S(AB^{-1}) \} \geq l_S(A)+l_S(B)-C_\delta. 
	\end{equation}
\end{Proposition}
Note that in $\mathrm{SL}(2,\C)$, Proposition \ref{prop:l(AB) intro} follows from equation \eqref{eq:edge relation}. Indeed, from \eqref{eq:edge relation} :
\begin{align*}
	|\mathrm{Tr}(A)| \; |\mathrm{Tr}(B)| = | \mathrm{Tr}(AB)+\mathrm{Tr}(AB^{-1})| & \leq |\mathrm{Tr}(AB)|+|\mathrm{Tr}(AB^{-1})| \leq 2 \max\{|\mathrm{Tr}(AB)|,|\mathrm{Tr}(AB^{-1})| \}, \\ 
\intertext{and now taking the $\log$ gives :}
\max\{ \log | \mathrm{Tr}(AB) | , \log | \mathrm{Tr}(AB^{-1})| \} & \geq \log |\mathrm{Tr}(A)| + \log | \mathrm{Tr}(B) | - \log 2. 	
\end{align*}
Now recall that $\log | \mathrm{Tr}(A)|$ and $l_S(A)$ are comparable in $\mathrm{SL}(2,\C)$ (up to an additive constant) to obtain the claim. Also note that if one is interested in the study of the $\log$ of the modulus of the trace rather than on the stable-norm $l_S$ (or displacement), then Proposition \ref{prop:l(AB) intro} holds with constant $\log 2$ and for all hyperbolic isometries, without having to assume that the modulus of the trace is large enough.
~\\ 

\subsection{Organization of the paper}~\\ 
In section \ref{sec:preliminaries} we gather and establish the necessary preliminaries on Gromov-hyperbolic spaces. In section \ref{sec:irreducibiility}, we prove the irreducibility of Bowditch representations using Busemann functions. In section \ref{sec:l(AB)}, we establish Proposition \ref{prop:ineq-l(AB)} which proves inequality \eqref{eq:l(AB)intro}. 
In section \ref{sec:growth Farey neighbours}, we study the growth of the length of the images of the Farey neighbours of a primitive element, that is we study the growth of the quantity $\l_S(\rho(a^nb))$, for $\{a,b\}$ a basis of $\F_2$ and $n \in \N$. In section \ref{sec:Farey tree}, we introduce the Farey tree and its associated objects, then use it to prove the equivalence \ref{cond:thm:Cdelta} $\iff$ \ref{cond:thm:K>Cdelta} $\iff$ \ref{cond:thm:displacing} of Theorem~\ref{thm:equivalence}. In section \ref{subsec:treeJ}, we introduce the tree $T_\rho(K)$ associated to a representation $\rho : \F_2 \to \mathrm{Isom}(\mathcal{X})$ and we prove equivalence \ref{cond:thm:Cdelta} $\iff $ \ref{cond:thm:tree} of Theorem \ref{thm:equivalence}. In section \ref{subsec:openness and dynamics}, we prove Theorem \ref{thm:open and dynamics} and in section \ref{subsec:recognition}, we prove Theorem \ref{thm:recognition}. 
	
\subsection*{Acknowledgements.} The author is very grateful to François Guéritaud for helpful discussions and comments. The author also warmly thanks Balthazar Fléchelles for many interesting conversations on the subject and Andrea Seppi for useful comments on the paper.
The author is funded by the European Union (ERC, GENERATE, 101124349). Views and opinions expressed are however those of the author only and do not necessarily reflect those of the European Union or the European Research Council Executive Agency. Neither the European Union nor the granting authority can be held responsible for them.
	
	\section{Preliminaries on hyperbolic spaces} \label{sec:preliminaries}
	In this section, we gather the required material on $\delta$-hyperbolic spaces needed for our work. The content of this section is classical. For a more detailed introduction or further results on $\delta$-hyperbolic spaces, we refer the interested reader, among other references, to \cite{coornaert_geometrie_1990}, \cite{bridson_metric_1999}, \cite{drutu_geometric_2018}. \\
	
	In this article, we will work in the context of geodesic $\delta$-hyperbolic spaces.  Let $(\mathcal{X},d)$ be a metric space. The metric space $\mathcal{X}$ is said to be \emph{geodesic} if there exists a geodesic segment between any two points in $\mathcal{X}$. We will write $[x,y]$ for a geodesic segment joining the points $x$ and $y$ in $\mathcal{X}$. Note that such a segment is not necessary unique. For $x$, $y$ and $z$ three points in $\mathcal{X}$, a \emph{triangle} $T(x,y,z)$ with vertices $x$, $y$ and $z$ is the union of three geodesic segments $[x,y]$, $[y,z]$ and $[z,x]$ which are called the sides of the triangle.	Let $\delta \geq 0$ be a constant. We say that a triangle is \emph{$\delta$-thin} if each side of the triangle is contained in the $\delta$-neighborhood of the union of the other two. A geodesic metric space $\mathcal{X}$ is said to be \emph{$\delta$-hyperbolic} if every triangle in $\mathcal{X}$ is $\delta$-thin. This definition of hyperbolicity is sometimes called \emph{hyperbolicity in the sense of Rips}. There are (many) other equivalent definitions of hyperbolicity, among them the notion of \emph{hyperbolicity in the sense of Gromov}, which we will define and compare to the one in the sense of Rips in Lemma \ref{lem:rips-to-gromov}. 
	
	In all that follows, we fix $(\mathcal{X},d)$ a geodesic $\delta$-hyperbolic space. \\
	
	\subsection{Shapes of triangles and quadrilaterals}~ \\
	\indent We can be more precise about the geometry of $\delta$-thin triangles. It is useful to compare the triangle  $T(x,y,z)$ with its associated tripod, which is the tripod with the same lengths of sides as $T(x,y,z)$. Lemma \ref{lem:points on thin triangles-centroid} finds three points, one on each side of the triangle, which correspond to the center of the associated tripod and which remain at bounded distance from each other.
	
		\begin{Lemma} \label{lem:points on thin triangles-centroid}
		Let $T(x,y,z)$ be a $\delta$-thin triangle in $\mathcal{X}$. Then there exist three points $p$, $q$ and $r$ respectively on $[x,y]$, $[y,z]$ and $[z,x]$ such that $d(x,p)=d(x,r)$, $d(y,p)=d(y,q)$, $d(z,q)=d(z,r)$ and $d(r,p) \leq 4\delta$, $d(p,q)\leq 4\delta$, $d(q,r)\leq 4\delta$.
	\end{Lemma}
	
	\begin{proof}
		See proof of implication (1)$\implies$(3) in Proposition 1.17 in \cite{bridson_metric_1999}. The statement of the proposition just gives the existence of a bound on $d(r,p)$, $d(p,q)$ and $d(q,r)$ but the proof computes the constant $4\delta$ explicitly. 
	\end{proof}
	
	If we don't impose the condition on the equality of the lengths in the conclusion of the previous lemma, we can obtain three points even closer.
		\begin{Lemma} \label{lem:points on thin triangles}
		Let $T(x,y,z)$ be a $\delta$-thin triangle in $\mathcal{X}$. Then there exist three points $p$, $q$ and $r$ respectively on $[x,z]$, $[y,z]$ and $[x,y]$ such that $d(p,r) \leq \delta$, $d(q,r)\leq \delta$ (hence in addition $d(p,q)\leq 2\delta$).
	\end{Lemma}
	
	\begin{proof}
	For a detailed proof, see Lemma A.1 in \cite{schlich_equivalence_2024}.
	\end{proof}
	
	For $x$, $y$, $z$ and $w$ four points in $\mathcal{X}$, a \emph{quadrilateral} $Q(x,y,z,w)$ is the union of four geodesic segments $[x,y]$, $[y,z]$, $[z,w]$ and $[w,x]$, which are called the \emph{sides} of the quadrilateral. The geometry of quadrilateral is also controlled. Indeed, it is a classical fact that in $\delta$-hyperbolic spaces, quadrilaterals are $2\delta$-thin.
		\begin{Lemma} [Quadrilaterals are $2\delta$-thin.]\label{lem:quadri-thin}
		Let $x,y,z,w $ be four points in $\mathcal{X}$ and $Q(x,y,z,w)$ be a quadrilateral. Then the quadrilateral $Q(x,y,z,w)$ is $2\delta$-thin, which means that $[w,x]$ is contained in the $2\delta$-neighborhood of $[x,y]\cup [y,z] \cup [z,w]$. 
	\end{Lemma}
	\begin{proof}
		See Lemma 11.6 in \cite{drutu_geometric_2018} or Lemma 1.5 in \cite{coornaert_geometrie_1990} for a more general statement for polygons.
	\end{proof}
	
	From the $2\delta$-thinness of quadrilaterals, we easily deduce, in the configuration where one side of the quadrilateral remains far from its opposite side, that we can find a point on it that is close to a point on each of its two adjacent sides.
	\begin{Lemma}\label{lem:quadrilateral-shape}
		Let $x,y,z,w $ be four points in $\mathcal{X}$ and $[x,y]$, $[y,z]$, $[z,w]$ and $[w,x]$ be four geodesic segments between these points. Assume moreover that for all $t \in [w,x]$, $d(t,[y,z])>2\delta$. Then there exist $t \in [w,x]$, $t_1 \in [x,y]$ and $t_2 \in [z,w]$ such that $d(t,t_1) \leq 2\delta$ and $d(t,t_2) \leq 2\delta$. 
	\end{Lemma}
	
	\begin{proof}
		It is a consequence of Lemma \ref{lem:quadri-thin}. For a detailed proof, see Lemma A.3 in \cite{schlich_equivalence_2024}. 
	\end{proof}~\\ 
	
	\subsection{Gromov-boundary and shape of ideal-bigons and asymptotic geodesics}~\\ 
	\indent Let $x$, $y$ and $z$ be three points in $\mathcal{X}$. The \emph{Gromov-product} of $x$ and $y$ with respect to $z$ is defined by :
	\begin{equation}\label{eq:def-gromov-product}
		(x\vert y)_{z} = \frac{1}{2}(d(z,x)+d(z,y)-d(x,y)).
	\end{equation}
	We say that a sequence of points $(x_n)_{n\in \N}$ \emph{converges at infinity} if $(x_n|x_m)_z \underset{n,m \to \infty}{\longrightarrow}+\infty$. This notion does not depend on the choice of $z$ and notice that this implies that $d(x_n,z) \underset{n \to \infty}{\longrightarrow}+\infty$. When two sequences $(x_n)_{n\in \N}$ and $(y_n)_{n \in \N}$ converge at infinity, we say that they are \emph{equivalent} when $(x_n|y_n)\underset{n \to \infty}{\longrightarrow}+\infty$. The equivalence relation between sequences of points gives rise to the notion of the Gromov-boundary $\partial \mathcal{X}$ of $\mathcal{X}$ by considering the equivalence classes of this relation. When $(x_n)_{n\in \N}$ is a sequence in $\mathcal{X}$ and $x \in \partial \mathcal{X}$, we denote $x_n \underset{n \to \infty}{\longrightarrow} x$ to say that $(x_n)_{n\in \N}$ is in the equivalent class $x \in \partial \mathcal{X}$. 
	
	When $l : [0,+\infty) \to \mathcal{X}$ is a geodesic ray in $\mathcal{X}$ (we will denote again, by a slight abuse of notation, its image by $l$) and $(x_n)_{n\in \N}$ is an unbounded sequence on $l$, then $(x_n)_{n \in \N}$ automatically converges at infinity. Moreover, any two unbounded sequences on $l$ are equivalent, giving rise to a well-defined point at infinity in $\partial \mathcal{X}$, which we call the \emph{endpoint at infinity} of $l$. If $x \in \mathcal{X}$ and $l_\infty \in \partial \mathcal{X}$, we say that $l$ is a geodesic ray or a\emph{ half geodesic} with endpoints $x$ and $l_\infty$ or joining $x$ and $l_\infty$ if $l_\infty$ is the endpoint of $l$ and $l(0)=x$, and we denote $[x,l_\infty)$. When two geodesic rays have the same endpoint, we say that they are \emph{asymptotic}. Note that two asymptotic geodesic rays remain at bounded distance from each other. If $l : \R\to \mathcal{X}$ is a bi-infinite geodesic, it has two well-defined distinct points in the boundary $\partial \mathcal{X}$ of $\mathcal{X}$, which we also call the \emph{endpoints} of $l$. 
	
	Let $l$ be either a geodesic, a geodesic ray, or a geodesic segment, and $x$ be a point  in $\mathcal{X}$. We say that $x'$ is a \emph{projection} of $x$ on $l$ if $d(x,x')=d(x,l)$. The projection is not necessarily unique. We will say more on projections in $\delta$-hyperbolic spaces in section \ref{sec:projection}. \\

	The following lemma states that geodesics with the same endpoints are uniformly close. It is an exercise in \cite{drutu_geometric_2018}. For completeness, we include a detailed proof here.
	\begin{Lemma}[Ideal bigons are $2\delta$-thin, \cite{drutu_geometric_2018} Exercise 11.85] \label{lem:bigon_thin}
		Let $l$ and $l'$ be two geodesic of $\mathcal{X}$ which share the same endpoints. Then, for all $x \in l$,
		\begin{equation*}
			d(x,l') \leq 2 \delta.
		\end{equation*}
	\end{Lemma}

	\begin{proof}
		Let $x \in l$. We consider, for all $n \in \N$, the points $x_n$ and $y_n$ on $l$ such that, denoting $[x_n,y_n] \subset l$ the geodesic segment from $x_n$ to $y_n$ included in $l$, $x \in [x_n,y_n]$ and $d(x_n,x)=d(x,y_n)=n$. Now consider $x'_n$ and $y'_n$ some projections of $x_n$ and $y_n$ on $l'$. Consider $[x'_n,y'_n] \subset l'$ the geodesic segment from $x'_n$ to $y'_n$ included in $l'$, $[x_n,x'_n]$ a geodesic segment between $x_n$ and $x'_n$ and $[y_n,y'_n]$ a geodesic segment between $y_n$ and $y'_n$. The quadrilateral made up with these four geodesic segments is $2\delta$-thin by Lemma \ref{lem:quadri-thin}, so we deduce that $d(x,[x_n,x'_n] \cup [x'_n,y'_n] \cup [y_n,y'_n])\leq 2\delta$. But recall that $l$ and $l'$ have the same endpoints and that $x'_n$ is a projection of $x_n$ on $l'$, so $d(x_n,x'_n)$ is bounded. Therefore, using that $d(x,x_n) \to \infty$, we obtain that $d(x,[x_n,x'_n]) \to \infty$. By the same reasoning, we also deduce that $d(x,[y_n,y'_n]) \to \infty$. Then, for $n$ sufficiently large, we have $d(x,[x_n,x'_n]\cup [y_n,y'_n]) > 2\delta$. Then it follows that $d(x,[x'_n,y'_n]) \leq 2\delta$ and since $[x'_n,y'_n] \subset l'$, we deduce $d(x,l') \leq 2\delta$, which proves the lemma. 
	\end{proof}

	Using similar techniques, we can also prove that geodesics sharing one endpoint are uniformly close, sufficiently close to the common endpoint.
	
	\begin{Lemma} \label{lem:asymptot_geod}
		Let $l$, $l'$ be two geodesics of $\mathcal{X}$ that share one endpoint $l_\infty \in \partial \mathcal{X}$. Then there exists $x \in l$  such that, denoting $[x,l_\infty)$ the half geodesic contained in $l$ with endpoints $x \in \mathcal{X}$ and $l_\infty \in \partial \mathcal{X}$, we have for all $y \in [x,l_\infty)$ :
		\begin{equation*}
			d(y,l') \leq 2\delta.
		\end{equation*}
	\end{Lemma}
	
	\begin{proof}
	Let $x_0 \in l$. Denote $[x_0,l_\infty)$ the half-geodesic contained in $l$ with endpoints $x_0 \in \mathcal{X}$ and $l_\infty \in \partial \mathcal{X}$. For all $t>0$, let $y_t \in [x_0,l_\infty)$ be the point on the half-geodesic $[x_0,l_\infty)$ such that $d(x_0,y_t)=2t$, and let $x_t$ be the midpoint of $[x_0,y_t]$, where $[x_0,y_t]$ denotes the geodesic segment contained in $l$ with endpoints $x_0$ and $y_t$. Hence $d(x_0,x_t)=d(x_t,y_t)=t$. Let $x'_0$ and $y'_t$ be, respectively, some projections on the geodesic $l'$ of $x_0$ and $y_t$. Now consider $[x_0,x'_0]$ and $[y_t,y'_t]$  be, respectively, some geodesic segments between $x_0$ and $x'_0$ and $y_t$ and $y'_t$. By Lemma \ref{lem:quadri-thin}, the quadrilateral formed by these four geodesic segments is $2\delta$-thin, so $d(x_t,[x_0,x'_0]\cup [x'_0,y'_t]\cup [y_t,y'_t])\leq 2\delta$. Now note that since $d(x_t,x_0)=t \underset{t \to \infty}{\longrightarrow} \infty$, we also have $d(x_t,[x_0,x'_0]) \underset{t \to \infty}{\longrightarrow} \infty$. On the other hand, since $y_t \underset{t \to \infty}{\longrightarrow} l_\infty$ and $l_\infty$ is an endpoint of both $l$ and $l'$, we deduce that $d(y_t,y'_t)$ is bounded. Therefore, using $d(x_t,y_t)=t\underset{t \to \infty}{\longrightarrow}\infty$, we also have $d(x_t,[y_t,y'_t]) \underset{t \to \infty}{\longrightarrow}\infty$. Thus, we necessarily have $d(x_t,[x'_0,y'_t]) \leq 2\delta$ when $t \geq t_0$ for some $t_0>0$, and consequently, $d(x_t,l') \leq 2\delta$. Setting $x=x_{t_0}$ gives the Lemma.
	\end{proof}
	
	Let us be a bit more precise and give some conditions to ensure that a point on a geodesic ray is sufficiently far along the ray to be uniformly close to another asymptotic geodesic ray.
	\begin{Lemma}\label{lem:pos-proj}
		Let $l$ be a geodesic in $\mathcal{X}$ and $y_0 \in l$. Denote $l_\infty$ one of the endpoints of $l$ in $\partial \mathcal{X}$ and $ [y_0,l_\infty)$ the half-geodesic contained in $l$ with endpoints $y_0 \in \mathcal{X}$ and $l_\infty \in \partial \mathcal{X}$.  Let $x_0 \in \mathcal{X}$ and consider $[x_0,l_\infty)$ a half-geodesic from $x_0$ to $l_\infty$. Let $x \in [x_0,l_\infty)$. Then  :
		\begin{itemize}
			\item If $d(x,y_0)>d(y_0,x_0) +2\delta$, then there exists $y \in [y_0,l_\infty)$ such that $d(x,y) \leq 2\delta$.
			\item Assume moreover that $d(y_0,x) >6\delta$, then any projection of $x$ on $l$ belongs to $[y_0,l_\infty)$. 
		\end{itemize}
	\end{Lemma}
	
	\begin{proof}
		Since the geodesic rays $[x_0,l_\infty)$ and $[y_0,l_\infty)$ are asymptotic, there exist two sequences $(x_n)_{n \in \N}$ and $(y_n)_{n \in \N}$ such that $x_n \in [x_0,l_\infty)$ and $y_n \in [y_0,l_\infty)$ for all $n \in \N$, $d(x_n,x) \underset{n \to \infty}{\longrightarrow} +\infty$ and $d(x_n,y_n)$ is bounded. Let $[x_n,y_n]$ be a geodesic segment joining $x_n$ and $y_n$. Then $d(x,[x_n,y_n]) \underset{n \to \infty}{\longrightarrow}+\infty$. Take $n$ such that $d(x,[x_n,y_n]) >2\delta$ and $x_n \in [x,l_\infty)$, where $ [x,l_\infty)$ is the half-geodesic contained in $[x_0,l_\infty)$ joining $x$ and $l_\infty$. Now consider $[x_0,y_0]$ a geodesic segment joining $x_0$ and $y_0$ and let $z \in [x_0,y_0]$. By the triangle inequality we have $d(x,z) \geq d(x,y_0)-d(y_0,x_0)$ and so using our asumption on $d(x,y_0)$ we deduce that $d(x,[x_0,y_0]) > 2\delta$. Consider $[x_0,x_n]$ and $[y_0,y_n]$ be respectively the geodesic segments contained respectively in $[x_0,l_\infty)$ and $l$ and joining respectively $x_0$ and $l_\infty$ and $y_0$ and $l_\infty$. The quadrilateral $[x_0,y_0]\cup[y_0,y_n]\cup[x_n,y_n] \cup [x_0,x_n]$ is 2$\delta$-thin by Lemma \ref{lem:quadri-thin} and $x \in [x_0,x_n]$, so $d(x,[x_0,y_0]\cup[y_0,y_n]\cup[x_n,y_n]) \leq 2\delta$. But we proved above that $d(x,[x_n,y_n])>2\delta$ and $d(x,[x_0,y_0])>2\delta$ so we have $d(x,[y_0,l_\infty)) \leq 2\delta$. Hence, the first claim.
		
		Now, let us assume, moreover, that $d(y_0,x) >6\delta$. The first part of the proof gives the existence of a point $y \in [y_0,l_\infty)$ such that $d(x,y) \leq 2\delta$. Pick any point $z$ on $l \, \backslash \, [y_0,l_\infty)$. Then $d(z,y) \geq d(y_0,y)$. But on the other hand, we have $d(x,z) \geq d(y,z)-d(x,y)$ by the triangle inequality so $d(x,z) \geq d(y_0,y)-d(x,y)$. Using again the triangle inequality we have $d(x,z) \geq d(y_0,x)-2d(x,y)$. Now recall that $d(y_0,x)>6\delta$ and $d(x,y)\leq 2\delta$, this implies that $d(x,z)>2\delta$. But all the projections of $x$ on $l$ are at a distance at most $2\delta$ of $x$ since $d(x,y)\leq 2\delta$, so no projections can lie on $l \, \backslash \, [y_0,l_\infty)$. 
	\end{proof}
	
	\begin{Lemma}\label{lem:asymptot-parametrization}
		Let $l$ and $l'$ be two geodesics in $\mathcal{X}$ which share one endpoint $l_\infty \in \partial \mathcal{X}$. Let $l : \R \to \mathcal{X}$ be any geodesic parametrization of $l$ such that $l(t) \underset{t \to +\infty}{\longrightarrow} l_{\infty}$. Then there exists a geodesic parametrization of $l' : \R \to \mathcal{X}$ such that for all $t$ sufficiently large, $d(l(t),l'(t)) \leq 6\delta$. 
	\end{Lemma}
	
	\begin{proof}
		By Lemma \ref{lem:asymptot_geod}, there exists $x \in l$ such that, denoting $[x,l_\infty)$ the geodesic ray contained in $l$ with endpoints $x \in \mathcal{X}$ and $l_\infty \in \partial \mathcal{X}$, we have for all $y \in [x,l_\infty)$, $d(y,l') \leq 2\delta$. Let $t_0 \in \R$ be such that $l(t_0)=x$ and let $x' \in l'$ be a projection of $x$ on $l'$. Now consider $l' : \R \to \mathcal{X}$ the geodesic parametrization of $l'$ such that $l'(t_0)=x'$ and $l'(t) \underset{t \to + \infty}{\longrightarrow} l_\infty$. Let $p : \mathcal{X} \to l'$ be a projection map on $l'$. For all $t \geq t_0$, we have, on the one hand :
		\begin{align*}
			d(x',p(l(t))) & \leq d(x',x) +d(x,l(t)) + d(l(t),p(l(t))) \text{ by the triangle inequality} \\
			& \leq 2\delta +d(l(t_0),l(t)) + 2\delta \text{ since $l(t) \in [x,l_\infty)$} \\ 
			& \leq 4\delta + d(l'(t_0),l'(t)) = 4\delta + d(x',l'(t)).
		\end{align*}
		On the other hand, we have :
		\begin{align*}
			d(x',l'(t)) - d(x',p(l(t))) & \leq d(l'(t_0),l'(t)) - d(x,l(t))+d(x,x')+d(l(t),p(l(t))) \\
			& \leq d(l'(t_0),l'(t)) -d(l(t_0),l(t))+4\delta \\
			& \leq 4\delta.
		\end{align*}

	Therefore, we have 
	\begin{equation}\label{eq:	|d(x',p(l(t)))-d(x',l'(t))}
		|d(x',p(l(t)))-d(x',l'(t)) | \leq 4\delta
	\end{equation}
	Using Lemma \ref{lem:pos-proj}, we deduce that for $t$ sufficiently large, $p(l(t))$ belongs to $[x',l_\infty)$. Therefore 
	\begin{align}
		d(p(l(t)),l'(t)) \leq | d(x',p(l(t))) - d(x',l'(t)) | \leq 4\delta \text{ by \eqref{eq:	|d(x',p(l(t)))-d(x',l'(t))}}
	\end{align}
	Finally, we can bound $d(l(t),l'(t))$ for $t$ sufficiently large : 
	\begin{equation*}
	 	d(l(t),l'(t)) \leq d(l(t),p(l(t))) + d(p(l(t)),l'(t)) \leq 6\delta. 
	\end{equation*}

	\end{proof}
	
	\subsection{Projections} \label{sec:projection} \\ 
	\indent We now give some results about projections in $\delta$-hyperbolic spaces. \\
	
	We first check the following :
	\begin{Lemma}\label{lem:projection-new}
		Let $l$ be a geodesic in $\mathcal{X}$ and $x \in \mathcal{X}$. Consider $z$ a projection of $x$ on the geodesic $l$. Let $[x,z]$ be a geodesic segment between $x$ and $z$ and $y \in [x,z]$. Then $z$ is also a projection of $y$ on~$l$.
	\end{Lemma}
	
	\begin{proof}
		Let $t \in l$. We have $d(y,z)	=d(x,z)-d(x,y) 
			\leq d(x,t)-d(x,y)$ since $z$ is a projection of $x$ on $l$ and then,  by the triangle inequality,   $d(y,z) \leq d(y,t)$, which proves the Lemma.
	\end{proof}
	
	The following Lemma studies triangles with a ``right angle"  and states that the triangle inequality corresponding to this right angle is coarsely an equality in this case. 
	\begin{Lemma} \label{lem:bound-Gromov-product-right-angle}
		Let $l$ be a geodesic in $\mathcal{X}$, $x \in \mathcal{X}$ and $y \in l$. Consider $z$ a projection of $x$ on the geodesic~$l$. Then $d(x,y) \geq d(x,z)+d(z,y)-8\delta$.
	\end{Lemma}
	
	\begin{proof}
		Consider $[x,z]$ a geodesic segment between $x$ and $z$, $[x,y]$ a geodesic segment  between $x$ and $y$, and $[y,z]$ the geodesic segment between $y$ and $z$ contained in $l$. The triangle made with these three geodesic segments is $\delta$-thin since $\mathcal{X}$ is $\delta$-hyperbolic and thus consider $p \in [x,z]$, $q \in [z,y]$ and $r \in [x,y]$ given by Lemma \ref{lem:points on thin triangles}. Then, using the properties on $p,q,r$ and the triangle inequality we have
		\begin{align}
			d(x,y)  =d(x,r)+d(r,y) \geq d(x,p)+d(q,y) -2\delta 
			 = d(x,z)-d(p,z)+d(z,y)-d(z,q) -2\delta. \label{eq:p,q,r}
		\end{align}
		
		Moreover, since $z$ is a projection of $x$ on $l$ and $p \in [x,z]$, Lemma \ref{lem:projection-new} ensures that $d(p,z) \leq d(p,q)$, and  then, because of the bound on $d(p,q)$ given by Lemma \ref{lem:points on thin triangles}, 
		\begin{equation}   \label{eq:bound d(p,z)}
			d(p,z) \leq 2\delta.
		\end{equation}
		Using this last bound and the bound on $d(p,q)$, we can now also bound $d(z,q)$~: 
		\begin{equation} \label{eq:bound d(z,q)}
			d(z,q) \leq d(z,p)+d(p,q) \leq 2\delta + 2\delta =4\delta.
		\end{equation}
		The lemma then follows from combining \eqref{eq:p,q,r}, \eqref{eq:bound d(p,z)} and \eqref{eq:bound d(z,q)}.
	\end{proof}
	
	If $l$ is a geodesic, we say that a map $p : \mathcal{X} \to l$ is a \emph{projection} on $l$ if $p(x)$ is a projection on $l$ for all $x \in \mathcal{X}$. Such a map is not necessarily unique nor continuous. However, the next lemma asserts that it is always coarsely $1$-Lipschitz. 
	\begin{Lemma} \label{lem:proj-convex}
		Let $l$ be a geodesic in $\mathcal{X}$ and $p : \mathcal{X} \to l$ a projection on $l$. For all $x, y \in \mathcal{X}$, we have :
		\begin{equation}\label{eq:proj-bound-3distance}
			d(p(x),p(y))\leq \max(8\delta, 12\delta+d(x,y)-d(x,p(x))-d(y,p(y))).
		\end{equation}
		In particular :
		\begin{equation} \label{eq:proj-bound-distance}
			d(p(x),p(y)) \leq d(x,y)+12\delta. 
		\end{equation}
	\end{Lemma}

	\begin{proof}
		This is an immediate consequence of Lemma A.5 in \cite{schlich_equivalence_2024}. 
	\end{proof}~ \\ 
	
	\subsection{Hyperbolic isometries and stable norm} \label{subsec:hyperbolic isometries} ~ \\ 
	\indent Let $A$ be an isometry of $\mathcal{X}$. By invariance of the Gromov-product under isometries, two sequences $(x_n)_{n\in \N}$ and $(y_n)_{n\in \N}$ of points in $\mathcal{X}$ are equivalent if and only if the two sequences $(Ax_n)_{n\in \N}$ and $(Ay_n)_{n\in \N}$ are, and therefore we obtain a well-defined action of the isometry $A$ on $\partial \mathcal{X}$. 
	
	Let $x \in \mathcal{X}$. An isometry $A$ of $\mathcal{X}$ is said to be \emph{hyperbolic} if the sequence $(A^nx)_{n\in \Z}$ is a quasi-geodesic in $\mathcal{X}$ (this notion does not depend on the choice of the basepoint $x$). In this case, $A$ has exactly two fixpoints in $\partial \mathcal{X}$ which are the equivalent classes of $(A^nx)_{n \geq 0}$ and $(A^nx)_{n \leq 0}$. These two (distinct) points are respectively called the \emph{attracting} and \emph{repelling} fixpoints of $A$ and respectively denoted $A^+$ and $A^-$. If $L_A$ is a geodesic with endpoints $A^+$ and $A^-$, it is not necessarily $A$-invariant, but it is close to be. Indeed, as a consequence of Lemma \ref{lem:bigon_thin}, we obtain :
	
	\begin{Corollary}[Corollary of Lemma \ref{lem:bigon_thin}]\label{cor:axe-quasi-inv}
		Let $A$ be a hyperbolic isometry of $\mathcal{X}$. Let $L_A$ be an infinite geodesic with endpoints $A^+$ and $A^-$ and $x \in L_A$, then 
		\begin{equation*}
			d(Ax,L_A)\leq 2\delta.
		\end{equation*}
	\end{Corollary}
	
	\begin{proof}
		Observe that $A(L_A)$ is a geodesic with endpoints $A(A^+)=A^+$ and $A(A^-)=A^-$ and that $Ax \in A(L_A)$. Lemma \ref{lem:bigon_thin} then gives $d(Ax,L_A) \leq 2\delta$.
	\end{proof}
	
	We define the \emph{stable norm} of an isometry $A$ as :
	\begin{equation}\label{eq:def-stable-norm}
		l_S(A)= \underset{n \to \infty}{\lim}\frac{1}{n}d(A^nx,x).
	\end{equation}
	One can check that $l_S(A)$ is well-defined, do not depend on the choice of $x$, and is invariant under conjugation and inversion of $A$. Moreover, for all $n \in \N$, $d(A^nx,x) \geq nl_S(A)$. As a consequence, the hyperbolicity of an isometry can be characterized using the stable-norm : an isometry $A$ is hyperbolic if and only if $l_S(A)>0$. 
	
	We also introduce the \emph{displacement} of an isometry $A$ :
	\begin{equation}\label{eq:def-displacement}
		l(A)= \underset{x \in \mathcal{X}}{\inf}  \, d(Ax,x).
	\end{equation}
	We can easily check that $l_S(A) \leq l(A)$. The converse is coarsely true :
	\begin{Lemma}[Proposition 6.4 in \cite{coornaert_geometrie_1990}]
		$l_S(A) \leq l(A) \leq l_S(A) + 16\delta$.
	\end{Lemma}
	Therefore, these two quantities can be interchanged in our results up to the addition of an additive constant depending on $\delta$. In the following, we will work with the stable-norm which will be more convenient.\\
	
	When $A$ is a hyperbolic isometry and $L_A$ an $A$-invariant geodesic with endpoints $A^+$ and $A^-$, then it is not hard to see that for all $x \in L_A$, its image $Ax \in [x,A^+)$, where $[x,A^+)$ denotes the geodesic ray with endpoints $x$ and $A^+$. If $L_A$ is not assumed to be invariant anymore, we prove that the projection of $Ax$ still satisfies this property assuming that the stable norm is large enough. 

	\begin{Lemma} \label{lem:projection-Bx}
		Let $A$ be a hyperbolic isometry of $\mathcal{X}$. Let $L_A$ be a geodesic between $A^-$ and $A^+$ and $x \in L_A$. Let $p(Ax)$ be a projection of $Ax$ on $L_A$.
		If $l_S(A) > 6\delta$, then $p(Ax) \in [x,A^+)$.
	\end{Lemma}
	
	\begin{proof}
		Let us define the following sequence of points on $L_A$ by induction : 
		\begin{align*}
			&	p_0 = x \\
			&	p_{n+1} \text{ is a projection of $Ap_{n}$ on $L_A$} .
		\end{align*}
		For all $n \in \N$, the point $p_n$ belongs to $L_A$, so by Corollary \ref{cor:axe-quasi-inv},
		\begin{equation} \label{eq:d(Apn,p_{n+1})<2d}
			d(Ap_n,p_{n+1}) \leq 2\delta.
		\end{equation}
		Now let us prove that $d(p_n,A^nx) \leq 2n\delta$ for all $n \in \N$ by induction on $n$. 
		\begin{itemize}
			\item If $n=0$, $d(p_0,A^0x)=d(x,x)=0$. 
			\item Suppose $d(p_n,A^nx) \leq 2n\delta$ for some $n\in \N$. Then 
			\begin{align*}
				d(p_{n+1},A^{n+1}x)& \leq d(p_{n+1},Ap_n)+d(Ap_n,A^{n+1}x) \\
				&\leq  2\delta+d(p_n,A^nx) \text{ by \eqref{eq:d(Apn,p_{n+1})<2d}}  \\
				& \leq (n+1)2\delta.
			\end{align*}
		\end{itemize}
		
		We can now show that the sequences $(p_n)_{n \in \N}$ and $(A^nx)_{n\in \N}$ are equivalent :
		\begin{align*}
			(A^nx | p_n)_x & =\frac{1}{2}(d(x,A^nx)+d(x,p_n)-d(A^nx,p_n)) \\ 
			& \geq  \frac{1}{2}(nl_S(A)-2n\delta) = \frac{n}{2}(l_S(A)-2\delta).
		\end{align*}
		Since $l_S(A) > 2\delta$, the Gromov product $(A^nx | p_n)_x \underset{n \to +\infty}{\longrightarrow} +\infty$. As a consequence, $p_n \underset{n \to +\infty}{\longrightarrow} A^+$ using that $A^nx \underset{n \to +\infty}{\longrightarrow} A^+$ since $A$ is hyperbolic. \\ 
		
		Now assume by contradiction that $p_1 \in (A^-,x)$. Then let us prove that for all $n \in \N^*$, the projection $p_{n} \in (A^-,p_{n-1})$ by induction on $n \in \N^*$. 
		\begin{itemize}
			\item For $n=1$, it is our assumption. 
			\item Suppose that for some $n \in \N$, $p_{n} \in (A^-,p_{n-1})$. We apply the isometry $A$ : $Ap_{n}$ belongs to $AL_A \cap (A^-,Ap_{n-1})$. We would like to apply Lemma \ref{lem:pos-proj} with $l=L_A$, $l_\infty=A^-$, $y_0=p_n$, $x_0=Ap_{n-1}$ and $x=Ap_n$. Let us check the required inequalities :
			\begin{align*}
				d(Ap_n,p_n) & \geq l_S(A) >6\delta>2\delta+2\delta \geq d(Ap_{n-1},p_n) +2\delta \text{ by \eqref{eq:d(Apn,p_{n+1})<2d}}\\ 
				d(p_n,Ap_n) & \geq l_S(A)>6\delta.
			\end{align*} 
			Hence, Lemma \ref{lem:pos-proj} applies, and we deduce that $p_{n+1}$ belongs to $[p_n,A^-)$. 
		\end{itemize}
		Therefore, since for all $n \in \N$, $p_{n+1}$ belongs to $[p_n,A^-)$, we deduce that the sequence $(p_n)_{n \in \N}$ cannot converge to $A^+$, which is absurd. As a consequence, $p_1=p(Ax) \in [x,A^+)$. \\ 
	\end{proof}
	
		\begin{Lemma}\label{lem:d(x,A^2x)}
		Let $A$ be a hyperbolic isometry and $L_A$ a geodesic between $A^+$ and $A^-$. Let $x \in L_A$. Then $d(x,A^2x) \geq 2d(x,Ax)-20\delta$.
	\end{Lemma}
	
	\begin{proof}
		We denote $p_{L_A} : \mathcal{X} \to L_A$ a projection on $L_A$. Note that the inequality of the lemma is automatically satisfied when $d(x,Ax) \leq 10\delta$. Thus, in the following, we may assume that 
		\begin{equation} \label{eq:12delta}
			d(x,Ax) > 10\delta. 
		\end{equation}
		Consider the projection $p_{L_A}(Ax)$, we have $p_{L_A}(Ax) \in [x,A_\infty)$, where $A_\infty \in \{A^+,A^-\}$ and $[x,A_\infty)$ denotes the geodesic ray between $x$ and $A_\infty$ contained in $L_A$. By applying the isometry $A$, we then also obtain that $Ap_{L_A}(Ax) \in [Ax,A_\infty)$, with $[Ax,A_\infty)$ the geodesic ray between $Ax$ and $A_\infty$ contained in $A(L_A)$. Now let us bound by below $d(Ap_{L_A}(Ax),p_{L_A}(Ax))$ :
		\begin{align*}
			d(Ap_{L_A}(Ax),p_{L_A}(Ax)) & \geq d(Ap_{L_A}(Ax),Ax)-d(p_{L_A}(Ax),Ax) \text{ by the triangle inequality} \\ 
			& \geq d(Ax,x) -4\delta \text{ by Corollary \ref{cor:axe-quasi-inv}} \\
			& > 6\delta \text{ by \eqref{eq:12delta} }.
		\end{align*}
		 We now apply Lemma \ref{lem:pos-proj} with $l=L_A$, $y_0=p_{L_A}(Ax)$, $l_\infty=A_\infty$, $x_0=Ax$, $x=Ap_{L_A}(Ax)$. The inequality above, together with $d(Ax,p_{L_A}(Ax))\leq 2\delta$ (Corollary \ref{cor:axe-quasi-inv}), ensures that the hypotheses of the Lemma are satisfied; therefore, we deduce that the projection $p_{L_A}(Ap_{L_A}(Ax))$ on $L_A$ belongs to $[p_{L_A}(Ax),A_\infty)$. Now let us bound $d(p_{L_A}(Ap_{L_A}(Ax)),p_{L_A}(A^2x))$ :
		\begin{align*}
			d(p_{L_A}(Ap_{L_A}(Ax)),p_{L_A}(A^2x)) & \leq d(Ap_{L_A}(Ax),A^2x)+4\delta \text{ by the triangle inequality and Corollary \ref{cor:axe-quasi-inv}} \\
			& \leq d(p_{L_A}(Ax),Ax)+4\delta\\
			&  \leq 6\delta. \\
		\text{Also :} \qquad	d(p_{L_A}(Ax),p_{L_A}(A^2x)) & \geq d(Ax,A^2x)-4\delta \text{ by Corollary \ref{cor:axe-quasi-inv}} \\
			& \geq d(x,Ax)-4\delta \\
			& >6\delta.
		\end{align*}
		Therefore, we deduce that $p_{L_A}(A^2x) \in [p_{L_A}(Ax),A_\infty)$. We now have that $x, p_{L_A}(Ax)$ and $p_{L_A}(A^2x)$ are aligned on this order on $L_A$, and thus we can write : 
		\begin{align*}
			d(x,A^2x) &  \geq d(x,p_{L_A}(A^2x)) -2\delta \text{ using Corollary \ref{cor:axe-quasi-inv}} \\
			& \geq d(x,p_{L_A}(Ax))+d(p_{L_A}(Ax),p_{L_A}(A^2x))-2\delta \text{ because of the order of the points} \\
			& \geq d(x,Ax)+d(Ax,A^2x)-8\delta \text{ using Corollary \ref{cor:axe-quasi-inv}} \\
			& \geq 2d(x,Ax)-8\delta. \qedhere 
		\end{align*}
	\end{proof}
	From the previous lemma we now deduce the following Proposition, which says that the minimal displacement, or the stable norm, is nearly  attained on the axis of $A$ :
	
	\begin{Proposition}\label{prop:lS atteint sur l'axe}
		Let $A$ be a hyperbolic isometry and $L_A$ a geodesic between $A^+$ and $A^-$. Then, for all $x \in L_A$, 
		\begin{equation*}
			d(x,Ax) \leq l_S(A) +20\delta \leq l(A)+20\delta.
		\end{equation*}

	\end{Proposition}
	
	\begin{proof}
		First note that we always have $l_S(A)\leq l(A)$, so the right inequality is immediate. Let us prove the left one.
		Lemma \ref{lem:d(x,A^2x)} states that $d(x,A^2x) \geq 2d(x,Ax)-20\delta$. Note that for all $n \in \N$, $L_A$ is a geodesic between $(A^n)^+=A^+$ and $(A^n)^-=A^-$. So a direct induction using Lemma \ref{lem:d(x,A^2x)} shows that, for all integers $n \geq 1$ :
		\begin{equation*}
			d(x,A^{2^n}x) \geq 2^nd(x,Ax)-(2^n-1)20\delta.
		\end{equation*}
		We then divide this inequality by $2^n$ and take the limit. The left-hand side tends to $l_S(A)$ while the right-hand side tends to $d(x,Ax)-20\delta$, which gives the Lemma. \\
	\end{proof}
	
	We end this section with a result on the continuity of the stable-norm. 
	
	\begin{Theorem}[Theorem 1.9 in \cite{oregon-reyes_properties_2018}] \label{thm:l_S-continuous}
		Let $\mathrm{Isom}(\mathcal{X})$ be endowed with the compact-open topology. Then, the map $l_S : A \mapsto l_S(A)$ is continuous on $\mathrm{Isom}(\mathcal{X})$. \\
	\end{Theorem}
	
	\subsection{Local-global Lemma}~ \\ 
	\indent We end these preliminaries with the so-called local-global Lemma. There are several versions of it in the literature and we will use the one from \cite{aougab_constructing_2025}.

	\begin{Lemma}[Local-to-Global, \cite{aougab_constructing_2025}]. \label{lem:local-global} 
		Let $\mathcal{X}$ be $\delta$-hyperbolic, $A>0$ and suppose $x_0,\dots, x_n \in \mathcal{X}$
		are such that $(x_{i-1} | x_{i+1})_{x_i} \leq A$ and $d(x_i, x_{i+1}) > L:=3A + 26\delta$ for each $i$. Then any
		geodesic joining $x_0$ and $x_n$ passes within $D := A + 10\delta$ of each point $x_i$ and
		$$d(x_0,x_n) \geq d(x_0, x_1)+\dots + d(x_{n-1}, x_n)-2D(n-1).$$
	\end{Lemma} 
	\begin{proof} See Lemma 2.1 in \cite{aougab_constructing_2025}. The constants $3A+26\delta$ and $D=A+10\delta$ in Lemma \ref{lem:local-global} differ from those given by the authors in Lemma 2.1 in \cite{aougab_constructing_2025} ; we will now explain why and how we obtain those for Lemma \ref{lem:local-global}. 
		
	First, notice that the definition of hyperbolicity used in \cite{aougab_constructing_2025} differs  from ours. The authors assumed that the space is $\delta$-hyperbolic in the sense of Gromov. A metric space is \emph{$\delta$-hyperbolic in the sense of Gromov} if for all $x,y,z,w$ in $\mathcal{X}$, we have : $(x|y)_w \geq \min((x|z)_w,(y|z)_w)-\delta$.
	Hyperbolicity in the sense of Gromov and in the sense of Rips (our Definition, see section \ref{sec:preliminaries}) agrees, up to a change of the constant $\delta$. Next lemma from \cite{drutu_geometric_2018} explicits the constant needed to go from Rips hyperbolicity to Gromov hyperbolicity :
	\begin{Lemma}[Lemma 11.27 in \cite{drutu_geometric_2018}]\label{lem:rips-to-gromov}
		If $\mathcal{X}$ is a geodesic metric space that is $\delta$–hyperbolic in the sense of Rips, then $\mathcal{X}$ is 3$\delta$–hyperbolic in the sense of Gromov. 
	\end{Lemma}
	Therefore, we deduce that the equation $A+2\delta \geq \min \{(x_{i-1} | x_0)_{x_i}, (x_0 | x_n)_{x_i}, (x_n | x_{i+1})_{x_i})\}$ in the proof of Lemma 2.1 in \cite{aougab_constructing_2025} becomes $A+6\delta \geq \min \{(x_{i-1} | x_0)_{x_i}, (x_0 | x_n)_{x_i}, (x_n | x_{i+1})_{x_i})\}$ for us. However, equation (2.3) in \cite{aougab_constructing_2025}, which is $d(z,[x,y])-4\delta \leq (x | y)_z \leq d(z,[x,y])$ is still true in our context without altering the constants $4\delta$. Indeed, this inequality is a consequence of our Lemma \ref{lem:points on thin triangles-centroid}. No more changes are needed to follow the proof of \cite{aougab_constructing_2025}. 
	
	Finally, note that in Lemma 2.1 in \cite{aougab_constructing_2025}, the authors also prove in their Lemma that $\sum_{j=1}^n d(x_{j-1},x_j) \leq 2d(x_0,x_n)$ (last part of Lemma 2.1). This requires  increasing the constant compared to the one that one could obtained without asking for this last inequality. Since we don't need this inequality in Lemma \ref{lem:local-global}, we end up with the constants $D=A+10\delta$ and $L=3A+26\delta$. 
	\end{proof}
	
	\section{Irreducibility of Bowditch representations}
	\label{sec:irreducibiility}
	\begin{Definition}
		Let $\mathcal{X}$ be a geodesic $\delta$-hyperbolic space and $\rho : \Gamma \to \mathrm{Isom}(\mathcal{X})$ a representation of a finitely generated group $\Gamma$ in $\mathrm{Isom}(\mathcal{X})$. We say that $\rho$ is \emph{irreducible} if $\rho$ does not globally fix a point in the boundary $\partial \mathcal{X}$ of $\mathcal{X}$. If $\rho$ is not irreducible, we say that it is \emph{reducible}.
	\end{Definition}
	
	\begin{Remark} \label{rem:irreducibleF2} If $\Gamma=\F_2$ is the free group of rank two and $\{a,b\}$ is a free generating set of $\F_2$ such that $A=\rho(a)$ and $B=\rho(b)$ are hyperbolic, this is equivalent to saying that $\{A^+,A^-\} \cap \{B^+,B^-\}= \emptyset$.
	\end{Remark}
	
	The goal of this section is to prove the following :
	\begin{Proposition}\label{prop:irreducible}
		Let $\mathcal{X}$ be a $\delta$-hyperbolic, geodesic, and visibility space, and $\rho : \F_2 \to \mathrm{Isom}(\mathcal{X})$ satisfy \ref{cond:thm:Cdelta}. Then $\rho$ is irreducible. \\ 
	\end{Proposition}
	
	\subsection{Busemann functions}~\\ 
	\indent In the following, we will use the Busemann functions. Let us introduce them now. 
	
	\begin{Definition}
		Let $\mathcal{X}$ be a metric space. Let $l : \R \to \mathcal{X}$ be a parametrized geodesic in $\mathcal{X}$ and $x \in \mathcal{X}$. We define : 
		\begin{equation}
			\mathrm{Bus}_l(x)= \underset{t \to +\infty}{\lim}(d(l(t),x) -t ).
		\end{equation}
	\end{Definition}
	
	We now establish the following classical facts. 
	\begin{Lemma}\label{lem:busemann} Let $\mathcal{X}$ be a metric space and $x \in \mathcal{X}$.  
		\begin{enumerate}[label=(\alph*)]
			\item \label{bus:well-defined} For all parametrized geodesic $l : \R \to \mathcal{X}$, $\mathrm{Bus}_l(x)$ is well-defined.
			\item \label{bus:indep-geod} Let $c\geq 0$. If $l_1 : \R \to \mathcal{X}$ and $l_2 : \R \to \mathcal{X}$ are two parametrized geodesics such that for $t$ sufficiently large, $d(l_1(t),l_2(t)) \leq c$, then $\mathrm{Bus}_{l_2}(x) -c \leq \mathrm{Bus}_{l_1}(x) \leq \mathrm{Bus}_{l_2}(x) +c $. 
		\end{enumerate}
		Now assume moreover that $\mathcal{X}$ is a geodesic $\delta$-hyperbolic space.
		
		\begin{enumerate}[resume,label=(\alph*)]
			\item \label{bus:limit} Let $(x_n)_{n\in \N}$ be a sequence in $\mathcal{X}$. If $\mathrm{Bus}_l(x_n) \underset{n \to + \infty}{\longrightarrow} -\infty$, then $x_n \underset{n \to +\infty}{\longrightarrow} l(+\infty) \in \partial \mathcal{X}$, where $l(+ \infty)$ denotes  the limit in the boundary $\partial \mathcal{X}$ of $l(t)$ when $t \to + \infty$.

			\item \label{bus:l(A)} Let $A$ an hyperbolic isometry and $L_A$ a geodesic between its fixpoints $A^+$ and $A^-$. Denote (with a slight abuse of notation) again $L_A : \R \to \mathcal{X}$ a parametrization of the geodesic $L_A$ such that $L_A(t) \underset{t \to +\infty}{\longrightarrow} A^+$. Then 
			\begin{equation*}
				\mathrm{Bus}_{L_A}(x)-l_S(A) -32\delta \leq \mathrm{Bus}_{L_A}(Ax)\leq \mathrm{Bus}_{L_A}(x)-l_S(A) +32\delta. 
			\end{equation*}
		\end{enumerate}
	\end{Lemma}

	\begin{proof}
		\begin{enumerate}[label=(\alph*)]
			\item By the triangle inequality, for all $x \in \mathcal{X}$ and all $t \in \R$, we have $d(l(t),x)-t =d(l(t),x)-d(l(0),l(t))\geq -d(l(0),x)$ so the map $t \mapsto d(l(t),x)-t$ is bounded from below. In addition, for all $t_1 \leq t_2$, we have :
			\begin{align*}
				d(l(t_2),x)-t_2  \leq d(l(t_2),l(t_1))+d(l(t_1),x)-t_2 
				 \leq d(l(t_1),x)+|t_2-t_1|-t_2 
				 = d(l(t_1),x)-t_1
			\end{align*}
			so the map $t \mapsto d(l(t),x)-t$ is non-increasing. Hence the limit exists. 
			
			\item We have for all $t\in \R$, $d(l_1(t),x) -t \leq d(l_1(t),l_2(t))+d(l_2(t),x)-t$ so by taking the limit, we have $\mathrm{Bus}_{l_1}(x) \leq \mathrm{Bus}_{l_2}(x)+c$. The reverse inequality $\mathrm{Bus}_{l_2}(x) \leq \mathrm{Bus}_{l_1}(x)+c$ is obtained by exchanging the role of $l_1$ and $l_2$. 
			
			\item First note that for all $n \in \N$ and for all $t >0$, we have by definition of the Gromov product, $d(l(t),x_n)-t=d(l(t),x_n)-d(l(0),l(t))=d(l(0),x_n)-(x_n \,| \, l(t))_{l(0)}$. Then, since $d(l(t),x_n)-t$ converges when $t$ tends to infinity by \ref{bus:well-defined}, we obtain that $(x_n \, | \, l(t))_{l(0)}$ converges when $t$ tends to infinity and $\underset{t \to +\infty}{\lim} (x_n \, | \, l(t))_{l(0)} = d(l(0),x_n)-\mathrm{Bus}_l(x_n) \geq -\mathrm{Bus}_l(x_n)$.
			By assumption, $\mathrm{Bus}_l(x_n) \underset{n \to +\infty}{\longrightarrow} -\infty$ so we deduce that $\underset{n \to +\infty}{\lim} \, \underset{t \to +\infty}{\lim}(x_n \, | \, l(t))_{l(0)} = +\infty$, which, by definition of the boundary of $\mathcal{X}$, implies that $x_n \underset{n \to +\infty}{\longrightarrow} l(+\infty)$. 			
			
			\item Let us first bound $d(A(L_A(t)),L_A(t+l_S(A)))$.\\
			Consider $p(A(L_A(t)))$ a projection of $A(L_A(t))$ on $L_A$. Then, there exists $t' \in \R$ such that $p(A(L_A(t)))=L_A(t')$. Moreover, since $L_A(t) \in L_A$, Corollary \ref{cor:axe-quasi-inv} ensures that $d(A(L_A(t)),L_A(t')) \leq 2\delta$. From this we deduce, using the triangle inequality and the fact that $L_A$ is an isometry :  
			\begin{equation} \label{eq:d(AL_A(t),L_A(t+lS))}
				d(AL_A(t),L_A(t+l_S(A)))  \leq 2\delta + | t'-t-l_S(A) |.
			\end{equation}
			Let us now bound $| |t'-t| -l_S(A) |$. On the one hand we have :
			\begin{align*}
				l_S(A)-|t'-t|  \leq d(L_A(t),A(L_A(t)))-d(L_A(t'),L_A(t)) 
				\leq d(L_A(t'),A(L_A(t))) 
				\leq 2\delta.
			\end{align*}
			And on the other hand we have : 
			\begin{align*}
				|t'-t|-l_S(A) & \leq d(L_A(t'),L_A(t))-d(A(L_A(t)),L_A(t))+20\delta  \quad \text{ by Proposition \ref{prop:lS atteint sur l'axe}} \\ 
				& \leq d(L_A(t'),A(L_A(t))) + 20\delta \quad \text{ by the triangle inequality} \\
				& \leq 22\delta.
			\end{align*}
			Therefore
			\begin{equation}\label{eq:||t'-t|-l_S|}
				||t'-t|-l_S(A)| \leq 22\delta.
			\end{equation}
			
			Now let us assume that $l_S(A) >6\delta$. Then Lemma \ref{lem:projection-Bx} ensures that $L_A(t') \in [L_A(t),A^+)$, and so that $|t'-t|=t'-t$. From this, together with equation \eqref{eq:d(AL_A(t),L_A(t+lS))} and \eqref{eq:||t'-t|-l_S|}, we conclude 
			
			\begin{equation*}
				d(AL_A(t),L_A(t+l_S(A))) \leq 24\delta.
			\end{equation*}
			
			If $l_S(A) \leq 6\delta$, then 
			\begin{align*}
				d(AL_A(t),L_A(t+l_S(A))) &  \leq d(AL_A(t),L_A(t))+d(L_A(t),L_A(t+l_S(A))) \\
				& \leq l_S(A) + 20\delta +l_S(A) \text{ by Proposition \ref{prop:lS atteint sur l'axe}} \\
				& \leq 32\delta. 
			\end{align*}
			Then we proved that in any case 
			\begin{equation}\label{eq:bound	d(AL_A(t),L_A(t+l_S(A))) }
				d(AL_A(t),L_A(t+l_S(A)))  \leq 32\delta. 
			\end{equation}
			Moreover, reversing the orientation of the geodesic $L_A$ we obtain also :
			\begin{equation}\label{eq:bound	d(A^{-1}L_A(t),L_A(t-l_S(A)))}
				d(A^{-1}L_A(t),L_A(t-l_S(A)))  \leq 32\delta.
			\end{equation}
			Finally we can compute the Busemann function at $Ax$ :
			\begin{align*}
				\mathrm{Bus}_{L_A}(Ax)& =\underset{t \to + \infty}{\lim} (d(L_A(t),Ax)-t)=\underset{t \to +\infty}{\lim}(d(A^{-1}(L_A(t)),x)-t) \\
				& \leq \underset{t \to + \infty}{\lim} (d(A^{-1}(L_A(t)),L_A(t-l_S(A)))+d(L_A(t-l_S(A)),x)-t+l_S(A)-l_S(A)) \\ 
				& \leq \underset{t \to + \infty}{\lim} (d(L_A(t-l_S(A)),x)-t+l_S(A))-l_S(A) + 32\delta \text{ by \eqref{eq:bound	d(A^{-1}L_A(t),L_A(t-l_S(A)))}}\\ 
				& \leq  \underset{t \to + \infty}{\lim} (d(L_A(t),x)-t)-l_S(A)+32\delta \\ 
				& \leq \mathrm{Bus}_{L_A}(x)-l_S(A)+32\delta.
			\end{align*}
			
			A similar computation of $\mathrm{Bus}_{L_A}$ at $x$, using $A(L_A(t))$ instead of $A^{-1}(L_A(t))$ and \eqref{eq:bound	d(AL_A(t),L_A(t+l_S(A))) } instead of \eqref{eq:bound	d(A^{-1}L_A(t),L_A(t-l_S(A)))}, gives the other inequality. 
		\end{enumerate}
	\end{proof}
	
	\begin{Remark} \label{rem:lem:bus:CAT(-1)}
		An easy adaptation of the proof shows that when $\mathcal{X}$ satisfies hypotheses \ref{hyp:unique-visibility} and \ref{hyp:strong-asymptot} defined in the introduction (Remark \ref{rem:thm} \ref{rem:thm:hyp}), then the conclusion of Lemma \ref{lem:busemann} \ref{bus:l(A)} can be strengthened by  $\mathrm{Bus}_{L_A}(Ax)=\mathrm{Bus}_{L_A}(x)-l_S(A)$. Moreover, a slight modification of the proof of Lemma \ref{lem:busemann} \ref{bus:indep-geod} shows that : if $l_1, l_2 : \R \to \mathcal{X}$ are two parametrized geodesics such that $d(l_1(t),l_2(t)) \underset{t \to +\infty}{\longrightarrow} 0$, then $\mathrm{Bus}_{l_1}(x)=\mathrm{Bus}_{l_2}(x)$.  \\
	\end{Remark}
	
	\subsection{Application of Busemann functions to representations}~\\
	\indent In this section, we apply Busemann functions to representations to obtain some properties on reducible representations.
	
	Let $\mathcal{X}$ be a $\delta$-hyperbolic, geodesic, visibility space. 
	\begin{Proposition}\label{prop:bus-AB} Let $A$ and $B$ be two hyperbolic isometries of $\mathcal{X}$ such that $AB$ is also hyperbolic and assume that $A^+=B^-$. We have :
		\begin{enumerate}[label=(\alph*)]
			\item \label{prop:bus-AB-bound-l(AB)} $l_S(AB) \leq |l_S(A)-l_S(B)| +120\delta$.
			\item \label{prop:bus-AB-(AB)+} If $l_S(A)>l_S(B)+88\delta$, then $(AB)^+=A^+=B^-$. 
			\item \label{prop:bus-AB-(AB)-}If $l_S(A) <l_S(B)-88\delta$, then $(AB)^-=A^+=B^-$.
		\end{enumerate}
	\end{Proposition}
	
	\begin{proof} 
	\begin{enumerate}[label=(\alph*)]
	\item 
	By the visibility assumption, we can consider $L_A$ a geodesic between $A^+$ and $A^-$, $L_B$ a geodesic between $B^+$ and $B^-$ and $L_{AB}$ a geodesic between $(AB)^+$ and $(AB)^-$. Let $L_A : \R \to \mathcal{X}$ be any parametrization of $L_A$ such that $L_A(t) \underset{t \to +\infty}{\longrightarrow} A^+$. Since $B^-=A^+$, we deduce that $A^+$ is also a fixpoint of $AB$, and then $L_B$ and $L_{AB}$ are asymptotic to $L_A$, therefore we can use Lemma \ref{lem:asymptot-parametrization} to find parametrizations of $L_B$ and $L_{AB}$ such that, for $t$ sufficiently large, $d(L_A(t),L_B(t)) \leq 6\delta$ and $d(L_A(t),L_{AB}(t)) \leq 6\delta$. As a consequence, we also have for $t$ sufficiently large, $d(L_B(t),L_{AB}(t))\leq 12\delta$. Also note that $L_B(t) \underset{t \to +\infty}{\longrightarrow} B^-$ and $L_{AB}(t) \underset{t \to +\infty}{\longrightarrow} A^+=B^-$.
	
	First, note that Lemma \ref{lem:busemann} \ref{bus:l(A)} applied to $AB$ gives :
		\begin{equation}\label{eq:BusL_AB}
			\mathrm{Bus}_{L_{AB}}(x) \pm l_S(AB) -32\delta \leq \mathrm{Bus}_{L_{AB}}(ABx) \leq \mathrm{Bus}_{L_{AB}}(x)\pm l_S(AB)+32\delta.
		\end{equation}
		and the sign $\pm$ in front of $l_S(AB)$ depend on wether $A^+=(AB)^-$ or $A^+=(AB)^+$. 
		
		Moreover, using Lemma \ref{lem:busemann}, we derive the following inequalities :
		\begin{align}
			\mathrm{Bus}_{L_{AB}}(ABx) &  \leq \mathrm{Bus}_{L_A}(ABx)+6\delta \quad \text{ by Lemma \ref{lem:busemann} \ref{bus:indep-geod} } \nonumber \\ 
			& \leq \mathrm{Bus}_{L_A}(Bx)-l_S(A) +32\delta +6\delta \text{ by Lemma \ref{lem:busemann} \ref{bus:l(A)}} \nonumber \\ 
			& \leq \mathrm{Bus}_{L_B}(Bx)+6\delta -l_S(A)+38\delta \text{ by Lemma \ref{lem:busemann} \ref{bus:indep-geod}} \nonumber \\
			& \leq \mathrm{Bus}_{L_B}(x)+l_S(B)+32\delta -l_S(A)+44\delta \text{by Lemma \ref{lem:busemann} \ref{bus:l(A)}} \nonumber \\
			& \leq \mathrm{Bus}_{L_{AB}}(x)+12\delta + l_S(B)-l_S(A)+76\delta \text{ by Lemma \ref{lem:busemann} \ref{bus:indep-geod}} \nonumber \\
			& \leq  \mathrm{Bus}_{L_{AB}}(x) + l_S(B)-l_S(A)+ 88\delta. \label{eq:<l(B)-l(A)} \\
		\intertext{And similarly, we obtain the other inequality :}
			\mathrm{Bus}_{L_{AB}}(ABx) & \geq \mathrm{Bus}_{L_{AB}}(x) + l_S(B)-l_S(A)- 88\delta. \label{eq:>l(B)-l(A)}
		\end{align}
		Combining \eqref{eq:BusL_AB} with \eqref{eq:<l(B)-l(A)} and \eqref{eq:>l(B)-l(A)} gives the desired inequality. 

			\item Recall that we proved in \eqref{eq:<l(B)-l(A)}, that for all $x \in \mathcal{X}$ :
			\begin{equation*}
				\mathrm{Bus}_{L_{AB}}(ABx) \leq \mathrm{Bus}_{L_{AB}}(x)+l_S(B)-l_S(A)+88\delta. 
			\end{equation*}
			An immediate induction on $n \geq 1$ gives : 
			\begin{equation}\label{eq:bus(AB)n}
				\mathrm{Bus}_{L_{AB}}((AB)^nx) \leq \mathrm{Bus}_{L_{AB}}(x)+n(l_S(B)-l_S(A)+88\delta). 
			\end{equation}
			We deduce, if $l_S(A)>l_S(B)+88\delta$, that $\mathrm{Bus}_{L_{AB}}(x)+n(l_S(B)-l_S(A)+88\delta) \underset{n \to +\infty}{\longrightarrow}-\infty$ and then also $\mathrm{Bus}_{L_{AB}}((AB)^nx) \underset{n \to +\infty}{\longrightarrow}-\infty$. By Lemma \ref{lem:busemann} \ref{bus:limit} and the orientation of the parametrization of $L_{AB}$, this implies that $(AB)^nx \underset{n \to +\infty}{\longrightarrow} A^+=B^-$. Since $(AB)^nx \underset{n\to +\infty}{\longrightarrow}(AB)^+$, we have $(AB)^+=A^+=B^-$. 
			\item We use inequality \eqref{eq:>l(B)-l(A)} to deduce by induction, that for all integer $n \geq 1$ :
			\begin{equation*}
			\mathrm{Bus}_{L_{AB}}((AB)^nx) \geq \mathrm{Bus}_{L_{AB}}(x) + n(l_S(B)-l_S(A)- 88\delta). 
			\end{equation*}
			Since this inequality is true for all $x \in \mathcal{X}$, we also deduce : 
			\begin{equation*}
			 \mathrm{Bus}_{L_{AB}}((AB)^{-n}x) \leq \mathrm{Bus}_{L_{AB}}(x) - n(l_S(B)-l_S(A)- 88\delta). 
			\end{equation*}
			and now we use the assumption $l_S(A)<l_S(B)-88\delta$ and a similar argument than for \ref{prop:bus-AB-(AB)+} to conclude.
		\end{enumerate}
	\end{proof}
	
	\begin{Remark} \label{rem:lS(AB)-CAT(-1)}
		Using Remark \ref{rem:lem:bus:CAT(-1)}, one can easily adapt the proof of Proposition \ref{prop:bus-AB} when $\mathcal{X}$ satisfies hypotheses \ref{hyp:unique-visibility} and \ref{hyp:strong-asymptot} to show that the conclusion of Proposition \ref{prop:bus-AB} in this case simplifies to :
		\begin{enumerate}[label=(\alph*)]
			\item $l_S(AB)=|l_S(A)-l_S(B)|$.
			\item If $l_S(A)>l_S(B)$, then $(AB)^+=A^+=B^-$.
			\item If $l_S(A)<l_S(B)$, then $(AB)^-=A^+=B^-$.
		\end{enumerate} 
	\end{Remark}
	
	\begin{Lemma}\label{lem:basis-F2}
		Let $\rho : \F_2 \to \mathrm{Isom}(\mathcal{X})$ be a reducible representation such that every primitive element of $\F_2$ is sent to a hyperbolic isometry by $\rho$. Then, there exists a sequence $(e_n=\{a_n,b_n\})_{n\in \N}$ of free basis of $F_2$ such that, denoting $A_n=\rho(a_n)$ and $B_n=\rho(b_n)$, the following holds : 	
		\begin{itemize}
			\item $A_n^+=B_n^-$
			\item $\{a_{n+1},b_{n+1}\}=\{a_n,a_nb_n\}$ or $\{a_{n+1},b_{n+1}\}=\{a_nb_n,b_n\}$
		\end{itemize}
	\end{Lemma}
	
	\begin{proof}
		To define $e_0=\{a_0,b_0\}$, take any basis $\{a,b\}$ of $\F_2$ and denote $A=\rho(a)$, $B=\rho(b)$. Since $\rho$ is reducible, it implies that $\{A^+,A^-\}\cap \{B^+,B^-\} \neq \emptyset$ and then, up to exchanging $a$ and $b$ and replacing $a$ with its inverse and $b$ with its inverse, we can assume that $A^+=B^-$. Now suppose that $e_n=\{a_n,b_n\}$ is defined and satisfy the conditions of the lemma, then consider the element $a_nb_n$. It is primitive so its image $A_nB_n$ is hyperbolic by asumption and since $A_n^+=B_n^-$, we deduce that $A_n^+=B_n^-$ is a fixed point of $A_nB_n$. If $A_n^+=B_n^-=(A_nB_n)^+$, then set $e_{n+1}=\{a_nb_n,b_n\}$ which is a free basis of $\F_2$ and if $A_n^+=B_n^-=(A_nB_n)^-$, then set $e_{n+1}=\{a_n,a_nb_n\}$ which is again also a free basis of $\F_2$.
	\end{proof}
	
	We now show that this sequence of basis ``reduces" the stable norm along the way. 
	\begin{Proposition}\label{prop:irred} Let $\rho : \F_2 \to \mathrm{Isom}(\mathcal{X})$ be a reducible representation such that every primitive element is sent to a hyperbolic isometry by $\rho$. Let $(e_n=\{a_n,b_n\})_{n \in \N}$ be a sequence a free basis of $\F_2$ satisfying the conditions of Lemma \ref{lem:basis-F2}. 
	Then, for all $n \in \N$, there exists $m>n$ such that one of the two following holds :
		\begin{enumerate}[label=(\alph*)]
			\item \label{cond:bounded} There exists $w_m \in e_m$ such that $w_m \notin e_n$ and $l_S(W_m) \leq 329\delta$. 
			\item \label{cond:smaller} Denote $e_n=\{u_n,v_n\}$ with $l_S(U_n)\geq l_S(V_n)$. Then $e_m=\{u_nv_n^{m-n},v_n\}$ and $l_S(U_nV_n^{m-n}) \leq l_S(U_n)-\delta$. 
		\end{enumerate}
	\end{Proposition}
	
	\begin{proof}
		Let $e_n =\{a_n,b_n\}$ be a basis in the sequence. In order to simplify the notations, we will omit the indices $n$ and simply denote $e_n=\{a,b\}$, together with $A=\rho(a)$, $B=\rho(b)$. 
		We distinguish the three following cases :
		\begin{enumerate}
			\item If $|l_S(A)-l_S(B)| \leq 88\delta$. We have by Proposition \ref{prop:bus-AB} \ref{prop:bus-AB-bound-l(AB)} $l_S(AB) \leq |l_S(A)-l_S(B)| +120\delta $ and thus $l_S(AB) \leq 88\delta +120\delta=208\delta$. Therefore condition \ref{cond:bounded} of Proposition \ref{prop:irred} is satisfied with $m=n+1$.
			
			\item If $l_S(A)>l_S(B)+88\delta$, then by Proposition \ref{prop:bus-AB} \ref{prop:bus-AB-(AB)+}, $(AB)^+=B^-$ so $e_{n+1}=\{ab,b\}$, and by Proposition \ref{prop:bus-AB} \ref{prop:bus-AB-bound-l(AB)}, $l_S(AB) \leq l_S(A)-l_S(B)+120\delta$. Now we distinguish two cases :
			\begin{enumerate}
				\item If $l_S(B) > 121\delta$. Then $l_S(AB) \leq l_S(A)-\delta$, and therefore condition \ref{cond:smaller} of Proposition \ref{prop:irred} is satisfied with $m=n+1$.  
				
				\item If $l_S(B) \leq 121\delta$. Notice that for all integer $k \geq 1$, $(B^k)^-=A^+$ and then by Proposition~\ref{prop:bus-AB}~\ref{prop:bus-AB-bound-l(AB)} $l_S(AB^k)\leq | l_S(A)-l_S(B^k)|+120\delta$. But $l_S(B^k)=kl_S(B)$, so $l_S(AB^k)\leq | l_S(A)-kl_S(B)|+120\delta$. Let $K=\lfloor \frac{l_S(A)-88\delta}{l_S(B)}\rfloor$. Then $K \geq1$. For all $k \leq K-1$, notice that $kl_S(B)+88\delta <l_S(A)$ and therefore we deduce by Proposition \ref{prop:bus-AB} \ref{prop:bus-AB-(AB)+} that $e_{n+k}=\{ab^k,b\}$. We also deduce that $ab^N$ belong to $e_{n+K}$ (which is either $\{ab^K,b\}$ or $\{ab^{K-1},ab^{K}\}$).
				
				Now let us first assume that $K=\frac{l_S(A)-88\delta}{l_S(B)}$. Then $l_S(B^K)=Kl_S(B)=l_S(A)-88\delta$, and $(B^K)^-=A^+$ so $l_S(AB^K) \leq |l_S(A)-Kl_S(B)|+120\delta=88\delta+120\delta=208\delta$ and therefore condition \ref{cond:bounded} of Proposition~\ref{prop:irred} is satisfied for $m=n+K$. 
				
				Now assume that $K<\frac{l_S(A)-88\delta}{l_S(B)}$. In this case $e_{n+K}=\{ab^K,b\}$ since $Kl_S(B)+88\delta<l_S(A)$.
				\begin{itemize}
					\item If $K$ satisfies $Kl_S(B) >121\delta$, then we have $l_S(AB^K) \leq l_S(A)-Kl_S(B)+120\delta \leq l(A)-\delta$ and so condition \ref{cond:smaller} of Proposition \ref{prop:irred} is satisfied with $m=n+K$. 
					\item If $Kl_S(B) \leq 121\delta$, we deduce from the inequality $l_S(A) <(K+1)l_S(B)+88\delta$ that $l_S(A)-l_S(B) < Kl_S(B)+88\delta \leq 209\delta$. If follows that $l_S(AB) \leq l_S(A)-l_S(B)+120\delta \leq 209\delta +120\delta = 329\delta$, and the condition \ref{cond:bounded} of Proposition \ref{prop:irred} is satisfied for $m=n+1$.
				\end{itemize}
			\end{enumerate}
			\item If $l_S(A)<l_S(B)-88\delta$, the proof is similar using Proposition \ref{prop:bus-AB} \ref{prop:bus-AB-(AB)-} to ensure that $(AB)^-=A^+$, and then exchanging the role of $A$ and $B$ in the proof. 
		\end{enumerate}
	\end{proof}
	This allows us to prove Proposition \ref{prop:irreducible}.
	\begin{proof}[Proof of Proposition \ref{prop:irreducible}]
		Let $\rho : \F_2 \to \mathrm{Isom}(\mathcal{X})$ be a reducible representation sending all primitive elements to hyperbolic isometries and let $(e_n)_{n\in \N}$ be a sequence of free bases of $\F_2$ satisfying the conditions of Lemma \ref{lem:basis-F2}. This sequence also satisfies the property of Proposition \ref{prop:irred}. As a consequence of Proposition \ref{prop:irred}, we immediately deduce that we can construct an infinite number of primitive elements $\gamma$ of $\F_2$ (each belonging to some $e_n$) satisfying $l_S(\rho(\gamma)) \leq 329 \delta$. This prevents $\rho$ from satisfying \ref{cond:thm:Cdelta}.  
	\end{proof}
	
	\begin{Remark} \label{rem:irreducible-CAT(-1)}
	When $\mathcal{X}$ satisfies hypotheses \ref{hyp:unique-visibility} and \ref{hyp:strong-asymptot} and $\rho : \F_2 \to \mathrm{Isom}(X)$ is a reducible representation, Remark \ref{rem:lS(AB)-CAT(-1)} together with Lemma \ref{lem:basis-F2} allow us to construct a sequence of primitive elements $(\gamma_n)_{n\in \N}$ such that $l_S(\gamma_n) \underset{n \to + \infty}{\longrightarrow} 0$. 
	We then deduce the following strongest version of Proposition~\ref{prop:irreducible}~: Let $\mathcal{X}$ be a $\delta$-hyperbolic, geodesic, visibility space satisfying hypotheses \ref{hyp:unique-visibility} and \ref{hyp:strong-asymptot} and $K>0$. Let $\rho : \F_2 \to \mathrm{Isom}(\mathcal{X})$ be a representation sending primitive elements of $\F_2$ to hyperbolic isometries of $\mathcal{X}$ such that the set $\{[\gamma] \in \mathcal{P}(\F_2) \, \vert \, l_S(\rho(\gamma)) \leq K \}$ is finite. Then $\rho$ is irreducible. 
	\end{Remark}
	
	\section{Length of a product of hyperbolic isometries} \label{sec:l(AB)}
	In this section, $\mathcal{X}$ is a $\delta$-hyperbolic, geodesic, visibility space. \\ 
	
	The goal of this section is to prove the following proposition which studies the length of a product of two hyperbolic isometries : 
	\begin{Proposition} \label{prop:ineq-l(AB)}
		Denote $C_\delta=256\delta$. \\
		Let $A$ and $B$ be two hyperbolic isometries of $\mathcal{X}$ such that $\{A^+,A^-\} \cap \{B^+,B^-\} = \emptyset$. Assume that $l_S(A) > C_\delta$, $l_S(B) > C_\delta$. Then the following inequality holds :
		\begin{equation*}
			\max(l_S(AB),l_S(AB^{-1}) \geq l_S(A) + l_S(B) -C_\delta.
		\end{equation*}
	\end{Proposition}
	
	Let us fix $A$ and $B$ two hyperbolic isometries of $\mathcal{X}$.
	By the hyperbolicity assumption, the attracting and repelling points $A^+,A^-,B^+,B^-$ are well-defined. We define $\mathrm{Axis}(A)$, respectively $\mathrm{Axis}(B)$ to be the union of all geodesic in $\mathcal{X}$ with endpoints $A^-$ and $A^+$, respectively $B^-$ and $B^+$. By the visibility assumption, $\mathrm{Axis}(A)$ and $\mathrm{Axis}(B)$ are non-empty. We also fix once and for all this section $L_A$ and respectively $L_B$ two geodesics between $A^+$ and $A^-$, and respectively between $B^+$ and $B^-$.\\
	
	The following lemma is a direct consequence of the continuity and the properness of the map $(x,y) \in L_A \times L_B \mapsto d(x,y)$. The properness of the map follows from the condition $\{A^+,A^-\}\cap \{B^+,B^-\}=\emptyset$.
	
	\begin{Lemma}\label{lem:xAxB}
		Assume $\{A^+,A^-\}\cap \{B^+,B^-\}=\emptyset$. Then the distance between $L_A$ and $L_B$ is attained~:
		There exists $(x_A,x_B) \in L_A \times L_B$ such that $d(x_A,x_B)=d(L_A,L_B)$. 
	\end{Lemma}
	
	Let us denote by $p_A : \mathcal{X} \to L_A$ a projection map onto $L_A$ and  $p_B : \mathcal{X} \to L_B$ a projection map onto $L_B$. By definition, these maps satisfy for all $x \in \mathcal{X}$, $d(x,L_A)=d(x,p_A(x))$ and $d(x,L_B)=d(x,p_B(x))$. Such maps always exist but are not necessarily unique. \\

	In the proof of Proposition \ref{prop:ineq-l(AB)}, we will distinguish the case where the geodesics $L_A$ and $L_B$ stay far from each other, and the case where they come close at some point. \\
	
	\subsection{When the axes are far from each other}~ \label{sec:axes-far} \\
	\indent Recall that $A$ and $B$ are two hyperbolic isometries which satisfy $\{A^+,A^-\} \cap \{B^+,B^-\}=\emptyset$. In this section, we assume moreover that $d(L_A,L_B) > 44\delta$. \\
	
	Let us introduce the points $(x_A,x_B)\in L_A \times L_B$ given by Lemma \ref{lem:xAxB} which satisfy $d(x_A,x_B)=d(L_A,L_B)$.
	Let us fix an integer $n \in \N$ and define the following sequence of points $(x_i)_{0 \leq i \leq 4n}$~in~$\mathcal{X}$~:
	\begin{equation}
		\left\{
		\begin{array}{rll}
		x_{4j}&=(AB)^jx_A, & \text{ for all } 0 \leq j \leq n, \\ 
		x_{4j+1}&=(AB)^jAx_A, &\text{ for all } 0 \leq j \leq n-1, \\
		x_{4j+2}&=(AB)^jAx_B, &\text{ for all } 0 \leq j \leq n-1, \\
		x_{4j+3}&=(AB)^{j+1}x_B, & \text{ for all } 0 \leq j \leq n-1.
		\end{array}
		 \right. 
	\end{equation}
	
	In this configuration, we are able to bound the Gromov-product between consecutive points of the sequence. 
	
	\begin{Lemma} 
		\label{lem:bound-Gromov-product-x_i}
		For all $0 < i < 4n$, $(x_{i-1}|x_{i+1})_{x_i} \leq 6\delta$.
	\end{Lemma}
	
	\begin{proof} Let us first compute these Gromov-products.
			\begin{itemize}
			\item Suppose that $i=4j$, with $j\in \N$, then \\ $(x_{i-1}|x_{i+1})_{x_i}=(x_{4j-1}|x_{4j+1})_{x_{4j}}=((AB)^{j}x_B|(AB)^jAx_A)_{(AB)^jx_A}=(x_B|Ax_A)_{x_A}$.
			\item Suppose that $i=4j+1$, with $j\in \N$, then \\ $(x_{i-1}|x_{i+1})_{x_i}=(x_{4j}|x_{4j+2})_{x_{4j+1}}=((AB)^jx_A|(AB)^jAx_B)_{(AB)^jAx_A}=(x_A|Ax_B)_{Ax_A}$.
			\item Suppose that $i=4j+2$, with $j\in \N$, then \\ $(x_{i-1}|x_{i+1})_{x_i}=(x_{4j+1}|x_{4j+3})_{x_{4j+2}}=((AB)^jAx_A|(AB)^{j+1}x_B)_{(AB)^jAx_B}=(x_A|Bx_B)_{x_B}$.
			\item Suppose that $i=4j+3$, with $j\in \N$, then \\ $(x_{i-1}|x_{i+1})_{x_i}=(x_{4j+2}|x_{4j+4})_{x_{4j+3}}=((AB)^jAx_B|(AB)^{j+1}x_A)_{(AB)^{j+1}x_B}=(x_B|Bx_A)_{Bx_B}$.
		\end{itemize}
		Hence it is sufficient to bound the four Gromov-products  $(x_B|Ax_A)_{x_A}$, $(x_A|Ax_B)_{Ax_A}$,$(x_A|Bx_B)_{x_B}$ and $(x_B|Bx_A)_{Bx_B}$. \\
		
		Let us first bound $(x_B|Ax_A)_{x_A}$. By definition, $x_A \in L_A$, so in particular Corollary \ref{cor:axe-quasi-inv} gives the existence of a point $y_A \in L_A$ such that 
		\begin{equation}\label{eq:Morse-constant}
			d(y_A,Ax_A)\leq 2\delta.
		\end{equation}
		Recall that $d(x_A,x_B)=d(L_A,L_B)$ so $x_A$ is a projection of $x_B$ on $L_A$, and $y_A \in L_A $, so Lemma \ref{lem:bound-Gromov-product-right-angle} immediately gives a bound on the Gromov product $(x_B|y_A)_{x_A}$ :
		\begin{equation}\label{eq:(x_B|y_A)}
			(x_B|y_A)_{x_A}=\frac{1}{2}(d(x_A,x_B)+d(x_A,y_A)-d(x_B,y_A))\leq \frac{1}{2}8\delta=4\delta.
		\end{equation}
		Using the triangle inequality and the two bounds \eqref{eq:Morse-constant} and \eqref{eq:(x_B|y_A)} we now deduce the bound on the Gromov product $(x_B | Ax_A)_{x_A}$ :
		\begin{align*}
			(x_B | Ax_A)_{x_A} & =\frac{1}{2}(d(x_A,x_B)+d(x_A,Ax_A)-d(x_B,Ax_A)) \\
			& \leq \frac{1}{2}(d(x_A,x_B)+d(x_A,y_A)+2\delta-d(x_B,y_A)+2\delta) \\
			& \leq (x_B|y_A)_{x_A}+2\delta \\ 
			& \leq 6\delta.
		\end{align*}
		
		The proof of the bound of the Gromov product $(x_A|Bx_B)_{x_B}$ is the same by exchanging the rôle of $A$ and $B$. A similar proof also works to bound $(x_A|Ax_B)_{Ax_A}$ using that $Ax_A$ is a projection of $Ax_B$ on $A(L_A)$ and to bound $(x_B|Bx_A)_{Bx_B}$ using that $Bx_B$ is a projection of $Bx_A$ on $B(L_B)$. 
	\end{proof}
	
	\begin{Lemma} \label{lem:ineg crucial cas loin}
		 Assume that $d(L_A,L_B) > 44\delta$, $l_S(A) > 44\delta$ and $l_S(B) > 44\delta$. Then :
		\begin{equation}
			l_S(AB) \geq l_S(A)+l_S(B)-40\delta.
		\end{equation}
	\end{Lemma}
	
	\begin{proof}
		We now use the local-global Lemma \ref{lem:local-global} with $A=6\delta$. Note that $L=3A+26\delta=44\delta$, and $D=A+10\delta=16\delta$. 
		
		\begin{itemize}
			\item Suppose that $i=4j$, with $j \in \N$, then \\  $d(x_i,x_{i+1})=d(x_{4j},x_{4j+1})=d((AB)^jx_A,(AB)^jAx_A)=d(x_A,Ax_A) \geq l_S(A)$.
			\item  Suppose that $i=4j+1$, with $j \in \N$, then \\  $d(x_i,x_{i+1})=d(x_{4j+1},x_{4j+2})=d((AB)^jAx_A,(AB)^jAx_B)=d(Ax_A,Ax_B) =d(L_A,L_B)$.
			\item Suppose that $i=4j+2$, with $j \in \N$, then \\  $d(x_i,x_{i+1})=d(x_{4j+2},x_{4j+3})=d((AB)^jAx_B,(AB)^{j+1}x_B)=d(x_B,Bx_B) \geq l_S(B)$.
			\item Suppose that $i=4j+3$, with $j \in \N$, then \\  $d(x_i,x_{i+1})=d(x_{4j+3},x_{4(j+1)})=d((AB)^{j+1}x_B,(AB)^{j+1}x_A)=d(x_B,x_A) =d( L_A,L_B)$.
		\end{itemize}
		
		Then, using our hypotheses on $d(L_A,L_B)$, $l_S(A)$ and $l_S(B)$ and Lemma \ref{lem:bound-Gromov-product-x_i} we deduce that our sequence of points $(x_i)_{0\leq i\leq n}$ satisfies the hypothesis of the local-global Lemma \ref{lem:local-global}, and thus we obtain~:
		\begin{align*}
			d(x_0,x_{4n}) & \geq d(x_0,x_1)+\dots+d(x_{4n-1},x_{4n})-2D(4n-1) \\ 
			d(x_A,(AB)^nx_A) & \geq nd(x_A,Ax_A)+nd(x_A,x_B)+nd(x_B,Bx_B)+nd(x_B,x_A)-32\delta(4n-1) \\
			d(x_A,(AB)^nx_A)& \geq n(d(x_A,Ax_A)+d(x_B,Bx_B)+2\times 44\delta )-32\delta(4n-1).
		\end{align*}
		To obtain the inequality on $l_S(AB)$, divide the previous one by $n$ then take the limit, and use  $\frac{1}{n}d(x_A,(AB)^nx_A)\underset{n \to \infty}{\longrightarrow} l_S(AB)$, $d(x_A,Ax_A)\geq l_S(A)$ and $d(x_B,Bx_B)\geq l_S(B)$. 
	\end{proof}
	
	\begin{Remark}
	In Lemma \ref{lem:ineg crucial cas loin}, we did not assume that $l_S(AB) \geq l_S(AB^{-1})$. In particular, in this case ($d(L_A,L_B)>44\delta$), the inequality of Proposition \ref{prop:ineq-l(AB)} is true both for $l_S(AB)$ and $l_S(AB^{-1})$, not only the maximum. 
	
	In fact, adapting the proof of Lemma \ref{lem:ineg crucial cas loin}, we could also prove that for all cyclically reduced word $W$ on the alphabet $\{A,A^{-1},B,B^{-1}\}$, we have $l_S(W) \geq  (\min(l_S(A),l_S(B))-20\delta)|W|$, where $|W|$ is the word length of $W$. This in particular would prove that a representation of $\F_2$ in $\mathrm{Isom}(\mathcal{X})$ is automatically convex-cocompact when $l_S(A)>44\delta$, $l_S(B)>44\delta$ and $d(L_A,L_B)>44\delta$.\\
	\end{Remark}
	
	\subsection{When the axes come close at some point}~\label{sec-axes-close}\\ 
	\indent Recall that $A$ and $B$ are two hyperbolic isometries which satisfy $\{A^+,A^-\} \cap \{B^+,B^-\}=\emptyset$. In this section, we assume moreover that $d(L_A,L_B) \leq 44\delta$ and we denote  $D_\delta=44\delta$. \\ 
	
	Let $\varphi_A$ be an isometry between $\R$ and $L_A$ such that $\varphi_A(t) \underset{t\to \pm\infty}\longrightarrow A^\pm$ and $\varphi_B$ an isometry between $\R$ and $L_B$ such that $\varphi_B(t) \underset{t\to \pm\infty}\longrightarrow B^\pm$. 
	We define the four following elements of $\mathcal{X}$ :
	
	\begin{Definition}\label{def:x_A,y_A,x_B,y_B}
		\begin{align*}
			x_A & = \varphi_A( \inf \{ t \in \R \mid d(\varphi_A(t),L_B) \leq D_\delta \}    ) \\ 
			y_A & = \varphi_A( \sup \{ t \in \R \mid d(\varphi_A(t),L_B) \leq D_\delta \}    ) \\ 
			x_B & = \varphi_B( \inf \{ t \in \R \mid d(\varphi_B(t),L_A) \leq D_\delta \}    ) \\
			y_B & = \varphi_B( \sup \{ t \in \R \mid d(\varphi_B(t),L_A) \leq D_\delta \}    )
		\end{align*}
	\end{Definition}
	Note that the isometries $\varphi_A$ and $\varphi_B$ induce an order on the geodesic $L_A$ and $L_B$ coming from the order on $\R$. We can extend this order to the endpoints of $L_A$ and $L_B$. We make the following observations on the order of the elements $x_A$, $y_A$, $x_B$ and $y_B$ defined previously. 
	\begin{Lemma}\label{lem:x_A<x_B}
		We have 
		\begin{align*}
			A^- < x_A \leq y_A < A^+, \\
			B^- < x_B \leq y_B < B^+.
		\end{align*}
		Moreover, the two infimum and the two supremum in Definition \ref{def:x_A,y_A,x_B,y_B} are attained, and then
		\begin{equation}\label{eq:d(xA,LA)}
		d(x_A,L_B)=d(y_A,L_B)=d(x_B, L_A)=d(y_B,L_A)= D_\delta.
		\end{equation}
	\end{Lemma}
	\begin{proof}
		First notice that the assumption $d(L_A,L_B)\leq D_\delta$ implies that  $x_A \neq A^+$, $y_A \neq A^-$, $x_B \neq B^+$ and $y_B \neq B^-$. 	The inequalities $x_A \leq y_A$ and $x_B \leq y_B$ follow immediately from the definitions. \\ 
		The fact that $x_A \neq A^-$, $y_A \neq A^+$, $x_B \neq B^-$ and $y_B \neq B^+$ is a consequence of the hypothesis $\{A^-,A^+\} \cap \{B^-,B^+\} = \emptyset$. Indeed, if $x_A = A^-$, then we could find a sequence $(x_n)_{n \in \N}$ of elements of $L_A$ which would stay at bounded distance from $L_B$ and such that $x_n \to A^-$. We could then conclude that $A^- \in \{B^-,B^+\}$ which is a contradiction. The proof of the other inequalities are similar. Finally, the fact that the infimum and the supremum are attained and that equation \eqref{eq:d(xA,LA)} holds is just a consequence of the continuity of the distance.
	\end{proof}
	Denote $[x_A,y_A]$ a geodesic segment between $x_A$ and $y_A$ contained in $L_A$ and $[x_B,y_B]$ a geodesic segment between $x_B$ and $y_B$ contained in $L_B$. Recall that we introduced two projection maps : $p_A : \mathcal{X} \to L_A$ and $p_B : \mathcal{X} \to L_B$. An easy consequence of Lemma \ref{lem:x_A<x_B} on the projections of $x_A, y_A, x_B$ and $y_B$  is the following :
	
	\begin{Corollary}\label{cor:xA<p(xB),p(yB)<yA}
		The projections $p_B(x_A)$ and $p_B(y_A)$ belong to $[x_B,y_B]$ and the projections $p_A(x_B)$ and $p_A(y_B)$ belong to $[x_A,y_A]$. 
	\end{Corollary} 
	
	\begin{proof} By Lemma \ref{lem:x_A<x_B}, $d(p_B(x_A),x_A) \leq D_\delta$ and $d(p_B(y_A),x_A) \leq D_\delta$, and so it immediately follows from the definition of $x_B$ and $y_B$ (Definition \ref{def:x_A,y_A,x_B,y_B}) that $x_B \leq p_B(x_A) \leq y_B$ and $x_B \leq p_B(y_A) \leq y_B$, since $p_B(x_A), p_B(y_A) \in L_B$ and $x_A \in L_A$. We obtain $p_A(x_B),p_A(y_B) \in [x_B,y_B]$ similarily by exchanging the roles of $A$ and $B$. 
	\end{proof}
	
	Let $[y_A,A^+)$ and $(A^-,y_A]$ be respectively the two half-geodesics contained in $L_A$ from $y_A $ to $A^+$ and from $y_A$ to $A^-$. Let $[y_A,p_B(y_A)]$ be a geodesic segment between $y_A$ and $p_B(y_A)$ and $[p_B(x_A),p_B(y_A)]$ the geodesic segment between $p_B(x_A)$ and $p_B(y_A)$ contained in $L_B$. For all $x \in L_B$, we also denote by $(B^-,x]$ and $[x,B^+)$ respectively the two half-geodesics contained in $L_B$ from $x$ to $B^-$ and from $x$ to $B^+$.\\
			
	We now want to prove that the Gromov-product between a point $q$ on $[y_A,A^+)$ and a point $p$ on $(B^-,x]$, seen from a basepoint on $[p_B(x_A),p_B(y_A)]$, is bounded by a constant depending only on $\delta$. 
		\begin{Lemma}\label{lem:ineg q-loin}
			Assume that $x_B,p_B(x_A),p_B(y_A)$ and $y_B$ are aligned on this order on $L_B$. Choose $x \in [p_B(x_A),p_B(y_A)]$. 
			If $p \in (B^-,x]$ and $q \in [y_A,A^+)$, or if $p \in [x,B^+)$ and $q \in (A^-,x_A]$, then we have:
			\begin{equation}
				d(p,q) \geq d(p,x)+d(x,q)-12\delta.
			\end{equation}
		\end{Lemma}
		
		\begin{proof} Let us assume that $p \in (B^-,x]$ and $q \in [y_A,A^+)$. The case $p \in [x,B^+)$ and $q \in (A^-,x_A]$ is similar changing $(A,B)$ to $(A^{-1},B^{-1})$. 
			
			Let $[p,p_B(y_A)]$ be the geodesic segment from $p$ to $p_B(y_A)$ contained in $L_B$ and $[y_A,q]$ the geodesic segment from $y_A$ to $q$ contained in $L_A$.	Consider the quadrilateral made up by the four geodesic segments $[p,p_B(y_A)]$, $[p_B(y_A),y_A]$, $[y_A,q]$ and $[q,p]$. Assume by contradiction that for all $t \in [p,q]$, $d(t,[p_B(y_A),y_A])>2\delta$. Then by Lemma \ref{lem:quadrilateral-shape}, there exist $t \in [p,q]$, $t_1 \in [p,p_B(y_A)]$ and $t_2 \in [y_A,q]$ such that $d(t,t_1) \leq 2\delta$ and $d(t,t_2) \leq 2\delta$. So $d(t_1,t_2)\leq 4\delta<D_\delta$. But recall that $p \in L_B$ and $q \in [y_A,A^+)$ so this contradicts the definition of $y_A$. So we deduce the existence of $r \in [p,q]$ such that $d(r,[p_B(y_A),y_A])\leq 2\delta$. Let $s \in [p_B(y_A),y_A]$ such that 
			\begin{equation} \label{eq:d(r,s)}
			d(r,s)\leq 2\delta.
			\end{equation}
			In addition, observe that since $s \in [p_B(y_A),y_A]$, Lemma \ref{lem:projection-new} ensures that $p_B(y_A)$ is a projection of $s$ on $L_B$. Then, using Lemma \ref{lem:bound-Gromov-product-right-angle}, we have :
			\begin{equation}\label{eq:right-angle}
				d(p,s) \geq d(p,p_B(y_A))+d(p_B(y_A),s)-8\delta.
			\end{equation}
			We can now derive the following inequalities :
			\begin{align*}
				d(p,q) & =d(p,r)+d(r,q) \text{ since $r \in [p,q]$ }\\
				& \geq d(p,s)+d(s,q) - 4\delta \text{ by the triangle inequality and \eqref{eq:d(r,s)} } \\ 
				& \geq d(p,p_B(y_A))+d(p_B(y_A),s)+d(s,q)-12\delta \text{ by \eqref{eq:right-angle}} \\
				& \geq d(p,p_B(y_A))+d(p_B(y_A),q)-12\delta \text{ by the triangle inequality} \\
				& \geq d(p,x)+d(x,p_B(y_A))+d(p_B(y_A),q)-12\delta \text{ because $x \in [p,p_B(y_A)]$} \\
				& \geq d(p,x)+d(x,q)-12\delta \text{ by the triangle inequality.}
			\end{align*}
		\end{proof}
		
	Our goal is now to obtain an analogous lemma to Lemma \ref{lem:ineg q-loin}, but in the case where $q$ is no longer beyond $y_A$ on the axis. For this purpose, we will now distinguish between two cases. Consider the quadrilateral $[x_A,y_A] \cup [y_A,p_B(y_A)]\cup [p_B(y_A),p_B(x_A)] \cup [p_B(x_A),x_A]$. Lemma \ref{lem:quadri-thin} ensures that we have the following alternative : 
	\begin{enumerate}[label=(Alt\arabic*)]
	\item \label{case:quadri-close}There exists $x \in [p_B(x_A),p_B(y_A)]$ and $y \in [x_A,y_A]$ such that $d(x,y) \leq 2\delta$.
	\item \label{case:quadri-loin} For all $x \in  [p_B(x_A),p_B(y_A)]$, $d(x,[x_A,y_A])>2\delta$. 
	\end{enumerate}
	
	Lemma \ref{lem:gromov-long} deals with case \ref{case:quadri-close}. 
	
	For $x \in L_B$, we denote respectively by $(B^-,x]$ and $[x,B^+)$ the two half-geodesics contained in $L_B$ from $x$ to $B^-$ and from $x$ to $B^+$. For $y \in [x_A,y_A]$, we denote respectively by $[x_A,y]$ and $[y,y_A]$ the two geodesic segments contained in $[x_A,y_A]$ with endpoints $x_A$ and $y$,  and $y$ and $y_A$. 
	
	\begin{Lemma}\label{lem:gromov-long} Assume that $x_B,p_B(x_A),p_B(y_A)$ and $y_B$ are aligned on this order on $L_B$. Assume moreover that condition \ref{case:quadri-close} is true and let $x \in [p_B(x_A),p_B(y_A)]$ and $y \in [x_A,y_A]$ be such that $d(x,y) \leq 2\delta$. 
	If $q \in [y,y_A]$ and $p \in (B^-,x]$, or if $q \in [x_A,y]$ and $p \in [x,B^+)$, then we have:
		\begin{equation}
			d(p,q) \geq d(p,x)+d(x,q)-12\delta. 
		\end{equation}
	\end{Lemma}
	
	\begin{proof} Let us assume that $q \in [y,y_A]$ and $p \in (B^-,x]$. The case $q \in [x_A,y]$ and $p \in [x,B^+)$ is similar changing $(A,B)$ to $(A^{-1},B^{-1})$. 
		
		Let $[x,y]$ be a geodesic segment between $x$ and $y$, $[x,p_B(y_A)]$ be the geodesic segment between $x$ and $p_B(y_A)$ contained in $L_B$. By considering the quadrilateral $[x,y]\cup [y,y_A] \cup [y_A,p_B(y_A)]\cup [x,p_B(y_A)]$ and applying Lemma \ref{lem:quadri-thin}, we deduce the existence of $q' \in [x,y] \cup [x,p_B(y_A)] \cup [y_A,p_B(y_A)] $ such that $d(q,q') \leq 2\delta$. Then :
		\begin{itemize}
			\item If $q' \in [x,y]$, then $d(x,q') \leq 2\delta$ and so $d(q,x) \leq 4\delta$.  Therefore :
			\begin{align*}
				d(p,q) & \geq d(p,x) - 4\delta \\
				& \geq d(p,x) +d(x,q)-d(x,q)-4\delta \\
				& \geq d(p,x) + d(x,q)-8 \delta. 
			\end{align*}
			\item If $q' \in [x,p_B(y_A)]$, then
			\begin{align*}
				d(p,q)& \geq d(p,q') -2\delta \\
				& \geq d(p,x)+d(x,q')-2\delta \text{ since $x \in [p,q']$} \\ 
				& \geq d(p,x)+d(x,q)-4\delta.
			\end{align*}
			\item If $q' \in [y_A,p_B(y_A)]$, then by Lemma \ref{lem:projection-new}, the point $p_B(y_A)$ is a projection of $q'$ on $L_B$ and so by Lemma \ref{lem:bound-Gromov-product-right-angle}, $d(p,q') \geq d(p,p_B(y_A))+d(p_B(y_A),q')-8\delta $ and therefore : 
			\begin{align*}
				d(p,q) & \geq d(p,q')-2\delta \\
				& \geq d(p,p_B(y_A))+d(p_B(y_A),q')-10\delta \\
				& \geq d(p,x)+d(x,p_B(y_A))+d(p_B(y_A),q)-12\delta \text{ since $x \in [p,p_B(y_A)]$ and $d(q,q')\leq 2\delta$}\\
				& \geq d(p,x)+d(x,q)-12\delta \text{ using the triangle inequality.}
			\end{align*}
		\end{itemize}
	\end{proof}
	
	Now we adress the case \ref{case:quadri-loin}. 
	
	\begin{Lemma} \label{lem:gromov-haut-new}
	Assume that $x_B,p_B(x_A),p_B(y_A)$ and $y_B$ are aligned on this order on $L_B$. Assume moreover that condition \ref{case:quadri-loin} is true and let $x =p_B(y_A)$.
	If $q \in [x_A,y_A]$ and $p \in [p_B(y_A),B^+)$, then:
	\begin{equation}
		d(p,q) \geq d(p,x)+d(x,q)-12\delta. 
	\end{equation}
	\end{Lemma}
	
	\begin{proof}Consider the quadrilateral made of the four segments $[x_A,y_A]$, $[y_A,p_B(y_A)]$, $[p_B(x_A),p_B(y_A)]$ and $[x_A,p_B(x_A)]$ is $2\delta$-thin by Lemma \ref{lem:quadri-thin}, and $p \in [x_A,y_A]$, so condition \ref{case:quadri-loin} implies that there exists $p' \in [x_A,p_B(x_A)]\cup [y_A,p_B(y_A)]$ such that $d(p,p')\leq 2\delta$. Let us distinguish between the two cases :
	\begin{itemize}
		\item If $p' \in [y_A,p_B(y_A)]$, then :
		\begin{align*}
			d(p,q) & \geq d(p',q) -2\delta  \\ 
			& \geq d(p',p_B(y_A))+d(p_B(y_A),q)-8\delta-2\delta \quad \text{ using Lemma \ref{lem:projection-new} and Lemma \ref{lem:bound-Gromov-product-right-angle}}  \\ 
			& \geq d(p,p_B(y_A))+d(p_B(y_A),q)-12\delta 
		\end{align*}
		\item If $p' \in [x_A,p_B(x_A)]$, then :
		\begin{align*}
			d(p,q) & \geq d(p',q) -2 \delta \\
			& \geq d(p',p_B(x_A))+d(p_B(x_A),q)-8\delta-2\delta \quad \text{ using Lemma \ref{lem:projection-new} and Lemma \ref{lem:bound-Gromov-product-right-angle}} \\
			& \geq d(p,p_B(x_A))+d(p_B(x_A),p_B(y_A))+d(p_B(y_A),q)-12\delta \\ \intertext{ since $p_B(x_A)$, $p_B(y_A)$ and $q$ are aligned on this order on $L_B$ } 
			& \geq d(p,p_B(y_A))+d(p_B(y_A),q)-12\delta \quad \text{ by the triangle inequality. }
		\end{align*}
	\end{itemize}
	\end{proof}
	
	Combining Lemma \ref{lem:ineg q-loin}, Lemma \ref{lem:gromov-long} and Lemma \ref{lem:gromov-haut-new}, we proved :
	
	\begin{Corollary}\label{cor:gromov-prod-gen}
		Assume that $x_B,p_B(x_A),p_B(y_A)$ and $y_B$ are aligned on this order on $L_B$. 
		There exist $x \in [p_B(x_A),p_B(y_A)]$ and $y \in [x_A,y_A]$ such that $d(x,y) \leq D_\delta$ and such that, if $p \in (B^-,x]$ and $q \in [y,A^+)$, or if $p \in [x,B^+)$ and $q \in (A^-,y]$, we have :
		\begin{equation}\label{eq:d(p,q)final}
			d(p,q) \geq d(p,x)+d(x,q)-12\delta. 
		\end{equation}
	\end{Corollary}
	
	\begin{proof}
		We distinguish two cases :
		\begin{itemize}
			\item If \ref{case:quadri-close} is true, then take $x \in [p_B(x_A),p_B(y_A)]$ and $y \in [x_A,y_A]$ such that $d(x,y)\leq 2\delta$, and apply Lemma \ref{lem:ineg q-loin} and Lemma \ref{lem:gromov-long} to have the inequality \eqref{eq:d(p,q)final} in all the cases. 
			\item If \ref{case:quadri-loin} is true, then take any $x=p_B(y_A)$ and $y=y_A$. Then $d(x,y) \leq D_\delta$ by Lemma \ref{lem:x_A<x_B}, and apply Lemma \ref{lem:ineg q-loin} and Lemma \ref{lem:gromov-haut-new} to have the inequality \eqref{eq:d(p,q)final} in all the cases. 
		\end{itemize}
	\end{proof}

		We now introduce four points, $p_+$ and $p_-$ on $L_A$ and $q_+$ and $q_-$ on $L_B$ which will approximate the points $Bx, B^{-1}x, Ax$ and $A^{-1}x$.
		\begin{Lemma} \label{lem:approx-proj}
		Assume that $l_S(A) > 6\delta$ and $l_S(B)>6\delta$. Let $x$ and $y$ be given by Corollary \ref{cor:gromov-prod-gen}. Then~:  
		\begin{enumerate}[label=(\alph*)]
			\item \label{p+} There exists $p_+ \in [x,B^+)$ such that $d(Bx,p_+) \leq 2\delta$. 
			\item \label{p-}There exists $p_- \in (B^-,x]$ such that $d(B^{-1}x,p_-)\leq 2\delta$. 
			\item \label{q+} There exists $q_+ \in [y,A^+)$ such that $d(Ax,q_+) \leq D_\delta+2\delta$. 
			\item \label{q-} There exists $q_- \in (A^-,y]$ such that $d(A^{-1}x,q_-) \leq D_\delta+2\delta$.
		\end{enumerate}
	\end{Lemma}
	
	\begin{proof} 
		\begin{enumerate}[label=(\alph*)]
			
			\item 
			Let $p_+$ be a projection of $Bx$ on $L_B$. The element $x$ belongs to $L_B$ so by Corollary \ref{cor:axe-quasi-inv} $d(Bx,p_+) \leq 2\delta$. Moreover, Lemma \ref{lem:projection-Bx} ensures that $p_+ \in [x,B^+)$ because $l_S(B) > 6\delta$.
			
			\item The proof is the same as for \ref{p+}, changing $B$ to $B^{-1}$. 
			
			\item Let $q_+$ be a projection of $Ay$ on $L_A$. The element $y$ belongs to $L_A$ so by Corollary \ref{cor:axe-quasi-inv} $d(Ay,q_+) \leq 2\delta$. Moreover, Lemma \ref{lem:projection-Bx} ensures that $q_+ \in [y,A^+)$ because $l_S(A) > 6\delta$. Then 
			\begin{align*}
				d(Ax,q_+) & \leq d(Ax,Ay)+d(Ay,q_+)=d(x,y)+d(Ay,q_+) \\
				& \leq D_\delta+2\delta \text{ by Corollary \ref{cor:gromov-prod-gen}}
			\end{align*}
			\item The proof is the same as for \ref{q+} changing $A$ to $A^{-1}$. 
			
		\end{enumerate}
	\end{proof}
	
		We have now settled and describe the general picture of the setting, and we will then define a sequence of point in $\mathcal{X}$ that will be a quasi-geodesic approximating the axis of the hyperbolic isometry $AB$, as in section \ref{sec:axes-far}. \\ 
	
	We fix $x \in [p_B(x_A),p_B(y_A)]$ and $y \in [x_A,y_A]$ satisfying Corollary \ref{cor:gromov-prod-gen}. We consider the sequence $(x_i)_{0 \leq i \leq 2n}$ defined by : 

	\begin{equation*}
		\left\{
		\begin{array}{rll}
			x_{2j} & =(AB)^jx &\text{ for all } 0 \leq j \leq n \\
			x_{2j+1} & = (AB)^jAx & \text{ for all } 0 \leq j \leq n-1
		\end{array}
		\right. 
	\end{equation*}

	\begin{Lemma} \label{lem:gromov-prod-cas-proche} 
		Assume that $l_S(A)>6\delta$ and $l_S(B)>6\delta$ and that $x_B,p_B(x_A),p_B(y_A)$ and $y_B$ are aligned on this order on $L_B$. 
		For all $0 \leq i \leq 2n-1$, the Gromov-product~$(x_{i-1}|x_{i+1})_{x_i}$ is bounded by $D_\delta+10\delta= 54\delta$. 
	\end{Lemma}

	\begin{proof}
	Let us first compute these Gromov-products.
	\begin{itemize}
	\item Suppose that $i=2j$, with $0 \leq j \leq n$, then : \\ 
	$(x_{i-1}|x_{i+1})_{x_i} =(x_{2j-1}|x_{2j+1})_{x_{2j}}  
		= ((AB)^{j-1}Ax|(AB)^jAx)_{(AB)^jx} 
		= (B^{-1}x | Ax)_{x}$.
 	\item Suppose that $i=2j+1$, with $0 \leq j \leq n-1$, then : \\ 
	$(x_{i-1}|x_{i+1})_{x_i}  =(x_{2j}|x_{2j+2})_{x_{2j+1}}  
		= ((AB)^{j}x|(AB)^{j+1}x)_{(AB)^jAx} 
		 = (A^{-1}x | Bx)_{x}$.
		\end{itemize}
	Therefore it is sufficient to bound $(B^{-1}x | Ax)_{x}$ and $(A^{-1}x|Bx)_{x}$.
	
	We introduce the four points $p_+$, $p_-$, $q_+$ and $q_-$ given by Lemma \ref{lem:approx-proj}. We first bound the Gromov-product $(B^{-1}x|Ax)_x $:
				\begin{align*}
					(B^{-1}x|Ax)_x & =\frac{1}{2}(d(x,B^{-1}x)+d(x,Ax)-d(B^{-1}x,Ax)) \\
					& \leq \frac{1}{2}(d(x,p_-)+2\delta+d(x,q_+)+D_\delta+2\delta-d(p_-,q_+)+2\delta +D_\delta+2\delta) \text{ by Lemma \ref{lem:approx-proj} \ref{p-} and \ref{q+}}, \\
					& \leq \frac{1}{2}(d(x,p_-)+d(x,q_+)-d(p_-,q_+))+D_\delta+4\delta \\
					& \leq \frac{1}{2}12\delta+D_\delta+4\delta \text{ by Corollary \ref{cor:gromov-prod-gen}},\\ 
					& \leq D_\delta+10\delta.
				\end{align*}
The proof of the bound on the Gromov product $(A^{-1}x|Bx)_{x}$ is the same exchanging the role of $A$ and $B$ and using the points $q_-$ and $p_+$ instead of $p_-$ and $q_+$ and Lemma \ref{lem:approx-proj} \ref{p+} and \ref{q-} instead of Lemma \ref{lem:approx-proj} \ref{p-} and \ref{q+}. 
		\end{proof}

	\begin{Lemma} \label{lem:ineg crucial cas proche}
		Assume that $B^-,p_B(x_A),p_B(y_A)$ and $B^+$ are aligned on this order on $L_B$.
		Assume moreover that $d(L_A,L_B)\leq 44\delta$, $l_S(A) > 188\delta$, $l_S(B) > 188\delta$, then, 
		\begin{equation}
			l_S(AB) \geq l_S(A)+l_S(B)-256\delta.
		\end{equation}
	\end{Lemma}
	
	\begin{proof}
		We now use the local-global Lemma \ref{lem:local-global} with $A=D_\delta+10\delta=54\delta$. Note that $L=3A+26\delta=188\delta$, and $D=A+10\delta=64\delta$. 
		
		\begin{itemize}
			\item Suppose that $i=2j$, with $0 \leq j \leq n-1$, then  \\ 
			$d(x_i,x_{i+1})=d(x_{2j},x_{2j+1})=d((AB)^jx,(AB)^jAx)=d(x,Ax)\geq l_S(A) > L$.
			\item Suppose that $i=2j+1$, with $0 \leq j \leq n-1$, then \\
			$d(x_i,x_{i+1})=d(x_{2j+1},x_{2j+2})=d((AB)^jAx,(AB)^{j+1}x)=d(x,Bx)\geq l_S(B)>L$.
		\end{itemize}
		Then our assumptions on $l_S(A)$ and $l_S(B)$ and Lemma \ref{lem:gromov-prod-cas-proche} ensure that our sequence of points $(x_i)_{0\leq i\leq 2n}$ satisfies the hypothesis of the local-global lemma \ref{lem:local-global}, and thus we have :
		\begin{align*}
			d(x_0,x_{2n}) & \geq d(x_0,x_1)+\dots+d(x_{2n-1},x_{2n})-2D(2n-1) \\ 
			d(x,(AB)^nx) & \geq nd(x,Ax)+nd(x,Bx)-128\delta(2n-1) 
		\end{align*}
		To obtain the inequality on $l_S(AB)$, divide the previous one by $n$ then take the limit, and use $\frac{1}{n}d(x,(AB)^nx) \underset{n \to \infty}{\longrightarrow} l_S(AB)$, $d(x,Ax)\geq l_S(A)$ and $d(x,Bx)\geq l_S(B)$. 
	\end{proof}
	
	The inequality on $l_S(AB)$ or $l_S(AB^{-1})$ has now been proven in both cases, when the axes of the isometries $A$ and $B$ are far (in section \ref{sec:axes-far}) and when they are close for some moment (in section \ref{sec-axes-close}), hence Proposition \ref{prop:ineq-l(AB)}.
	
	\begin{proof}[Proof of Proposition \ref{prop:ineq-l(AB)}]
		Consider $L_A$ and $L_B$ be two geodesic respectively from $A^-$ to $A^+$ and from $B^-$ to $B^+$.
		There is two cases :
		\begin{itemize}
			\item If $d(L_A,L_B) > 44\delta$, then Lemma \ref{lem:ineg crucial cas loin}, gives the inequality both for $l_S(AB)$ and $l_S(AB^{-1})$.
			\item If $d(L_A,L_B) \leq 44 \delta$, then up to changing $B$ into $B^{-1}$, we can assume that $x_B,p_B(x_A),p_B(y_A),y_B$ are aligned on this order on $L_B$ and so Lemma \ref{lem:ineg crucial cas proche} gives the inequality. 
		\end{itemize}
	\end{proof}

	\begin{Remark} \label{rem:ineg-CAT(-1)}
		If the space $\mathcal{X}$ satisfies, moreover, the hypothesis \ref{hyp:unique-visibility} defined in the introduction (see Remark \ref{rem:thm} \ref{rem:thm:hyp}), the constant $C_\delta$ in Proposition \ref{prop:ineq-l(AB)} can be taken to be $216\delta$.
	\end{Remark}

	\section{Growth of the length of the Farey neighbours of an element} \label{sec:growth Farey neighbours}

	In this section, we are interested in the growth of the quantity $l_S(A^nB)$ with $n \in \N$, for $A$ and $B$ two hyperbolic isometries. We first show that this growth is linear in $n$, and give explicit expressions for the constants. 
	
	\begin{Proposition} \label{prop:AnB}
		Let $A$ be a hyperbolic isometry and $B$ and isometry such that $B(A^+) \neq A^-$. Let $L_A$ be a geodesic between $A^+$ and $A^-$ and $o \in L_A$. Let $(A^-,B(A^+))$ be a geodesic between $B(A^+)$ and $A^-$ and $p : \mathcal{X} \to (A^-,B(A^+))$ a projection on $(A^-,B(A^+))$.
		
		Then, there exists $N=N(A,B,o,p) \in \N$ and $k=k(A,B,o,p) \geq0$ such that, for all $n \geq N$ :
		\begin{equation*}
			l_S(BA^n)\geq l_S(A)n -k.
		\end{equation*}
		Moreover, we have the following formulas for $N(A,B,o,p)$ and $k(A,B,o,p)$ :
	\begin{align}
	k(A,B,o,p) & =3d(Bo,o)+2d(p(o),o)+36\delta \label{eq:k} \\ 
	N(A,B,o,p)& =\frac{1}{l_S(A)} \max \{ 4d(Bo,o)+3d(p(o),o)+51\delta, \, 2d(Bo,o)+4d(p(o),o)+5\delta \}.  \label{eq:N} 
   \end{align}
	\end{Proposition}
	
	\begin{proof} Note that the assumption $B(A^+)\neq A^-$ ensures that there exists a geodesic $(A^-,B(A^+))$ with endpoints $A^-$ and $B(A^+)$. Consider  $p_A : \mathcal{X} \to L_A$ a projection on $L_A$. \\
		
	Our first goal is to find $N_1 \in \N$ such that for all $n > N_1$, the projection $p(p_A(A^{-n}o))$ of $A^{-n}o$ on $(A^-,B(A^+))$ satisfies :
		\begin{enumerate}[label=(\alph*)]
			\item \label{cond:bound-d(A^{-n}o,p(A^{-n}o))} $d(A^{-n}o,p(p_A(A^{-n}o))) \leq 4\delta$,
			\item \label{cond:proj-A^-no} $p(p_A(A^{-n}o)) \in (A^-,p(o)]$, where $(A^-,p(o)]$ is the geodesic ray between $p(o)$ and $A^-$ contained in $(A^-,B(A^+))$. 
		\end{enumerate}
	First note that Corollary \ref{cor:axe-quasi-inv} ensures that for all $n \in \N$,  $d(A^{-n}o,p_A(A^{-n}o)) \leq 2\delta$.   
	Now let us assume that $n >\frac{6\delta}{l_S(A)}$, then $l_S(A^{-n})=nl_S(A)>6\delta$, and therefore by Lemma \ref{lem:projection-Bx}, $p_A(A^{-n}o) \in [o,A^-)$ (with $[o,A^-)$ the geodesic ray from $o$ to $A^-$ contained in $L_A$). In order to prove conditions  \ref{cond:bound-d(A^{-n}o,p(A^{-n}o))} and \ref{cond:proj-A^-no}, we will apply Lemma \ref{lem:pos-proj} to $l=(A^-,B(A^+))$, $y_0=p(o) \in (A^-,B(A^+))$, $l_\infty=A^-$, $x_0=o$ and $x=p_A(A^{-n}o) \in [o,A^-)$. We have :
	\begin{align*}
		d(p_A(A^{-n}o),p(o)) & \geq d(A^{-n}o,o)-d(o,p(o))-2\delta \\
		& \geq nl_S(A)-d(o,p(o))-2\delta
	\end{align*}
	so for all $n> \frac{2}{l_S(A)}(d(o,p(o)+2\delta)$, we have $d(p_A(A^{-n}o),p(o)) > d(o,p(o)) +2\delta$. In addition, for all $n > \frac{1}{l_S(A)}(d(o,p(o))+8\delta)$, we also have $d(p_A(A^{-n}o),p(o)) >6\delta$. Therefore, Lemma \ref{lem:pos-proj} ensures that \ref{cond:bound-d(A^{-n}o,p(A^{-n}o))} and \ref{cond:proj-A^-no} are satisfied whenever $n>N_1:=\frac{1}{l_S(A)}\max\{2d(p(o),o)+4\delta,d(o,p(o))+8\delta\}$.
	
	Our second goal is to find, similarly, $N_2 \in \N$ such that for all $n > N_2$, the projection $p(Bp_A(A^no))$ of $Bp_A(A^no)$ on $(A^-,B(A^+))$ satisfies :
	\begin{enumerate}[resume,label=(\alph*)]
		\item \label{cond:d(BA^no,p(Bp_L_A(A^no)))} $d(BA^no,p(Bp_A(A^no)))\leq 4\delta$,
		\item \label{cond:proj-BA^no} $p(Bp_A(A^no)) \in [p(o),B(A^+))$, where $[p(o),B(A^+))$ is the geodesic ray between $p(o)$ and $B(A^+)$ contained in $(A^-,B(A^+))$. 
	\end{enumerate}
	On the one hand, we have $d(A^no,p_A(A^no)) \leq 2\delta$ by Corollary \ref{cor:axe-quasi-inv} and Lemma \ref{lem:projection-Bx} ensures that for $n >\frac{6\delta}{l_S(A)}$, $p_A(A^no) \in [o,A^+)$. Applying the isometry $B$, we obtain $Bp_A(A^no) \in [B(o),B(A^+))$, where $[B(o),B(A^+))$ is contained in $B(L_A)$. Now we apply Lemma \ref{lem:pos-proj} to $l=(A^-,B(A^+))$, $y_0=p(o)$, $l_\infty = B(A^+)$, $x_0=Bo$ and $x=Bp_A(A^no)$. Let us bound $d(Bp_A(A^no),p(o))$ from below :
	\begin{align*}
		d(Bp_A(A^no),p(o)) & \geq d(p_A(A^no),p(o))-d(Bp(o),p(o)) \\
		& \geq d(A^no,o)-d(p(o),o)-d(Bp(o),p(o))-2\delta \\
		& \geq nl_S(A)-3d(p(o),o)-d(Bo,o)-2\delta, \\
		\text{ Also,} \qquad \qquad
		d(Bo,p(o)) & \leq d(Bo,o)+d(p(o),o).
	\end{align*}
		Then, for all $n>N_2:=\frac{1}{l_S(A)}\max\{ 3d(p(o),o)+d(Bo,o)+8\delta,4d(p(o),o)+2d(Bo,o)+4\delta\}$, the required inequalities of Lemma \ref{lem:pos-proj} are satisfied and therefore we have the inequality $d(p(Bp_A(A^no)),Bp_A(A^no))\leq 2\delta$ and condition \ref{cond:proj-BA^no} : $p(Bp_A(A^no)) \in [p(o),B(A^+))$. It only remains to check condition \ref{cond:d(BA^no,p(Bp_L_A(A^no)))} :
	\begin{align*}
		d(BA^no,p(Bp_A(A^no))) & \leq d(BA^no,Bp_A(A^no))+d(Bp_A(A^no),p(Bp_A(A^no))) \\
		& \leq d(A^no,p_A(A^no))+2\delta \\
		& \leq 4\delta.
	\end{align*}
	
	Note that $N_2 \geq N_1$ and let us assume that $n> N_2$. Then conditions \ref{cond:bound-d(A^{-n}o,p(A^{-n}o))}, \ref{cond:proj-A^-no}, \ref{cond:d(BA^no,p(Bp_L_A(A^no)))} and \ref{cond:proj-BA^no} are satisfied, and thus, using \ref{cond:proj-A^-no} and \ref{cond:proj-BA^no}, we can write :
	\begin{align}
		d(p(p_A(A^{-n}o)),p(Bp_A(A^no))) & = 	d(p(p_A(A^{-n}o)),p(o)) +	d(p(o),p(Bp_A(A^no))) \nonumber \\
		& \geq d(A^{-n}o,o)-d(p(o),o)-4\delta+d(o,BA^no)-d(o,p(o))-4\delta \label{eq:d(p(p_A(A^{-n}o)),p(Bp_A(A^no)))}
		\end{align}
	using the triangle inequality and conditions \ref{cond:bound-d(A^{-n}o,p(A^{-n}o))} and \ref{cond:d(BA^no,p(Bp_L_A(A^no)))}.

	Now, let us bound the Gromov-product $((BA^n)^{-1}o \, | \, BA^no)_o$ :
	\begin{align}
		((BA^n)^{-1}o \, | \, BA^no)_o & = \frac{1}{2}(d(o,(BA^n)^{-1}o)+d(o,BA^no))-d((BA^n)^{-1}o,BA^no) ) \nonumber \\
		& \leq \frac{1}{2}(d(o,A^no)+d(o,BA^no)-d(A^{-n}o,BA^no))+2d(o,Bo) \nonumber  \\
		& \leq \frac{1}{2}(d(o,A^no)+d(o,BA^no)-d(p(p_A(A^{-n}o)),p(Bp_A(A^no)))+8\delta +2d(o,Bo))  \nonumber \\ 
		\intertext{using conditions \ref{cond:bound-d(A^{-n}o,p(A^{-n}o))} and \ref{cond:d(BA^no,p(Bp_L_A(A^no)))}, and then by \eqref{eq:d(p(p_A(A^{-n}o)),p(Bp_A(A^no)))} :}
		& \leq \frac{1}{2}(d(o,A^no)+d(o,BA^no)-d(A^no,o)-d(o,BA^no)+2d(p(o),o)+16\delta+2d(o,Bo)) \nonumber \\
		& \leq d(p(o),o)+d(o,Bo)+8\delta. \label{eq:grom-prod-BA}
	\end{align}
	
	Let $m \geq 2$ be an integer et define the following sequence of points in $\mathcal{X}$ :  $x_i=(BA^n)^io$, for $i=0, \dots, m$. Note that for all $1 \leq i \leq m-1$, $(x_{i-1} \, | \, x_{i+1})_{x_i}=((BA^n)^{-1}o \, | \, BA^no)_o$ and thus, by \eqref{eq:grom-prod-BA}, we proved :
	\begin{equation}
		(x_{i-1} \, | \, x_{i+1})_{x_i} \leq d(p(o),o)+d(o,Bo)+8\delta.
	\end{equation}
	Now we apply the local-global Lemma \ref{lem:local-global} with $A=d(p(o),o)+d(o,Bo)+8\delta$. Then $D=A+10\delta=d(p(o),o)+d(o,Bo)+18\delta$ and $L=3A+26\delta=3d(p(o),o)+3d(o,Bo)+50\delta$. Moreover,
	\begin{align*}
		d(x_i,x_{i+1})&=d((BA^n)^io,(BA^n)^{i+1}o) =d(o,BA^no) \\
		& \geq d(o,A^no)-d(o,Bo) \geq nl_S(A)-d(o,Bo). 
	\end{align*}
	Then, if $n > N_3:=\frac{1}{l_S(A)} (d(o,Bo)+L)$, we deduce that for all $0\leq i \leq m-1$, we have $d(x_i,x_{i+1}) >L$. In this case, the conclusion of the Lemma applies and we have :
	\begin{align*}
		d(o,(BA^n)^mo)=d(x_0,x_m) & \geq \Sigma_{i=1}^{m} d(x_{i-1},x_i) -2D(m-1) 
		\geq md(o,BA^no)-2D(m-1) \\
		& \geq m(nl_S(A)-d(o,Bo))-2D(m-1).
	\end{align*}
	Dividing by $m$ and letting $m \to +\infty$, we obtain :
	\begin{equation*}
		l_S(BA^n) \geq nl_S(A)-d(o,Bo)-2D.
	\end{equation*}
	Therefore, we proved the Proposition with $N(A,B,o,p)=\max\{N_2+\delta,N_3+\delta\}$ and $k(A,B,o,p)=d(o,Bo)+2D$.
	\end{proof}
	
	We now endow $\mathrm{Isom}(\mathcal{X})$ with the compact-open topology. We would like to control the constants and $N(A,B,o,p)$ and $k(A,B,o,p)$ obtained in Proposition \ref{prop:AnB} under small deformations of $A$ and $B$. Unfortunately, $k$ and $N$ may not vary continuously with $A$ and $B$ because the projection $p$ could not be continuous. However, we prove :
	  
	\begin{Lemma} \label{lem:nearly-continuity}
		 Fix $K>0$. 
		Let $A$ by a hyperbolic isometry, $B$ an isometry such that $B(A^+)\neq A^-$, $L_A$ a geodesic between the fixpoints $A^+$ and $A^-$ of $A$ and $o \in L_A$. Let $(A^-,B(A^+))$ be a geodesic between $A^-$ and $B(A^+)$ and $p : \mathcal{X} \to (A^-,B(A^+)$ a projection on $(A^-,B(A^+)$. Denote by $N:=N(A,B,o,p)$ and $k:=k(A,B,o,p)$ the constants defined in Proposition \ref{prop:AnB}. 
		
		There exists a neighborhood $\mathcal{U_A}$ of $A$ and a neighborhood $\mathcal{U_B}$ of $B$ in $\mathrm{Isom}(\mathcal{X})$ such that for all $A' \in \mathcal{U}_A$ and $B' \in \mathcal{U}_B$ the following holds. 
		 $A'$ is a hyperbolic isometry, $B'$ is an isometry satisfying $B'(A'^+) \neq A'^-$, and there exist a geodesic $L_{A'}$ between the fixpoints $A'^+$ and $A'^-$ of $A'$, a point $o' \in L_{A'}$, a geodesic $(A'^-,B(A^+))$ between $A'^-$ and $B'(A'^+)$, a projection $p'$ on $(A'^-,B'(A'^+))$, and a constant $C(l_S(A),\delta) \geq 0$ depending only on $l_S(A)$ and $\delta$, such that, denoting $N':=N(A',B',o',p')$ and $k':=k(A',B',o',p')$ the constants defined in Proposition \ref{prop:AnB} corresponding to these data :
		\begin{equation}\label{eq:r>N}
			\left| \max \left\{ {N,\frac{K+k}{l_S(A)}}\right\} -\max \left\{ {N',\frac{K+k'}{l_S(A')}}\right\} \right| \leq C(l_S(A),\delta).
		\end{equation}
	\end{Lemma}
	
	\begin{proof}
		First recall that the stable norm $l_S$ is continuous (Theorem \ref{thm:l_S-continuous}), and that the positivity of the stable norm caracterises hyperbolicity. Then small deformations of $A$ preserves hyperbolicity. Moreover, since $B(A^+)\neq A^-$, small deformations of $A$ and $B$ also imply that $B'(A'^+)$ and $A'^-$ are distinct. We can choose $A'^-$ sufficiently close to $A^-$ and $A'^+$ sufficiently close to $A^+$ to ensure that any geodesic $L_{A'}$ between $A'^+$ and $A'^-$ passes at a bounded distance from $o \in L_A$, which means that there exists $o' \in L_{A'}$ and a constant $c_\delta$ only depending on $\delta$ so that $d(o,o') \leq c_\delta$. Now denote $p' : \mathcal{X} \to (A'^-,B'(A'^+))$ a projection on $(A'^-,B'(A'^+))$. Similarly, when $A'^-$ is sufficiently close to $A^-$ and $B'(A'^+)$ is sufficiently close to $B(A^+)$, we can ensure that $p(o) \in L_A$ is at bounded distance from $(A'^-,B'(A'^+))$, that is $d(p(o),p'(p(o))) \leq c_\delta$. Moreover, we can further assume that $p'(o)$ is close to the geodesic $(A^-,B(A^+))$, that is there exists a constant $c'_\delta$ only depending on $\delta$ such that $d(p'(o),p(p'(o))) \leq c'_\delta$. From this we can now deduce the existence of a constant $d_\delta$, only depending on $\delta$, such that the distance $d(p(o),p'(o))$ is bounded by $d_\delta$. Finally, for $B'$ close enough to $B$, we also have $d(B'o,Bo) \leq \delta$.  
		
		Let us first bound $|d(p(o),o)-d(p'(o'),o')|$. On the one hand :
		\begin{align*}
			d(p'(o'),o') &  \leq d(p'(o'),p'(o))+d(p'(o),p(o))+d(p(o),o) \\
		& 	\leq d(o,o')+12\delta + d_\delta + d(p(o),o) \quad \text{ by \eqref{eq:proj-bound-3distance} of Lemma \ref{lem:proj-convex}} \\ 
		& \leq c_\delta + 12\delta +d_\delta +d(p(o),o)
		\end{align*}
		And on the other hand :
		\begin{align*}
			d(p(o),o) & \leq d(p(o),p'(o))+d(p'(o),p'(o'))+d(p'(o'),o')+d(o',o) \\
			& \leq d_\delta +d(o,o')+12\delta + d(p'(o'),o')+d(o',o) \text{ by \eqref{eq:proj-bound-3distance} of Lemma \ref{lem:proj-convex}} \\
			& \leq d_\delta +2c_\delta +12\delta +d(p'(o'),o'),
		\end{align*}
	    Then 
		\begin{equation} \label{bound:p(o)}
		|d(p(o),o)-d(p'(o'),o')| \leq d_\delta +2c_\delta +12\delta.
		\end{equation}
		We can also bound $|d(Bo,o)-d(B'o',o')|$ :
		\begin{equation}\label{bound:Bo}
			|d(Bo,o)-d(B'o',o')|  \leq |d(B'o',B'o)+d(B'o,Bo)+d(o,o') | \leq \delta +2c_\delta. 
		\end{equation}
		Now note that we have $\max \left\{N,\frac{K+k}{l_S(A)}\right\} = \frac{1}{l_S(A)} \max \{N_1,N_2,N_3 \} $,  with 
		\begin{align*}
			N_1 & = 4d(Bo,o)+3d(p(o),o)+51\delta, \\
			N_2 & = 2d(Bo,o)+4d(p(o),o)+5\delta, \\
			N_3 & = K+3d(Bo,o)+2d(p(o),o)+36\delta, 
		\end{align*}
	and define similarly $N'_1, N'_2$ and $N'_3$ using $B'$, $o'$ and $p'$. Because of the bounds \eqref{bound:p(o)} and \eqref{bound:Bo} we deduce the existence of three constants $c_{1,\delta}, c_{2,\delta}$ and $c_{3,\delta}$ only depending on $\delta$ such that for all $i=1,\dots,3$, $|N_i-N'_i| \leq c_{i,\delta}$, and therefore, we also have $\vert \max \{N_1,N_2,N_3\}-\max \{N'_1,N'_2,N'_3\}| \leq e_\delta:=\max  \{c_{1,\delta},c_{2,\delta},c_{3,\delta}\}$. Denote $M=\max \{N_1,N_2,N_3\}$ and $M'=\max \{N'_1,N'_2,N'_3\}$. We have $|M-M'|\leq e_\delta$. Now :
	\begin{align*}
		\left| \frac{M}{l_S(A)}-\frac{M'}{l_S(A')} \right|  & \leq  \frac{ \left| Ml_S(A')+Ml_S(A)-Ml_S(A)+M'l_S(A) \right|}{l_S(A)l_S(A')} \leq \frac{\left| M\right| \left| l_S(A')-l_S(A)\right| + \left|M-M'\right|l_S(A)}{l_S(A)l_S(A')}. \\ 
		\intertext{Let us finally impose that $A'$ is sufficiently close to $A$ so that $\displaystyle l_S(A') \geq \frac{l_S(A)}{2}$ and $\displaystyle |l_S(A')-l_S(A)| \leq \frac{\delta l_S(A)^2}{2|M|}$. Then we have :}
		\left| \frac{M}{l_S(A)}-\frac{M'}{l_S(A')} \right| & \leq \delta  + \frac{2|M-M'|}{l_S(A)} \leq \delta + \frac{2e_\delta}{l_S(A)}=:C(l_S(A),\delta), \text{ and the Lemma is proved.} \qedhere
	\end{align*}
	\end{proof}

	We now give a criteria to ensure the condition $B(A^+)\neq A^-$. We show that, on the contrary, if $B(A^+)=A^-$, for $n$ sufficiently large the elements of the form $A^nB$ have a bounded stable norm.

	\begin{Proposition} \label{prop:B(A+)neqA-explicite}
		Let $A$ be a hyperbolic isometry and $B$ be an isometry such that $B(A^+)=A^-$. Then there exists $N$ such that for all $n \geq N$,
		\begin{equation*}
			l_S(A^nB) \leq 30\delta.
		\end{equation*}
		Moreover, we have the following explicite expression for $N$ :
		 \begin{equation}\label{eq:N-lS bounded} 
		N =\frac{1}{l_S(A)} \max\{3d(Bo,o)+18\delta,5d(Bo,o)+10\delta \}.
		 \end{equation}
		 Here $o$ is any point on $L_A$, with $L_A$ being a geodesic between $A^+$ and $A^-$. 
	\end{Proposition}
	\begin{proof}
		
		Let us assume that $B(A^+)=A^-$. Choose a geodesic $L_A$ from $A^-$ to $A^+$ and $o \in L_A$. Consider $p : \mathcal{X} \to L_A$ a projection map on $L_A$.\\
		
		Let us first define $o_1 \in L_A$ such that, denoting $(A^-,o]$ the geodesic ray between $o$ and $A^-$ contained in $L_A$ :
		\begin{align}
			\text{For all } x \in [o_1,A^+), \qquad & d(Bx,p(Bx)) \leq 2\delta, \label{eq: d(Bx,p(Bx)} \\ 
			\text{and } \qquad & p(Bx) \in (A^-,o]. \label{eq:p(Bx)}
		\end{align}
		
		Consider $[o,A^+)$ the geodesic ray between $o$ and $A^+$ contained in $L_A$. Its image $[Bo,A^-)$ by $B$ is the geodesic ray between $Bo$ and $B(A^+)=A^-$ contained in $B(L_A)$. It is asymptotic to the geodesic ray $[o,A^-)$. Thus let us apply \ref{lem:pos-proj} to these two geodesic rays. 
        First note that when $x \in [o,A^+)$, its image $Bx \in [Bo,A^-)$ and 
		\begin{align*}
			d(Bx,o) \geq d(Bx,Bo)-d(Bo,o)=d(x,o)-d(Bo,o).
		\end{align*}
 		Therefore, whenever $d(x,o) > \max \{ d(Bo,o)+6\delta,2d(Bo,o)+2\delta\}$, Lemma \ref{lem:pos-proj} applies and ensures that $d(p(Bx),Bx) \leq 2\delta$ and $p(Bx) \in (A^-,o]$. Let $o_1 \in [o,A^+)$ be the point so that $d(o,o_1)= \max \{ d(Bo,o)+7\delta,2d(Bo,o)+3\delta\}$, then \eqref{eq: d(Bx,p(Bx)} and \eqref{eq:p(Bx)} hold for this choice of $o_1$. \\
		
		Let us now fix $n \geq N$, where $N$ is defined in \eqref{eq:N-lS bounded}, and let us verify that $A^-$, $p(A^{-n}o_1)$ and $p(Bo_1)$ are aligned in this order on $L_A$. First note that $l_S(A^{-n})=nl_S(A)>6\delta$, then by Lemma \ref{lem:projection-Bx}, we have $p(A^{-n}o_1) \in (A^-,o_1]$, with $(A^-,o_1]\subset L_A$. We can bound $d(p(A^{-n}o_1),o_1)$ from below  :
		\begin{align*}
			d(p(A^{-n}o_1),o_1) \geq d(A^{-n}o_1,o_1)-2\delta \geq ml_S(A)-2\delta \text{ using Corollary \ref{cor:axe-quasi-inv}}.
		\end{align*}
		Also, 
		\begin{align*}
			d(p(Bo_1),o_1)&  \leq d(Bo_1,o_1) +d(p(Bo_1),Bo_1) \leq d(Bo,o)+2d(o,o_1)+2\delta \text{ using \eqref{eq: d(Bx,p(Bx)}}, \\
			& \leq  d(Bo,o)+2\max\{d(Bo,o)+7\delta,2d(Bo,o)+3\delta \}+2\delta \text{ by definition of $o_1$}, \\
			& \leq \max\{3d(Bo,o)+16\delta,5d(Bo,o)+8\delta \}.
		\end{align*}
		With these inequalities we see that the definition of $N$ and the hypothesis $n \geq N$ ensure that $d(p(A^{-n}o_1),o_1) \geq d(p(Bo_1),o_1)$ and therefore, since $p(A^{-n}o_1)$ and $p(Bo_1)$ both belong to $(A^-,o_1]$, we deduce that $A^-$, $p(A^{-m}o_1)$ and $p(Bo_1)$ are aligned in this order on $L_A$. \\
		
	 	Now define the maps
		\[
		\begin{tabular}{cc}
			$\begin{array}{rcl}
				f : L_A & \longrightarrow & \mathcal{X} \\
				x & \mapsto & Bx
			\end{array}$ 
			& 
			$\begin{array}{rcl}
			\text{ and } \qquad	g : L_A & \longrightarrow & \mathcal{X} \\
				x & \mapsto & A^{-n}x
			\end{array}$
		\end{tabular}
		\]

		The maps $f$ and $g$ are continuous and $f(x) \underset{x \to A^+}{\longrightarrow} A^-$, $g(x) \underset{x\to A^+}{\longrightarrow} A^+$. In particular, we also have $p(f(x)) \underset{x \to A^+}{\longrightarrow} A^-$ and $p(g(x)) \underset{x\to A^+}{\longrightarrow} A^+$. From these limits at $A^+$, we deduce that there exists $o_2 \in [o_1,A^+)$ such that $A^-$,  $p(f(o_2))$ and $p(g(o_2))$ are aligned on this order on $L_A$. \\

		Our goal is now to find a point $x \in [o_1,o_2]$ such that $d(p(Bx),p(A^{-n}x)) \leq 26\delta$. To this end, define :
		\begin{eqnarray*}
			h : [o_1,o_2]&\to & \R \\
			x & \mapsto& \pm d(p(f(x)),p(g(x)))
		\end{eqnarray*}
		where the sign is taken to be $+$ when $A^-$, $p(g(o))$ and $p(f(o))$ are aligned on this order on $L_A$ and~$-$ otherwise. Therefore with our convention $h(o_1) \geq 0$ and $h(o_2) \leq 0$.  Note that the map $h$ is not necessarily continuous because the projection map $p$ is not. If $h$ was to be continuous, the intermediate value theorem would allow us to conclude. \\
		
		We will overcome this technical difficulty by defining three sequences $(x_k)_{k \in \N}, (x'_k)_{k \in \N}, (x''_k)_{k \in \N}$ in $[o_1,o_2]$ which satisfy the three following condition :
		\begin{itemize}
			\item $x_k$ is the midpoint of $[x'_k,x''_k]$, where $[x'_k,x''_k]$ denote the geodesic segment from $x'_k$ to $x''_k$ contained in $L_A$
			\item $h(x'_k) \geq 0$ and $h(x''_k) \leq 0$, 
			\item $d(x'_k,x''_k) \underset{k \to \infty}{\longrightarrow} 0$.
		\end{itemize}
		
		This is a standard construction : take $x'_0=o_1$, $x''_0=o_2$ and let $x_0$ be the mid-point of $o_1$ and $o_2$ in $[o_1,o_2]$. To construct $x_{k+1},x'_{k+1}$ and $x''_{k+1}$ from $x_k,x'_k$ and $x''_{k+1}$, proceed as follows. If $h(x_k)\geq 0$, then set $x'_{k+1}=x_k$, $x''_{k+1}=x''_k$, and if $h(x_k) < 0$, then set $x'_{k+1}=x'_k$, $x''_{k+1}=x_k$. In any case, let $x_{k+1}$ be the mid-point of $[x'_{k+1}, x''_{k+1}] \subset L_A$. Therefore we have $d(x'_k,x''_k)= 2^{-k} d(o_1,o_2)$, so $d(x'_k,x''_k) \underset{k \to \infty}{\longrightarrow} 0$ and thus the three conditions above are satisfied. \\ 
		
		By compactness, up to extracting, we can assume that there exists $x_\infty \in [o_1,o_2]$ such that $x_k \to x_\infty , x'_k \to x_\infty $ and $x''_k \to x_\infty $. \\
		
		We now want to prove that there exists $k \in \N$ such that $h(x'_k)\leq 26\delta$ or $h(x''_k) \geq -26 \delta$. 
		Thus, assume by contradiction that for all $k \in \N$, $h(x'_k)>26\delta$ and $h(x''_k) < -26 \delta$. Since $f$ and $g$ are continuous, for $k$ sufficiently large, we have $d(f(x''_k),f(x_\infty)) \leq \delta$, $d(f(x'_k),f(x_\infty)) \leq \delta$, $d(g(x''_k),g(x_\infty)) \leq \delta$, and $d(g(x'_k),g(x_\infty)) \leq \delta$. Therefore, by equation \eqref{eq:proj-bound-distance} of Lemma \ref{lem:proj-convex}, we can deduce an inequality on the projection for $k$ sufficiently large : 
		\begin{align}
			\label{eq:f(x'_k)}	d(p(f(x'_k)),p(f(x_\infty))) & \leq d(f(x'_k),f(x_\infty))+12\delta  \leq 13\delta, \\ 
			\label{eq:f(x''_k)}	d(p(f(x''_k)),p(f(x_\infty))) & \leq d(f(x''_k),f(x_\infty))+12\delta  \leq 13 \delta,  \\ 
			\label{eq:g(x'_k)}	d(p(g(x'_k)),p(g(x_\infty))) & \leq d(g(x'_k),g(x_\infty))+12\delta  \leq 13\delta, \\ 
			\label{eq:g(x''_k)}	d(p(g(x''_k)),p(g(x_\infty))) &  \leq d(g(x''_k),g(x_\infty))+12\delta  \leq 13 \delta. 
		\end{align}
		Now the condition $h(x'_k) > 26\delta$ together with \eqref{eq:f(x'_k)} and \eqref{eq:g(x'_k)} implies that  $A^-$, $p(g(x_\infty))$ and $p(f(x_\infty))$ are aligned on this order on $L_A$ and more precisely that 
		\begin{equation} \label{eq:h>0}
			h(x_\infty) >0.
		\end{equation}
		On the other hand, the condition $h(x''_k) < -26\delta$ together with \eqref{eq:f(x''_k)} and \eqref{eq:g(x''_k)} implies that $A^-$, $p(f(x_\infty))$ and $p(g(x_\infty))$ are aligned on this order on $L_A$ and more precisely that 
		\begin{equation} \label{eq:h<0}
			h(x_\infty) < 0. 
		\end{equation}
		Of course, \eqref{eq:h>0} and \eqref{eq:h<0} gives a contradiction. Therefore we find $k \in \N$ such that either $|h(x'_k)| \leq 26\delta$ or $|h(x''_k)| \leq 26 \delta$. In any case, we found a point $x \in [o_1,o_2]$ such that 
		\begin{equation}\label{eq:d(p(f(o)),p(g(o)))}
			d(p(Bx),p(A^{-n}x)) \leq 26\delta.
		\end{equation}
		
		The end of the proof goes as follow : For this choice of point $o \in L_A$, we have :
		\begin{align*}
			l_S(A^nB) & \leq d(A^nBx,x) \leq d(Bx,A^{-n}x) \\
			& \leq d(Bx,p(Bx))+d(p(Bx),p(A^{-n}x))+d(p(A^{-n}x),A^{-n}x) \\
			& \leq 2\delta+26\delta+2\delta \text{ by equality \eqref{eq: d(Bx,p(Bx)}, \eqref{eq:d(p(f(o)),p(g(o)))} and Corollary \ref{cor:axe-quasi-inv}}, \\
			& \leq 30\delta.
		\end{align*}
	\end{proof} 
	
	\section{The Farey tree and proof of the main equivalences of Theorem \ref{thm:equivalence}} \label{sec:Farey tree}
	In this section, we will introduce and use the Farey tree to prove equivalence \ref{cond:thm:Cdelta} $\iff$ \ref{cond:thm:K>Cdelta} $\iff$ \ref{cond:thm:displacing} in Theorem \ref{thm:equivalence}. We will broadly follow Bowditch's proof strategy \cite{bowditch_markoff_1998}, while making some adaptations along the way which are required in our broader setting. \\
	
	\subsection{Notations}~ \label{subsec:notations} \\
	\indent To ease the reading, we will follow the same notations as in Bowditch's paper (\cite{bowditch_markoff_1998}) which we introduce hereafter. 
	Let $\Sigma$ be a countably infinite simplicial tree with all vertices of degree 3, which we call the \emph{Farey tree}. We shall see $\Sigma$ as embedded in the plane. As an (important) example, one can think of $\Sigma$ as the dual of the Farey tesselation of the hyperbolic plane by ideal triangles (see for example \cite[Chapter 1]{hatcher_topology_2022} or \cite{series_modular_1985}). Denote $V(\Sigma)$ its set of vertex and $E(\Sigma)$ its set of edges. Define a \emph{complementary region} of $\Sigma$ or more simply a \emph{region} to be the closure of a connected component of the complement of the tree $\Sigma$ in the plane, and denote $\Omega$ the set of complementary regions. We will use the capital letters $X$, $Y$, $Z\dots$ to talk about the complementary regions in $\Omega$. Notice that for every vertex $v$ of $\Sigma$, since it has valence 3, $v$ lies in the boundary of exactly three complementary regions $X$, $Y$ and $Z$ of $\Omega$, and in this case we will write $v \leftrightarrow (X,Y,Z)$. An edge $e$ of the tree $\Sigma$, as for it, lies in the boundary of four complementary regions $X$, $Y$, $Z$ and $W$. Two of them, let's say $X$ and $Y$, satisfy $X \cap Y =e$, and the other two are such that $e \cap Z$ and $e \cap W$ are the two endpoints of $e$. In this case, we write $e \leftrightarrow (X,Y ; Z,W)$ and we also denote $\Omega^0(e)=\{X,Y\}$. We can also consider edges with orientation and we denote $\vec{E}(\Sigma)$ the set of oriented edges of $\Sigma$. We use the notation $\vec{e}$ to refer to an oriented edge in $\vec{E}(\Sigma)$ and we write $e$ for the underlying unoriented edge in $E(\Sigma)$. When $e \leftrightarrow (X,Y;Z,W)$ and $\vec{e}$ is oriented from $ Z\cap e$ to $W\cap e$, we say that $Z\cap e$ is the tail of $\vec{e}$ and $W\cap e$ is the head of $\vec{e}$ and we write $\vec{e} \leftrightarrow (X,Y;Z \to W)$. Let $\vec{e} \in \vec{E}(\Sigma)$. By removing the interior of $e$ from $\Sigma$, we obtain two disjoint subtrees, which we denote $\Sigma^{\pm}(\vec{e})$, with $e \cap \Sigma^+(\vec{e})$ being the head and $e \cap \Sigma^-(\vec{e})$ the tail of $\vec{e}$. We can the consider $\Omega^{\pm}(\vec{e})$ the set of complementary regions who boundaries lie in $\Sigma^\pm(\vec{e})$. Then we have the following decomposition of $\Omega$ : $\Omega = \Omega^0(e) \bigsqcup \Omega^+(\vec{e}) \bigsqcup \Omega^-(\vec{e})$. We also denote $E^{\pm}(\vec{e})$ the set of edges and $V^{\pm}(\vec{e})$ the set of vertices in $\Sigma^\pm(\vec{e})$, and  $\Omega^{0\pm}(\vec{e})=\Omega^0(e)\bigsqcup \Omega^\pm(\vec{e})$,   $E^{0\pm}(\vec{e})=\{e\}\sqcup E^{\pm}(\vec{e})$. If $v \in V(\Sigma)$ is a vertex and $X \in \Omega$ is a region, we denote $d(v,X)$ the distance in the tree $\Sigma$ between $v$ and $X$, that is the number of edges of $\Sigma$ in a shortest path from $v$ to $X$. Is $\vec{e}$ is an oriented edge and $X$ is a region in $\Omega^{0-}(\vec{e})$, we denote $d(X,\vec{e})$ the distance between the head of $\vec{e}$ and $X$. Then, if $d(X,\vec{e})=0$, then $X \in \Omega^0(e)$. If $d(X,\vec{e})>0$, there exists an oriented edge $\vec{e'} \in \vec{E}^{0-}(\vec{e})$ and two regions $X', Y' \in \Omega^{0-}(\vec{e})$ such that $X$, $X'$ and $Y'$ meet at the tail of $\vec{e'}$, $d(X',\vec{e})<d(X,\vec{e})$, and $d(Y',\vec{e})<d(X,\vec{e})$. In this case, we also have $\vec{e'}=X'\cap Y'$. Finally, if $e$ and $e'$ are two edges in $E(\Sigma)$, we define $d(e,e')=d(m_e,m'_e)$, where $m_e$ and $m'_e$ are respectivelly the midpoints of $e$ and $e'$. Therefore, $d(e,e')=0$ if and only if $e=e'$ and for all edges $e,e'$, there exist exactly two edges $e_1$ and $e_2$ such that $e_1\cap e_2\cap e' \neq \emptyset$ and $d(e_1,e)=d(e_2,e)=d(e',e)+1$. \\ 
	
	\subsection{Primitive elements in $\F_2$ and complementary regions of $\Sigma$}~ \\
	\indent Recall that we say that an element of $\F_2$ is \emph{primitive} if it is contained is some free basis of $\F_2$. This set of primitive elements is stable by conjugation and inversion and we denote $\mathcal{P}(\F_2)$ the set of primitive elements up to conjugacy and inversion. We now explain how to associate to every complementary region of the Farey tree a primitive element (up to conjugacy and inversion). Fix a free basis $\{a,b\}$ of $\F_2$ and take a distinguished oriented edge $\vec{e_0} \leftrightarrow (X_0,Y_0 ; Z_0 \to W_0)$ in $\vec{E}(\Sigma)$. We first define the map $P$ on $\Omega^{0-}(\vec{e_0})$ by induction on $d(X,\vec{e_0})$, with $X \in \Omega^{0-}(\vec{e_0})$. If $d(X,\vec{e_0})=0$, note that $X \in \Omega^0(\vec{e_0})=\{X_0,Y_0\}$, and then define $P(X_0)=a$ and $P(Y_0)=b$. If $d(X,\vec{e_0})>0$, let $X'$ and $Y'$ be the two regions which meet $X$ and such that $d(X',\vec{e_0}) < d(X,\vec{e_0})$, and $d(Y',\vec{e_0}) <d(X,\vec{e_0})$. By induction, $P(X')$ and $P(Y')$ are already defined and now let $P(X)$ be the concatenation $P(X')P(Y')$. We now define the map $P$ on $\Omega^+(\vec{e_0})$. Note that $\Omega^+(\vec{e_0})=\Omega^-(-\vec{e_0})$, where $-\vec{e_0}$ is the edge $e_0$ with the opposite orientation as $\vec{e_0}$. We define $P$ on $\Omega^-(-\vec{e_0})$ by induction on $d(X,-\vec{e_0}) \geq 1$, with $X \in \Omega^-(-\vec{e_0})$. If $d(X,-\vec{e_0})=1$, then $X=W_0$, and define $P(W_0)=ab^{-1}$. If $d(X,-\vec{e_0})>1$, we define, similarly as on $\Omega^-(\vec{e_0})$, $P(X')$ as the concatenation $P(X')P(Y')$, with $X'$ and $Y'$ the two regions such that $X'$, $Y'$ and $X$ meet and $d(X',-\vec{e_0})<d(X,-\vec{e_0})$, $d(Y',-\vec{e_0})<d(X,-\vec{e_0})$. Therefore, using the basic fact that if $\{u,v\}$ is a basis of $\F_2$, so are $\{u,uv\}$ and $\{uv,v\}$, we easily deduce that for every edge $e \leftrightarrow(X,Y;Z,W)$, denoting $u=P(X)$ and $v=P(Y)$, the set $\{u,v\}$ is a free basis of $\F_2$ and moreover, up to exchanging $Z$ and $W$, $P(Z) \in \{uv,vu\}$ and $P(W) \in \{uv^{-1},u^{-1}v,vu^{-1},v^{-1}u\}$. In addition, all the words which appears as $P(X)$, for $X \in \Omega$, are cyclically reduced. 
	
	Let us now define a map $S : \Omega \to \Q \cup \{\infty\}$ which associate a rational to each region. Again, we proceed by induction. Let $S(X_0)=\infty=\frac{1}{0}$ and $S(Y_0)=0=\frac{0}{1}$. Then, if $X \in \Omega^{-}(\vec{e_0})$ with $X', Y'$ the two regions such that $X$, $X'$ and $Y'$ meet and $d(X',\vec{e_0})<d(X,\vec{e_0})$ and  $d(Y',\vec{e_0})<d(X,\vec{e_0})$, then $S(X)=\frac{p_X+p_Y}{q_X+q_Y}$, with $S(X')=\frac{p_X}{q_X}$ and $S(Y')=\frac{p_Y}{q_Y}$. 
	If $X=W_0$, define $S(W_0)=\frac{-1}{1}$ and if $X \in \Omega^-(-\vec{e_0})$ with $d(X,-\vec{e_0}) >1$, then again $S(X)=\frac{p_X+p_Y}{q_X+q_Y}$ with $S(X')=\frac{p_X}{q_X}$ and $S(Y')=\frac{p_Y}{q_Y}$. Note that if we see the tree $\Sigma$ as the dual as the Farey tesselation of the (hyperbolic) plane, the complementary regions of $\Sigma$ are by construction in natural correspondance with the vertices of this tesselation. It is a classical result that the map $S$ gives a bijection between the vertices of the Farey tesselation and $\Q \cup \{\infty \}$ (see for example \cite[Chapter 1]{hatcher_topology_2022}). Now let us recall that we can associate to each conjugacy and inversion class of primitive element a ``slope" which is defined as follows. Let $\mathrm{Ab} : \F_2 \to \Z^2$ be the abelianisation map, it induces a map from $\mathcal{P}(\F_2)$ to $\Q \cup \{\infty\}$ which we call the ``Slope" map. This map can be simply thought of as the ratio between the number of $a$ and of $b$ in a word $u$. The map ``Slope" turns out to be a bijection between $\mathcal{P}(\F_2)$ and $\Q \cup \{\infty \}$ (see \cite{nielsen_isomorphismen_1917} or \cite{osborne_primitives_1981}). Now note that by construction, $\mathrm{Slope} \circ P=S$. Hence, we deduce that the assignment $P$ of a (conjugacy and inversion class of) primitive elements to each complementary region we defined is a bijection. \\
	
	\subsection{Word length and the maps $F_e$}~ \\ 
	\indent Let $\gamma \in \F_2$. We define $\Vert \gamma \Vert$ the cyclically reduced word length of $\gamma$ in the generating set $\{a,b\}$. Let $\vec{e} \in \vec{E}(\Sigma)$ be an oriented edge in $\Sigma$ and $e$ the underlying unoriented edge. Following Bowditch in \cite{bowditch_markoff_1998}, we are now going to define the map $F_e : \Omega \to \N$ by induction. First define $F_{\vec{e}}(X)$ on $\Omega^{0-}(\vec{e})$. If $X \in \Omega^0(e)$, then let $F_{\vec{e}}(X)=1$, and if $X \in \Omega^{-}(\vec{e})$, and $X',Y'$ are the two regions such that $X$, $X'$ and $Y'$ meet, $d(X',\vec{e})<d(X,\vec{e})$ and $d(Y',\vec{e})<d(X,\vec{e})$, then define $F_{\vec{e}}(X)=F_{\vec{e}}(X')+F_{\vec{e}}(Y')$. Now we define $F_e(X)=F_{\vec{e}}(X)$ on $\Omega^{0-}(\vec{e})$ and $F_e(X)=F_{-\vec{e}}(X)$ on $\Omega^+({\vec{e}})$. We easily check that $F_{\vec{e_0}}(X)=\Vert P(X) \Vert$. Moreover, we can always compare $F_e$ and $F_{e'}$ : for all $e,e' \in E(\Sigma)$, there exists $K>0$ such that $\frac{1}{K}F_e \leq F_{e'} \leq KF_e(X)$. \\ 
	
	\subsection{Representations, maps $l_\rho$ and induced orientations $\alpha_\rho$ }~ \\
	\indent Let $\rho : \F_2 \to \mathrm{Isom}(\mathcal{X})$ be a representation of $\F_2$. We define the map $l_\rho$ from the set of regions $\Omega$ to $\R$ as follows : $l_\rho(X):=l_S(\rho(P(X)))$, for all $X \in \Omega$. Recall that $l_S(A)$ denotes the stable norm of the isometry $A$ and note that the quantity $l_S(\rho(P(X)))$ is well-defined since the stable norm is invariant under conjugation and inversion. 
	
	The map $l_\rho : \Omega \to \R_{\geq 0}$, will induce an orientation on the set of edges as follows. Let $e \leftrightarrow (X,Y;Z,W) \in E(\Sigma)$. Compare $l_\rho(Z)$ and $l_\rho(W)$ and put an arrow on $e$ from $Z$ to $W$ if $l_\rho(Z)>f(W)$ or from $W$ to $Z$ if $l_\rho(W)>l_\rho(Z)$ . If $l_\rho(Z)=l_\rho(W)$, choose any orientation on the edge $e$. This gives a map $\alpha_\rho : E(\Sigma) \to \vec{E}(\Sigma)$ which we call the induced orientation of $\rho$ on $E(\Sigma)$. \\
	
	Let $C_\delta=256 \delta$.  
	
	\begin{Remark}\label{rem:apply Prop to edge}
	Let us make an observation that will be used extensively in this section. 
	Let $\rho : \F_2 \to \mathrm{Isom}(\mathcal{X})$ be an irreducible representation.
	Let $e \leftrightarrow (X,Y;Z,W)$ be an edge in the tree $\Sigma$ satisfying $l_\rho(X)>C_\delta$ and $l_\rho(Y)>C_\delta$. Then $\rho(P(X))$ and $\rho(P(Y))$ are hyperbolic. Moreover recall that the set $\{P(X),P(Y)\}$ is a free basis of $\F_2$, so by irreducibility of $\rho$,  $\{\rho(P(X))^+,\rho(P(Y))^-\} \cap \{\rho(P(Y))^+,\rho(P(Y))^-\}=\emptyset$ (see remark \ref{rem:irreducibleF2}). Therefore, Proposition \ref{prop:ineq-l(AB)} applies, and we can write~:
	\begin{equation*}
		\max ( l_\rho(Z),l_\rho(W)) \geq l_\rho(X)+l_\rho(Y)-C_\delta. 
	\end{equation*}
	If moreover we know that $\alpha_\rho(e) \leftrightarrow (X,Y;Z \to W)$, then we deduce $l_\rho(Z) \geq l_\rho(X)+l_\rho(Y)-C_\delta$. 
	\end{Remark}
	
	Starting from this observation, we derive a first easy and very useful result to study the orientation~$\alpha_\rho$. 
	
	\begin{Lemma}\label{lem:fork}
		Let $\rho : \F_2\to \mathrm{Isom}(\mathcal{X})$ be an irreducible representation.
		Let $v \leftrightarrow (X,Y,Z)$ be a vertex of $\Sigma$. Assume that $l_\rho(X)> C_{\delta},l_\rho(Y)> C_{\delta}$ and $l_\rho(Z) > C_{\delta}$. Then, no two arrows among $\{ \alpha_\rho(X\cap Y),\alpha_\rho(Y\cap Z),\alpha_\rho(X\cap Z) \}$ are pointing away from the vertex $v$. 
	\end{Lemma}
	
	\begin{proof}
		Assume by contradiction that two arrows point away from $v$, and without loss of generality assume that these arrows are $\alpha_\rho(X \cap Y)$ and $\alpha_\rho(X \cap Z)$. Then, as in Remark \ref{rem:apply Prop to edge}, we can apply Proposition \ref{prop:ineq-l(AB)} to $X$ and $Y$ to obtain : 
		\begin{equation}
			l_\rho(X)+l_\rho(Y) \leq l_\rho(Z) +C_{\delta}.
		\end{equation}
		Similarly, the same Proposition \ref{prop:ineq-l(AB)} applied to $X$ and $Z$ gives :
		\begin{equation}
			l_\rho(X)+l_\rho(Z) \leq l_\rho(Y) +C_{\delta}.
		\end{equation}
		By summing these two inequalities we deduce that 
		\begin{equation}
			2l_\rho(X)+l_\rho(Y)+l_\rho(Z) \leq l_\rho(Z)+l_\rho(Y)+2C_{\delta},
		\end{equation}
		and then we have $l_\rho(X)\leq C_{\delta}$, which is a contradiction.  
	\end{proof}
	
	Using Lemma \ref{lem:fork}, we can study the sublevet sets $\Omega_\rho(C):=\{ X \in \Omega : l_\rho(X) \leq  C \}$ for all $C\geq C_\delta$, and prove : \\ 
	
	\begin{Lemma} \label{lem:connected}
		Let $\rho : \F_2 \to \mathrm{Isom}(\mathcal{X})$ be an irreducible representation. 
		For all $C \geq C_\delta$, the subset of the plane $\displaystyle \underset{X \in \Omega_\rho(C)}{\bigcup} X$ is connected. 
	\end{Lemma}
	
	\begin{proof}
		The proof of this Lemma will very much follow the proof of the corresponding proposition in \cite{bowditch_markoff_1998}. 
		
		First notice that since $C \geq C_\delta$, Proposition \ref{prop:ineq-l(AB)} is still true when replacing $C_\delta$ by $C$. Assume by contradiction that this subset of the plane is not connected, and consider two regions $Z$ and $W$ satisfying the following : $l_\rho(Z)\leq C$, $l_\rho(W)\leq C$, $Z\cap W = \emptyset$, and the shortest path between $Z$ and $W$ does not meet in its interior a region $R$ with $l_\rho(R) \leq C_\delta$. If the length of this path is $1$, let $e \leftrightarrow (X,Y ; Z,W)$ be the only edge of the path. Then $l_\rho(X)>C$ and $l_\rho(Y) >C$ so Proposition \ref{prop:ineq-l(AB)} applies and we deduce that $\max(l_\rho(Z),l_\rho(W)) \geq l_\rho(X)+l_\rho(Y)-C >C$. Of course, this contradicts the assumption on $Z$ and $W$. Now assume that the length of the path is at least 2. Denote $e_1$ the edge of the path which meets $Z$ at one of its endpoint and $e_2$ the edge of the path which meets $W$ at one of its endpoint. Then, since $e_1$ meets no region with length smaller or equal to $C$ except $Z$, we deduce that the orientation on $\alpha_\rho(e_1)$ is toward $Z$. Similarly, the orientation on $\alpha_\rho(e_2)$ is toward $W$. This implies the existence of a vertex in the interior of the path with two oriented edge (for the orientation $\alpha_\rho$) pointing away from it. Since all the region meeting this vertex have length greater that $C$, this contradicts Lemma \ref{lem:fork}.
	\end{proof}
	
	When $\vec{e}$ and $\vec{e'}$ are two oriented edges, we say that $\vec{e'}$ \emph{point toward} $\vec{e}$ if $\vec{e'}=\vec{e}$ or if $e \in E^+(\vec{e'})$ and $e' \in E^-(\vec{e})$. Note that this is a transitive relation. 
	
	\begin{Lemma}\label{lem:growth on big edge}
		Let $\rho : \F_2 \to \mathrm{Isom}(\mathcal{X})$ be an irreducible representation. 
		Let $e \leftrightarrow (X,Y ; Z,W)$ be an edge of $\Sigma$ such that $l_\rho(X)>C_\delta$ and $l_\rho(Y)>C_\delta$. Let $\vec{e}:=\alpha_\rho(e)$. Then for every $e' \in \Sigma^{0-}(\overset{\rightarrow}{e})$, $\alpha_\rho(e')$ points towards $\vec{e}$ and for every $X' \in \Omega^{0-}(\vec{e})$, we have the inequality $l_\rho(X') \geq (m-C_\delta)F_{\vec{e}}(X')+C_\delta$, where $m=\min\{l_\rho(X),l_\rho(Y)\}$.
	\end{Lemma}
	
	\begin{proof}
		Let us first prove that for all edge $e'=X'\cap Y' \in E^{0-}(\vec{e})$,  we have $l_\rho(X')>C_\delta$, $l_\rho(Y') >C_\delta$ and $\alpha(e')$ points toward $e$. We proceed by induction on $d=d(e',e) \geq 0$. 
		\begin{itemize}
			\item The property is true by assumption for $d=0$, since in this case $e'=e$.
			\item Suppose now that $e_0 \in E$ is at distance $d \geq 0$ of $e$. Let $e_1$ and $e_2$ be the two edges at distance $d+1$ of $e$ which intersect with $e_0$, and $X', Y', Z'$ be the three regions such that $e_0=X'\cap Y'$, $e_1=X' \cap Z'$ and $e_2=Y' \cap Z'$. Also denote $v=e_0 \cap e_1 \cap e_2$.  Then by the induction hypothesis, $l_\rho(X')>C_\delta$ and $l_\rho(Y')>C_\delta$ and $\alpha_\rho(e_0)$ points away from $v$. Therefore, we can use Proposition \ref{prop:ineq-l(AB)} to deduce $l(Z') \geq l(X')+l(Y')-C_\delta >C_\delta$. Now we know that the three regions around $v$ satisfy the hypothesis of Lemma \ref{lem:fork} and that $\alpha_\rho(e_0)$ points away from $v$, so both $\alpha_\rho(e_1)$ and $\alpha_\rho(e_2)$ point toward $\alpha_\rho(e_0)$, which itself points toward $\vec{e}$, so both $\alpha_\rho(e_1)$ and $\alpha_\rho(e_2)$ point toward $\vec{e}$ and the induction is proved. 
		\end{itemize}
		Now the inequality $l_\rho(X') \geq (m-C_\delta)F_{\vec{e}}(X')+C_\delta$ follow from an easy induction on $d=d(X',\vec{e})\geq 0$.
		\begin{itemize}
			\item If $d(X',\vec{e})=0$, then $X'=X$ or $X'=Y$ and then the inequality is trivial.
			\item If $d(X',\vec{e})>0$, then there exists $\vec{e'} \in \vec{E}^{0-}(\vec{e})$ and two regions $Y', Z' \in \Omega^{0-}(\vec{e})$ such that $X'$, $Y'$ and $Z'$ meet at the tail of $\vec{e'}$, $d(Y',\vec{e}) <d(X',\vec{e})$ and $d(Z',\vec{e})<d(X',\vec{e})$. We have $Y'\cap Z'=e'$.
			 By induction $\alpha_\rho(e')=\vec{e'}$ and $l_\rho(Y') >C_\delta$, $l_\rho(Z')>C_\delta$, so Proposition \ref{prop:ineq-l(AB)} ensures $l_\rho(X') \geq l(Y')+l(Z')-C_\delta$. Now, by the induction hypothesis we also have $l(Y') \geq (m-C_\delta)F_{\vec{e}}(Y')+C_\delta$ and $l(Z') \geq (m-C_\delta)F_{\vec{e}}(Z')+C_\delta$. The desired inequality is now obtained combining these three inequalities and using that $F_{\vec{e}}(X')=F_{\vec{e}}(Y')+F_{\vec{e}}(Z')$.
		\end{itemize}
	\end{proof}
	
	\subsection{The equivalence \ref{cond:thm:Cdelta} $\iff$ \ref{cond:thm:K>Cdelta} $\iff$ \ref{cond:thm:displacing} in Theorem \ref{thm:equivalence}}~ \\ 
	\indent We now study more specificaly the representations which satisfies the conditions of Theorem \ref{thm:equivalence}. Let us recall the conditions \ref{cond:thm:Cdelta}, \ref{cond:thm:K>Cdelta} and \ref{cond:thm:displacing} here.
	
	\begin{Definition} Let $\mathcal{X}$ be a $\delta$-hyperbolic, geodesic, visibility space. Let $K_\delta=329\delta$. \\
		Let $\rho : \F_2 \to \mathrm{Isom}(\mathcal{X})$ be a representation. We say that $\rho$ satisfies :
		\begin{enumerate}[label=(BQ\alph*)]
			\item \label{def:BQa}if for all $[\gamma] \in \mathcal{P}(\F_2)$, $\rho(\gamma)$ is hyperbolic and \\ 
			the set $\left\{ [\gamma] \in \mathcal{P}(\F_2) \, \vert \, l_S(\rho(\gamma)\leq K_\delta)\right\}$ is finite, 
			\item \label{def:BQb} if for all $[\gamma] \in \mathcal{P}(\F_2)$, $\rho(\gamma)$ is hyperbolic and \\
			for every $K>0$, the set $\left\{ [\gamma] \in \mathcal{P}(\F_2) \, \vert \, l_S(\rho(\gamma)\leq K)\right\}$ is finite,
			\addtocounter{enumi}{1}
			\item \label{def:BQd} if there exist $C>0$ and $D\geq 0$ such that for all $[\gamma] \in \mathcal{P}(\F_2)$, $l_S(\rho(\gamma)) \geq \frac{1}{C}\Vert \gamma \Vert -D$.
	\end{enumerate}
	\end{Definition}
	
	\begin{Theorem} \label{thm:equivalence a,b,d} Let $\mathcal{X}$ be a $\delta$-hyperbolic, geodesic, visibility space. Let $K_\delta=329\delta$. \\
		Let $\rho : \F_2 \to \mathrm{Isom}(\mathcal{X})$ be a representation. Then, \\ 
		$\rho$ satisfies \ref{def:BQa} $\iff$ $\rho$ satisfies \ref{def:BQb} $\iff$  $\rho$ satisfies \ref{def:BQd}.
	\end{Theorem}
	
	\begin{proof}
		The implication \ref{def:BQb} $\implies$ \ref{def:BQa} is trivial. 
		
		The implication \ref{def:BQd} $\implies$ \ref{def:BQb} is nearly immediate when noticing the following claim : Let $C>0$ and $D\geq 0$ such that for all $[\gamma] \in \mathcal{P}(\F_2)$, $l_S(\rho(\gamma)) \geq \frac{1}{C} \Vert \gamma \Vert -D$, then the stronger inequality $l_S(\rho(\gamma)) \geq \frac{1}{C}\Vert \gamma \Vert$ holds for every $[\gamma] \in \mathcal{P}(\F_2)$. Indeed, let $\gamma$ be a cyclically reduced primitive element of $\F_2$. Take $\beta$ another primitive element of $\F_2$ such that for all $n \in \N$, $\gamma^n\beta$ is cyclically reduced and primitive. Then the length $\Vert \gamma^n\beta \Vert$ is equal to $n\Vert \gamma \Vert + \beta$. Let $o \in \mathcal{X}$. Therefore, condition \ref{def:BQd} implies, for all $n \in \N$ : 
		\begin{align*}
			d(\rho(\gamma)^no,o)+d(\rho(\beta)o,o) \geq d(\rho(\gamma^n\beta)o,o) \geq l_S(\rho(\gamma^n\beta )) \geq \frac{1}{C} \Vert \gamma^n\beta \Vert -D \geq \frac{1}{C}n\Vert \gamma \Vert +\frac{1}{C} \Vert \beta \Vert -D.
		\end{align*} 
		Dividing by $n$ and taking the limit in the previous inequality proves the claim. Now from the claim and the characterization of hyperbolicity via the positivity of the stable norm, we deduce that for all $[\gamma] \in \mathcal{P}(\F_2)$, $\rho(\gamma)$ is hyperbolic. The claim also implies that the set $\left\{ [\gamma] \in \mathcal{P}(\F_2) \, \vert \, l_S(\rho(\gamma)\leq K)\right\}$ is finite for every $K >0$ since there is only a finite number of cyclically reduced word $\gamma$ such that $\Vert \gamma \Vert$ is bounded by some constant.  
		
		The implication \ref{def:BQa} $\implies$ \ref{def:BQd} is therefore the core of Theorem \ref{thm:equivalence a,b,d}. It is proven in Proposition \ref{lem:Omega vide} in the case $\Omega_\rho(C_\delta)=\emptyset$ and in Proposition \ref{lem:Omega non vide} in the case $\Omega_\rho(C_\delta) \neq \emptyset$.

	\end{proof}
	
	\subsection{The case $\Omega_\rho(C_{\delta})=\emptyset$}~ \\ 
	Let $(\vec{e_n})_{n\in \N}$ be a sequence of oriented edges. We say that $(\vec{e_n})_{n\in \N}$ is an \emph{escaping ray} if for all $n\in \N$, the head of $\vec{e_n}$ is the tail of $\vec{e_{n+1}}$. 
	We say that the orientation $\alpha_\rho$ has an \emph{escaping ray} if there exists an escaping ray $(\vec{e_n})_{n\in \N}$ with $\alpha_\rho(e_n)=\vec{e_n}$ for all $n\in \N$. 
	\begin{Lemma}\label{lem:no-escaping-ray}
	Let $\rho : \F_2 \to \mathrm{Isom}(\mathcal{X})$ satisfying \ref{def:BQa}. If $\Omega_\rho(C_{\delta}) = \emptyset$, then the orientation $\alpha_\rho$ has no escaping ray. 
	\end{Lemma}
	
	Before starting the proof, we introduce a tri-coloring of $\Omega$, that is a map $\mathrm{Col} : \Omega \to \{1,2,3\}$ satisfying that for every edge $e \leftrightarrow (X,Y;Z,W)$, $\mathrm{Col}(X)\neq \mathrm{Col}(Y)$. Note that as a consequence, $\mathrm{Col}(Z)=\mathrm{Col}(W)$ and the three colors around a vertex are all distinct. For all $j=1,2,3$, let us denote $\Omega^j \subset \Omega$ the subset of regions of color $j$. 
	
	\begin{proof}
		Assume by contradiction that there exists an escaping ray $(\vec{e_n})_{n\in \N}$ for the orientation $\alpha_\rho$. The tri-coloring of $\Omega$ gives for all $j =1,2,3$, three subsets $I_j$ of $\N$ and three sequences of regions $(X^j_m)_{m\in I_j}$ in $\Omega^j$ meeting the escaping ray, which means that $\underset{n \in \N}{\bigcup} \Omega^0(e_n) =  \underset{j=1,2,3}{\bigsqcup} \; \underset{m \in I_j}{\bigsqcup} X^j_m$. We may also assume that the regions $(X^j_m)_{m \in I_j}$ are ordered following the order induced on the regions when following the escaping ray $(\vec{e_n})_{n\in \N}$. 
		 Note that the subsets $I_j$ for $j=1,2,3$ might a priori be finite, but that at least two of them are infinite. Let us now justify that, in fact, the three subsets $I_j$, for $j=1,2,3$, are infinite. Indeed, assume by contradiction that one on them is finite, say $I_1$ without loss of generality, and notice that this means that for $n$ sufficiently large, all the edges $e_n$ lie in the boundary of the same region $X$ belonging to $\Omega^1$. In particular, $I_2$ and $I_3$ are inifinite and we may assume without loss of generality that $I_1=I_2=\N$ and for all $n \geq 0$, $\Omega^0(e_{2n})=\{X^2_n,X\}$ and $\Omega^0(e_{2n+1})=\{X^3_n,X\}$. Let us denote $Y=X^2_0 \in \Omega^2$. Also denote $a=P(X)$, $b=P(Y)$, $A=\rho(a)$ and $B=\rho(P(b))$. We deduce that for all $n\geq 0$, $\rho(P(X^2_{n}))$ is conjugated to $BA^{2n}$ and so $l_\rho(X^2_n)=l_S(BA^{2n})$. One the one hand, we deduce, using that $(\vec{e_n})_{n\in \N}$ is an escaping ray for $\alpha_\rho$, that the sequence $(l_\rho(X^2_n))_{n\in \N}$ is non-increasing, hence bounded. On the other hand, note that since $\rho$ satisfy \ref{def:BQa}, Proposition \ref{prop:B(A+)neqA-explicite} implies that $B(A^+)\neq A^-$ and then using Proposition \ref{prop:AnB} we obtain that $l_S(BA^{2n})$ grows linearly in $n$ for n sufficiently large. These two facts contradict each other. Then all three subset $I_j$ are infinite and let $I_j=\N$, for $j=1,2,3$ without loss of generality. 
		
		The assumption that the sequence $(\vec{e_n})_{n\in \N}$ is an escaping ray for $\alpha_\rho$, gives that the three sequences $(l_\rho(X^j_n))_{n\in \N}$ are non-increasing. Since they are obviously bounded from bellow, we deduce that they are converging. This convergence implies that for $n$ large enough we can assume that the terms of the sequence are as close as we want. Thus, there exists $M$, such that for all $m \geq M$, for all $j=1,2,3$, $l(X^j_{m}) \leq l(X^j_{m+1})+2\delta$. Now consider $n \in \N$ such that $\vec{e_n} \leftrightarrow (X,Y;X^j_m \to X^j_{m+1})$ is the edge from $X^1_m$ to $X^1_{m+1}$ in the sequence $(\vec{e_n})_{n\in \N}$. Up to exchanging $X$ and $Y$, assume that $X \cap X^j_{m+1}=e_{m+1}$. The arrow on $\vec{e_n}$ goes from $X^1_m$ to $X^1_{m+1}$ and all region satisfy $l_\rho(X) > C_{\delta}$, so by Proposition \ref{prop:ineq-l(AB)} :
		\begin{equation}
			l_\rho(X)+l_\rho(Y) \leq l_\rho(X^1_m)+C_{\delta}.
		\end{equation}
		Similarly, because the arrow on $\vec{e_{n+1}}$ points away from $X\cap Y \cap X^1_{m+1}$, we have 
		\begin{equation}
			l_\rho(X)+l_\rho(X^1_{m+1})\leq l_\rho(Y)+C_{\delta}.
		\end{equation}
		Now combining these two we obtain :
		\begin{equation}
			2l_\rho(X)+l_\rho(X^1_{m+1}) \leq l_\rho(X)+l_\rho(Y)+C_{\delta} \leq l_\rho(X^1_m)+2C_{\delta} \leq l_\rho(X^1_{m+1})+2\delta+2C_{\delta}
		\end{equation}
		and thus $l_\rho(X)\leq C_{\delta}+\delta$. Since the three sequences $(X^j_m)_{m \in \N}$ are infinite, by repeting this procedure we can find infinitely many regions $X$ so that $l_\rho(X) \leq C_{\delta}+\delta=257\delta <K_\delta$. This contradicts the hypothesis \ref{def:BQa} on $\rho$.  
	\end{proof}
	
	We say that a vertex $v \leftrightarrow (X,Y,Z)$ in $V(\Sigma)$ is a \emph{sink} for an orientation $\alpha : E(\Sigma) \to \vec{E}(\Sigma)$ is $v$ is the head of $\alpha(X\cap Y)$, $\alpha(X\cap Z)$ and $\alpha(Y\cap Z)$. 
	\begin{Proposition} \label{lem:sink}
		Let $\rho : \F_2 \to \mathrm{Isom}(\mathcal{X})$ satisfying \ref{def:BQa}. If $\Omega_\rho(C_{\delta}) = \emptyset$, then the orientation $\alpha_\rho$ has a unique sink.
	\end{Proposition}
	\begin{proof} First note that by Proposition \ref{prop:irreducible}, $\rho$ is irreducible. Therefore, since $\Omega_\rho(C_{\delta}) = \emptyset$ we can apply Lemma \ref{lem:fork} to every vertex of $V(\Sigma)$.
		
	The existence of the sink comes from the fact that the orientation $\alpha_\rho$ has no escaping ray (Lemma \ref{lem:no-escaping-ray}) combined with Lemma \ref{lem:fork}. Indeed, take any edge $e$, its orientation $\alpha_\rho(e)$ points toward a vertex $v$ of the tree. By Lemma \ref{lem:fork}, the orientation $\alpha_\rho$ on the two other edges at $v$ cannot both point away from $v$. If they are both pointing toward $v$, then we found a sink. If not, exactly one of them points away from $v$. Repeat the procedure with this new edge. If there was no sink, we could continue the procedure forever and then find an escaping ray for $\alpha_\rho$, which is forbidden by Lemma \ref{lem:no-escaping-ray}. Hence the procedure stops, and we find a sink for $\alpha_\rho$. 
	
		Let us show unicity. If they were two sinks for $\alpha_\rho$, then consider a path of edges between these two sinks. Because of the orientation $\alpha_\rho$ on the first and last edges of the path (which must be distinct), we deduce that we could find a vertex on the path with two arrows pointing away from it. This is forbidden by Lemma \ref{lem:fork}.
	\end{proof}

	\begin{Proposition}\label{lem:Omega vide}
		Let $\rho : \F_2 \to \mathrm{Isom}(\mathcal{X})$ satisfying \ref{def:BQa}. If $\Omega_\rho(C_{\delta}) = \emptyset$, then there exists $C>0$, such that for all $X \in \Omega$, $l_\rho(X) \geq \frac{1}{C}F_{e_0}(X)$.
	\end{Proposition}
	
	\begin{proof}
		Let $v$ be the unique sink given by Proposition \ref{lem:sink} and denote by $e_1$, $e_2$ and $e_3$ the three edges which meet at $v$. Notice that $\Omega=\Omega^{0-}(\alpha_\rho(e_1)) \cup \Omega^{0-}(\alpha_\rho(e_2)) \cup \Omega^{0-}(\alpha_\rho(e_3))$. Then we apply Lemma \ref{lem:growth on big edge} to $e_1$, $e_2$ and $e_3$ and we obtain the desired growth (recall that for all edges $e,e'$, there exists $K>0$ such that $\frac{1}{K}F_e \leq F_{e'} \leq KF_e$). \\
	\end{proof}
	
	\subsection{The tree $\mathcal{T}_{\rho}$} \label{subsec:tree initial}~ \\ 
	\indent Before studying the case where $\Omega_\rho(C_\delta) \neq \emptyset$, let us define a subtree $\mathcal{T}_{\rho}$ of $\Sigma$ which will help us handle situations involving small regions. \\ 
	
	Fix $\rho : \F_2 \to \mathrm{Isom}(\mathcal{X})$ an irreducible  representation.
	Let $E_\rho(C_{\delta})=\{ e \in E(\Sigma) : \Omega^0(e) \subset \Omega_\rho(C_{\delta})\}$. 
	
	\begin{Remark}\label{rem:Erho-finite}
	Note that if $\rho$ satisfy \ref{def:BQa}, then $E_\rho(C_{\delta})$ is finite because $\Omega_\rho(C_\delta)=\{X \in \Omega : l_\rho(X) \leq  C_{\delta} \}$ is finite. Indeed, $\Omega^0$ gives a map from $E_\rho(C_{\delta})$ to the finite set $\Omega_\rho(C_{\delta})^{(2)}$ (the set of distincts pairs in $\Omega_\rho(C_\delta)$) and this map is injective because two regions intersect along at most one edge. 
	\end{Remark} 

	Let us assume that $\Omega_\rho(C_\delta)\neq \emptyset$ and suppose first that $E_\rho(C_\delta)\neq \emptyset$. Then define $\mathcal{T}_\rho$ to be the tree spanned by $E_\rho(C_{\delta})$. It is a finite tree when $E_\rho(C_{\delta})$ is finite. If $E_\rho(C_\delta) = \emptyset$ but still $\Omega_\rho(C_\delta) \neq \emptyset$, the connectedness of $\Omega_\rho(C_\delta)$ shown in Lemma \ref{lem:connected} implies that $\Omega_\rho(C_\delta)$ consist of a single region. In this case, we take $\mathcal{T}_\rho$ to be any edge on the boundary of this region. 	
	 
	Let us give a simple and useful observation. 
	\begin{Lemma}\label{lem:edge of tree}
		Let $\rho : \F_2 \to \mathrm{Isom}(\mathcal{X})$ be an irreducible representation. 
		Let $e=X \cap Y \in \mathcal{T}_\rho$. Then either $l_\rho(X) \leq C_\delta$ or $l_\rho(Y) \leq C_\delta$. 
	\end{Lemma}
	\begin{proof}
		This is a direct consequence of the connectedness of the union of the regions in $\Omega_\rho(C_\delta)$ (see Lemma \ref{lem:connected}).
	\end{proof}
	
	When $T$ is a subtree of $\Sigma$, we define its \emph{circular set} $C(T) \subset \vec{E}(\Sigma)$ to be the subset of oriented edge $\vec{e}$ such that $\vec{e} \cap T$ is exactly the head of $\vec{e}$. Note that if $T$ is finite, its circular set $C(T)$ is also finite. 
	
    Let us give some basic properties on the edges of the circular set  $C(\mathcal{T}_\rho)$. 
	\begin{Lemma}\label{lem:circular-set}
		Let $\rho : \F_2 \to \mathrm{Isom}(\mathcal{X})$ be an irreducible representation and assume $\Omega_\rho(C_\delta)\neq \emptyset$. \\
		Let $\vec{e} \leftrightarrow (X,Y;Z \to W)$ be an edge in $C(\mathcal{T}_\rho)$. Then 
		\begin{enumerate}[label=(\alph*)]
			\item \label{lemcase:l(X)} Either $l_\rho(X) >C_\delta$ or $l_\rho(Y)>C_\delta$.
			\item \label{lemcase:l(Z)} $l_\rho(W) \leq C_\delta$.
			\item \label{lemcase:l(W)} Let $X' \in \Omega^-(\vec{e})$. Then $l_\rho(X') >C_\delta$.
			\item \label{lemcase:alpha} $\alpha_\rho(e)=\vec{e}$.
		\end{enumerate}
	\end{Lemma}
	\begin{proof}
		\begin{enumerate}[label=(\alph*)]
			\item Assume by contradiction that $l_\rho(X) \leq C_\delta$ and $l_\rho(Y) \leq C_\delta$, then this means that $e \in E_\rho(C_\delta) \subset \mathcal{T}_\rho$ which is impossible by definition of $C(\mathcal{T}_\rho)$.
			\item Denote $e_1=X\cap W$ and $e_2=Y \cap W$. Then, since $\vec{e} \in C(\mathcal{T}_\rho)$, by definition of the circular set, either both $e_1$ and $e_2$ belong to $\mathcal{T}_\rho$ or exactly one of them belong to $\mathcal{T}_\rho$. If both belong to $\mathcal{T}_\rho$ and $l_\rho(W) >C_\delta$, then by Lemma \ref{lem:edge of tree}, both $l_\rho(X) \leq C_\delta$ and $l_\rho(Y)\leq C_\delta$, but this would imply that $e \in \mathcal{T}_\rho$ which contradicts the fact that $\vec{e} \in C(\mathcal{T}_\rho)$. Now assume without loss of generality that $e_1 \in \mathcal{T}_\rho$ and $e_2 \neq \mathcal{T}_\rho$, then it implies that $e_1 \in E_\rho(C_\delta)$ and thus that $l_\rho(X) \leq C_\delta$ and $l_\rho(W) \leq C_\delta$. 
			\item Assume by contradiction that $X' \in \Omega^-(\vec{e})$ is such that $l_\rho(X') >C_\delta$. Then, by connectedness of $ \underset{X \in \Omega_\rho(C_\delta)}{\bigcup}X$ (see Lemma \ref{lem:connected}), we would find an edge $e' \in \Sigma^-(\vec{e})$ such that $e' \in E_\rho(C_\delta) \subset \mathcal{T}_\rho$. Now, using the same notations as in \ref{lemcase:l(Z)}, namely $e_1=X\cap W$ and $e_2=Y\cap W$, recall that either $e_1$ or $e_2$ belongs to $\mathcal{T}_\rho$. By connectedness of the tree $\mathcal{T}_\rho$, this implies that also $e \in \mathcal{T}_\rho$ which is false since $\vec{e} \in C(\mathcal{T}_\rho)$.
			\item This is immediate from the previous points. Indeed, $Z \in \Omega^{-}(\vec{e})$ so by \ref{lemcase:l(W)}, $l_\rho(Z) >C_\delta$ and by \ref{lemcase:l(Z)}, $l_\rho(W)\leq C_\delta$.
		\end{enumerate}
	
	\end{proof}

	\subsection{The case $\Omega_\rho(C_\delta) \neq \emptyset$}~ \\ 
	\indent We now prove in this section implication \ref{def:BQa} $\implies $ \ref{def:BQd} in the case $\Omega_\rho(C_\delta) \neq \emptyset$. The key point is the following Lemma :
	\begin{Lemma} \label{lem:growth when l(X) petit}
	Let $\rho : \F_2\to \mathrm{Isom}(\mathcal{X})$ be a representation satisfying \ref{def:BQa}. Assume that $\Omega_\rho(C_\delta)\neq \emptyset$. \\
	Let $\vec{e} \leftrightarrow (X,Y; Z\to W)$ be an oriented edge in $C(\mathcal{T}_\rho)$. Assume moreover that $l_\rho(X) \leq C_\delta$. Then, there exists $C>0$ such that for all $X' \in \Omega^{0-}(\vec{e})$, we have $l_\rho(X') \geq \frac{1}{C}F_e(X')$. 
	\end{Lemma}
	\begin{proof}
	Denote by $\vec{e_n}$, for all $n \in \N$, the edges on the boundary of the region $X$ such that $\vec{e_0}=\vec{e}$ and 
	the head of $\vec{e_{n+1}}$ is the tail of $\vec{e_n}$. Denote $Y_n$, for $n\in \N$, the regions such that $\Omega^0(e_n)=\{X,Y_n\}$ and let $\vec{f_n} \in \vec{E}(\Sigma)$ and $W_n \in \Omega$ be such that $\vec{f_n} \leftrightarrow (Y_n,Y_{n+1} ; W_n \to X)$. 
	
	First notice that for all $n \geq 1$, $Y_n \in \Omega^-(\vec{e})$ and then we deduce by Lemma \ref{lem:circular-set} \ref{lemcase:l(W)} that $l_\rho(Y_n)>C_\delta$. In addition, $\vec{e} \in C(\mathcal{T}_\rho)$ and $l_\rho(X) \leq C_\delta$, so by Lemma \ref{lem:circular-set} \ref{lemcase:l(X)} we also have $l_\rho(Y)=l_\rho(Y_0) > C_\delta$. Now let $n \in \N$ and let us justify that $\alpha_\rho(f_n)=\vec{f_n}$. The region $W_n \in \Omega^{-}(\vec{e})$ so by Lemma \ref{lem:circular-set} \ref{lemcase:l(W)}, we have $l_\rho(W_n) >C_\delta$. Since $l_\rho(X) \leq C_\delta$ by assumption, we have $l_\rho(W_n) >l_\rho(X)$, that is $\alpha_\rho(f_n)=\vec{f_n}$. As a consequence, $\vec{f_n}$ satisfy the hypothesis of Lemma \ref{lem:growth on big edge}, and we deduce :
	\begin{equation*}\label{eq:l_rho(X)>Ffn}
		\text{ For all } X' \in \Omega^{0-}(\vec{f_n}), \quad l_\rho(X') \geq (m_n-C_\delta)F_{f_n}(X')+C_\delta, \quad \text{ with } m_n=\min\{l_\rho(Y_n),l_\rho(Y_{n+1})\}.
	\end{equation*}
	Note that $m_n >C_\delta$ and that the maps~$F_{f_n}$ and $F_e$ are related as follows : $F_{f_n} \geq \frac{1}{n+2}F_e$. Then, 
	\begin{equation}\label{eq:l_rho(X)>Fe}
		\text{ For all } X' \in \Omega^{0-}(\vec{f_n}), \quad l_\rho(X') \geq \frac{m_n-C_\delta}{n+2}F_{e}(X'), \quad \text{ with } m_n=\min\{l_\rho(Y_n),l_\rho(Y_{n+1})\}.
	\end{equation}
	Moreover, denoting $a=P(X)$, $b=P(Y)$, $A=\rho(a)$ and $B=\rho(b)$, we have that $P(Y_n)$ is conjugated to $a^nb$. Moreover, since $\rho$ satisfies \ref{def:BQa} and $a^nb$ is primitive for all $n \in \N$, Proposition \ref{prop:B(A+)neqA-explicite} implies that $B(A^+) \neq A^-$. Therefore, by Proposition \ref{prop:AnB}, there exists $N$ and $k$ such that for all $n \geq N$, $l_\rho(Y_n)=l_S(BA^n) \geq l_S(A)n-k$. Therefore, $\displaystyle \frac{m_n-C_\delta}{n+2} \geq \frac{l_S(A)n-k-C_\delta}{n+2} \underset{n \to + \infty}{\longrightarrow} l_S(A)$, and so for $n$ large enough, we deduce that 
	\begin{equation}\label{eq:l_rho(X)>lS(A)Fe}
		\text{ For all } X' \in \Omega^{0-}(\vec{f_n}), \quad l_\rho(X') \geq \frac{l_S(A)}{2}F_{e}(X').
	\end{equation}
	Finally note that we can decompose $\Omega^{0-}(\vec{e})$ as $\Omega^{0-}(\vec{e})=\{X\} \bigsqcup \underset{n \in \N}{\bigcup} \Omega^{0-}(\vec{f_n})$, so combining \eqref{eq:l_rho(X)>lS(A)Fe} for $n$ large and \eqref{eq:l_rho(X)>Fe} for $n$ small, we find a constant $C>0$ so that the inequalities $l_\rho(X') \geq \frac{1}{C}F_e(X')$ hold for all $X' \in \Omega^{0-}(\vec{e})$. 
	\end{proof}
	
	We deduce the following Proposition :
	\begin{Proposition}\label{lem:Omega non vide}
	Let $\rho : \F_2\to \mathrm{Isom}(\mathcal{X})$ be a representation satisfying \ref{def:BQa}. 
	If $\Omega_\rho(C_\delta)\neq \emptyset$, then there exists $C>0$ such that for all $X \in \Omega$, $l_\rho(X)\geq \frac{1}{C} F_{e_0}(X)$. 
 	\end{Proposition}
	
	\begin{proof}
		Consider the tree $\mathcal{T}_\rho$ and its circular set $C(\mathcal{T}_\rho)$. Since $E_\rho(C_\delta)$ is finite (see Remark \ref{rem:Erho-finite}), so are $\mathcal{T}_\rho$ and $C(\mathcal{T}_\rho)$. Moreover, note that $\Omega=\underset{\vec{e} \in C(\mathcal{T}_\rho)}{\bigcup} \Omega^{0-}(\vec{e})$. Since this is a finite union, it is sufficient to show that for all $\vec{e} \in C(\mathcal{T}_\rho)$, there exists $C_{\vec{e}} >0$ such that for all $X \in \Omega^{0-}(\vec{e})$, $l_\rho(X) \geq \frac{1}{C_{\vec{e}}}F_{e}(X)$. This is done in Lemma \ref{lem:growth on big edge} in the case where the two regions meeting at $e$ have lengths greater than $C_\delta$ (note that we can apply this Lemma since $\alpha_\rho(e)=\vec{e}$ by Lemma \ref{lem:circular-set} \ref{lemcase:alpha}) and in Lemma \ref{lem:growth when l(X) petit} in the case where the length of one of two regions is smaller or equal to $C_\delta$. 
	\end{proof}

\section{Openness, dynamics of $\mathrm{Out}(\F_2)$ and recognition} \label{sec:openness dynamics recognition}
In this section, we explain how we can use the framework developed in the previous section to show that the set of Bowditch representation is open and that the action of the outer automorphism group of $\F_2$ acts properly discontinuously on it. We will proceed along lines similar to those in the corresponding parts of previous work (\cite{bowditch_markoff_1998}, \cite{tan_generalized_2008}, \cite{lawton_dynamics_2025}); however, the large-scale behavior of our setting necessitates some adjustments which we will try to highlight.\\

\subsection{The tree $T_\rho(K)$}~ \label{subsec:treeJ}\\ 
\indent Let us choose, for every region $X \in \Omega$, another region $Y_X \in \Omega$ such that $X \cap Y_X$ is an edge of $\Sigma$, equivalently $X \cap Y_X \in \partial X$. Now denote $a=P(X)$ and $b=P(Y_X)$, we have that $\{a,b\}$ form a basis of $\F_2$. Let $Y_n$, for every $n \in \Z$, be the region such that $P(Y_n)$ is a conjugate of $a^nb$. The regions $Y_n$ are exactly the regions which intersect $X$ in an edge on the boundary $\partial X$ of $X$. Note that by definition $Y_0=Y_X$. 

Let $\rho : \F_2 \to \mathrm{Isom}(\mathcal{X})$ be a representation and $K>0$. 
Let us denote $A=\rho(a)$ and $B=\rho(b)$ and assume that $A$ is hyperbolic. If $B(A^+) \neq A^-$, then Proposition \ref{prop:AnB} applies, and in this case denote $N_+$ and $k_+$ the constants given by the Proposition \ref{prop:AnB} applied to $A$ and $B$. Similartly, if  $B(A^-) \neq A$, denote $N_-$ and $k_-$ the constants given by Proposition \ref{prop:AnB} applied to $A^{-1}$ and $B$. Now define :

\begin{Definition}
\begin{align*}
	J_\rho^+(K,X) & = \left\{ 
	\begin{array}{ll}
	 \bigcup \left\{X \cap Y_n \; | \; 0 \leq n \leq \max \left\{N_+, \frac{K+k_+}{l_\rho(X)}\right\} \right\} & \text{ if $A$ is hyperbolic and } B(A^+) \neq A^-, \\ 
 \bigcup \left\{X \cap Y_n \; | \; n \geq 0 \right\} &  \text{ if $A$ is hyperbolic and } B(A^+)=A^-, \\
	\end{array}
	\right. \\
	J_\rho^-(K,X) & =
	\left\{
	\begin{array}{ll}
	 \bigcup \left\{X \cap Y_n \; | \; 0 \geq n \geq -\max \left\{N_-, \frac{K+k_-}{l_\rho(X)}\right\} \right\} & \text{if $A$ is hyperbolic and } B(A^-) \neq A^+, \\ 
	 \bigcup \left\{X \cap Y_n \; | \; n \leq 0 \right\} & \text{if $A$ is hyperbolic and } B(A^-)=A^+. 
	\end{array}
	\right. \\
	J_\rho(K,X) & = 
	\left\{
	\begin{array}{ll}
	J_\rho^+(K,X) \cup J_\rho^-(K,X) & \text{if $A$ is hyperbolic} \\
	 \partial X & \text{if $A$ is not hyperbolic.}
	\end{array}
	\right.
\end{align*}
\end{Definition}

\begin{Remark}\label{rem:J(K,X)}
Assume that $A$ is hyperbolic, $B(A^+)\neq A^-$ and $B(A^-)\neq A^+$. Let $n \in \Z$ such that $X \cap Y_n$ does not belong to $J_\rho(K,X)$. Assume $n \geq 0$. Then $n >N_+$ and therefore by Proposition~\ref{prop:AnB}, $l_S(A^nB) \geq l_S(A)n-k_+$. But the integer $n$ is also bigger than $\frac{K+k_+}{l_S(A)}$, so $l_\rho(Y_n) = l_S(A^nB) > K$.  Similarly if $n \leq 0$, we also deduce $l_\rho(Y_n) >K$. Therefore, if $Z$ is a region which intersect $X$ and $l_\rho(Z) \leq K$, then $Z \cap X \subset J_\rho(K,X)$.
\end{Remark}

\begin{Definition}
	$T_\rho(K)=\underset{X \in \Omega_\rho(K)}{\bigcup} J_\rho(K,X)$.
\end{Definition}

\begin{Remark}
	It is straighforward to see that $T_{\rho}(K_1) \subset T_{\rho}(K_2)$, whenever $K_1 \leq K_2$. 
\end{Remark}

\begin{Lemma}
	 Let $\rho : \F_2\to \mathrm{Isom}(\mathcal{X})$ be an irreducible representation and $K\geq C_\delta$.  The tree $T_\rho(K)$ is connected.
\end{Lemma}

\begin{proof}
	Note that by construction, the subtrees $J_\rho(K,X)$ are connected for all $X \in \Omega$ : they are arcs in the boundary of the region $X$. Let $e$ and $e'$ be two edges in $T_{\rho}(K)$. Then there exists $X$ and $X'$ in $\Omega_\rho(K)$ such that $e \in J_\rho(K,X)$ and $e' \in J_\rho(K,X')$. Since $\bigcup \{X \; | \; X \in \Omega_\rho(K)\}$ is connected (Lemma \ref{lem:connected}), we deduce the existence of a sequence of regions $(X_i)_{0\leq i \leq n}$ such that $X=X_0$, $X'=X_n$ and $X_{i} \cap X_{i+1} \neq \emptyset$ for all $1\leq i \leq n-1$. Denote $e_i = X_i \cap X_{i+1}$ the edge at which $X_i$ and $X_{i+1}$ meet. Then $e_i \in J_\rho(K,X_i) \cap J_\rho(K,X_{i+1})$. Therefore, we deduce that there exists a path from $e$ to $e_0$ which stays in $J_{\rho}(K,X_0)$, a path from $e_i$ to $e_{i+1}$ which stays in $J_{\rho}(K,X_{i+1})$ for all $1 \leq i \leq n-1$, and a path from $e_{n-1}$ to $e$ which stays in $J_\rho(K,X_n)$. This proves that $T_\rho(K)$ is connected.\\
\end{proof}

\begin{Remark}~
Building on Remark \ref{rem:J(K,X)}, we deduce that if $e$ is an edge such that $\Omega^0(e)=\{X,Y\}$, with $l_\rho(X)\leq K$ and $l_\rho(Y) \leq K$, then $e \subset T_{\rho}(K,X)$. As a consequence, using the connectedness of $T_{\rho}(K)$, we deduce that the tree $T_\rho(C_\delta)$ contains the tree $\mathcal{T}_\rho$ defined in the previous section (\ref{subsec:tree initial}). 
\end{Remark}
	
	We say that a tree $T$ is \emph{attracting} for $\rho$, if for every edge $e$ which does not belong to $T$, the orientation $\alpha_\rho(e)$ points toward $T$. 
	Unlike in the previous works in $\mathrm{PSL}(2,\C)$ (\cite{bowditch_markoff_1998}, \cite{tan_generalized_2008}) or in $\mathrm{PU}(2,1)$ (\cite{lawton_dynamics_2025}), the tree $T_\rho(K)$ constructed here in this large-scale geometry context may not be attracting. In particular, if $X$ is a region such that $l_\rho(X) \leq C_\delta$, and $e \leftrightarrow (X,Y_n ; Y_{n-1}, Y_{n+1})$ (with the notation from above) is an edge in the boundary of $X$ which does not belong to $J_\rho(K,X)$, then a priori we could have $l_\rho(Y_{n+1}) < l_\rho(Y_{n-1})$, which would mean that the orientation $\alpha_\rho(e)$ doesn't point toward $T_\rho(K)$. Indeed, $l_\rho(Y_{n+1})=l_S(A A^nB)$, $l_\rho(Y_{n-1})=l_S(A^{-1}A^nB)$, and the length $l_S(A^nB)$ is big whereas the length $l_S(A)$ could be arbitrarily small. Since that length $l_S(A)$ is under the scale of precision of our space, we cannot deduce which inequality holds. \\
	
	However, this bad phenomena can only happen in the boundary of a small region ($l_\rho(X)\leq C_\delta$) and the tree is attracting elsewhere as noted in the following proposition : 
	\begin{Proposition} \label{prop:orientation outiside tree}
		Let $\rho : \F_2 \to \mathrm{Isom}(\mathcal{X})$ and irreducible representation and $K>C_\delta$. \\ 
		If an edge $e$ of $\Sigma$ does not belong to the tree $T_\rho(K)$ and does not belong to the boundary of a region $X$ with $l_\rho(X) \leq K$, then the orientation $\alpha_\rho(e)$ points toward $T_\rho(K)$. 
	\end{Proposition} 
	\begin{proof}
		The edge $e$ belong to some $E^{0-}(\vec{e_0})$, for some oriented edge $\vec{e_0} \leftrightarrow (X_0,Y_0 ; Z \to W)$ in the circular set $C(T_\rho(K))$. 
		\begin{itemize}
			\item Assume first that $l_\rho(X_0)>K$ and $l_\rho(Y_0)>K$. Then, similarly as in the proof of Lemma \ref{lem:circular-set} \ref{lemcase:l(Z)}, we obtain that $l_\rho(W) \leq K$ and similarly as in Lemma \ref{lem:circular-set} \ref{lemcase:l(X)} that $l_\rho(Z) >K$. As a consequence $\alpha_\rho(e_0)=\vec{e_0}$. Therefore we can apply Lemma \ref{lem:growth on big edge} to conclude that $\alpha_\rho(e)$ points toward $e_0$, hence toward $T_\rho(K)$. 
			\item Assume now that $l_\rho(X_0)\leq K$. Since we assumed that $e$ does not belong to the boundary of a region of length not bigger that $K$, we deduce that $e \in E^{0-}(\vec{e_1})$, where $\vec{e_1} \leftrightarrow (X_1,Y_1 ; Z_1 \to X_0)$ is in  $E^{-}(\vec{e_0})$. Note that by construction we also have $l_\rho(X_1)>K$, $ l_\rho(Y_1)>K$ and $l_\rho(Z_1)>K$. Then $\alpha_\rho(e_1)=\vec{e_1}$ and therefore using Lemma \ref{lem:growth on big edge} we obtain that the orientation on $\alpha_\rho(e)$ points toward $e_1$, hence $e_0$, hence the tree $T_\rho(K)$.
		\end{itemize}
	\end{proof}
	
		In addition, along the boundary of a region $X$ such that $l_\rho(X) \leq K$, we can also argue that the orientation cannot be ``escaping", meaning that if $(\vec{e_n})_{n \in \N}$ is a sequence of oriented edges lying in the boundary of $X$ and not in the tree $T_\rho(K)$ such that the head of $\vec{e_n}$ is the tail of $\vec{e_{n +1}}$, then the orientation $\alpha_\rho(e_n)$ cannot enventually coincide with that of $\vec{e_n}$. Indeed, first note that the existence of the sequence $(\vec{e_n})_{n\in \N}$ in the boundary of $X$ and not in $T_\rho(K)$ implies that $A:=\rho(P(X))$ is hyperbolic and that, denoting $B=\rho(P(Y_X))$, $B(A^+)\neq A^+$ or $B(A^-)\neq A^+$ depending on which side of the boundary of $X$ the sequence $(\vec{e_n})_{n \in \N}$ lies in. Then, in view of Proposition \ref{prop:AnB}, we have $l_S(A^nB)\geq l_S(A)|n|-k$ for $k \geq 0$ a constant and $|n|$ large enough. But if $\alpha_\rho(e_n)=\vec{e_n}$ for $n$ sufficiently large, this would also imply that the sequences $(l_S(A^{2n}B))_{n \in \N}$ and $(l_S(A^{2n+1}B))_{n\in \N}$ are bounded, which of course contradicts the conclusion of Proposition \ref{prop:AnB}. 
		
		We could also certainly further strenghten this result by showing the existence of an integer $p$ satisfying the following. If $(\vec{e_n})_{n \in \N}$ is a sequence as above, then the orientation $\alpha_{\rho}(e_n)$ can only coincide with the orientation $\vec{e_n}$ on at most $p$ consecutive edges. \\

	Let us now show that the finitess of this tree characterizes the representations satisfying \ref{def:BQa}. 
\begin{Proposition} \label{prop:eq finite tree}
	Let $K\geq K_\delta$. Let $\rho : \F_2 \to \mathrm{Isom}(\mathcal{X})$ be an irreducible representation. \\ 
	The representation $\rho$ satisfy \ref{def:BQa} if and only if the tree $T_\rho(K)$ is finite.
\end{Proposition}

\begin{proof}
	First assume that $\rho$ satisfies \ref{def:BQa}. We proved in Theorem \ref{thm:equivalence a,b,d} that this implies that $\rho$ satisfies \ref{def:BQb}. Therefore the set $\Omega_\rho(K)$ is finite. Moreover, for all basis $\{a,b\}$ of $\F_2$, $B(A^+) \neq A^-$ and $B(A^-) \neq A^+$ by Proposition~\ref{prop:B(A+)neqA-explicite}. Then, the arcs $J_\rho(K,X)$ are finite for all $X \in \Omega_\rho(K)$ and the tree $T_\rho(K)$ is finite. \\
	Now assume that the tree $T_\rho(K)$ is finite. Then, since the arcs $J_\rho(K,X)$ are never empty and an edge is contained in the boundary of at most two regions, we deduce that the set $\Omega_\rho(K)$ must be finite hence also $\Omega_\rho(K_\delta)$. Then the set of regions $X$ such that $\rho(P(X))$ is not hyperbolic must be contained in $\Omega_\rho(K)$, but in this case, the arc $J_\rho(K,X)$ would be infinite. Therefore there are no region $X \in \Omega$ such that $\rho(P(X))$ is not hyperbolic and hence $\rho$ satisfy \ref{def:BQa}.\\
\end{proof}

\subsection{Openness and dynamics of $\mathrm{Out}(\F_2)$}~ \label{subsec:openness and dynamics}\\
\indent We can now give the key Lemma which implies openness and proper discontinuity. Compare with the proof of Theorem 3.16 in \cite{bowditch_markoff_1998}, Theorem 3.2 in \cite{tan_generalized_2008} or Lemma 4.12 in \cite{lawton_dynamics_2025}. The main difference with these previous works here is that the finite tree $T$ will not be of the form $T_\rho(K)$. The fact that we need to consider a bigger tree than $T_\rho(K)$ comes from the lack of continuity of the constants found in Proposition \ref{prop:AnB} (and which we overcome by the nearly continuity result of Lemma~\ref{lem:nearly-continuity}). \\

Recall that $\mathrm{Isom}(\mathcal{X})$, endowed with the compact-open topology, is a topological group and that we endow $\mathrm{Hom}(\F_2,\mathrm{Isom}(\mathcal{X}))$ with the compact-open topology, which coincide (since $\F_2$ is finitely generated) with the topology of the pointwise convergence.
\begin{Proposition}\label{prop:neighborhood finite tree}
	Let $\rho : \F_2 \to \mathrm{Isom}(\mathcal{X})$ a representation which satisfies \ref{def:BQa}. Then, there exists a constant $K \geq K_\delta$, $T$ a finite subtree in $\Sigma$ and a neighborhood $\mathcal{U}_\rho$ of $\rho$ in $\mathrm{Hom}(\F_2,\mathrm{Isom}(\mathcal{X}))$, such that for all $\rho' \in \mathcal{U}_\rho$, the trees $T_{\rho'}(K)$ are contained in $T$.
\end{Proposition}

\begin{proof}
	Let $K_0 \geq K_\delta$ such that $T_\rho(K_0)$ is non-empty and let $K>K_0$.
	Then there is at least one region~$X$ such that $l_\rho(X)\leq K_0<K$. We can then ensure that, if $\rho'$ is in a sufficiently small neighborhood of $\rho$, $l_{\rho'}(X) < K$ by continuity of the stable norm (Theorem \ref{thm:l_S-continuous}), hence 
	$T_{\rho'}(K)$ is not empty. Moreover, the edge $X\cap Y_X$, which belong both to $J_{\rho}(K,X)$ and $J_{\rho'}(K,X)$, belong both to $T_{\rho}(K)$ and $T_{\rho'}(K)$. Thus $T_\rho(K) \cap T_{\rho'}(K) \neq \emptyset$. 
	We are now going to define a finite tree $T$ containing $T_{\rho}(K)$. For every $X \in \Omega_{\rho}(K)$, denote $A=\rho(P(X))$, $B=\rho(P(Y_X))$. Then, since $\rho$ satisfies \ref{def:BQa}, $A$ is hyperbolic, $B(A^+)\neq A^-$, and $B(A^-)\neq A^+$, so we can apply Lemma \ref{lem:nearly-continuity} to the two pairs $(A,B)$ and  $(A^{-1}, B)$ and then we introduce the constant $C=C(l_S(A),\delta)=C(l_S(A^{-1}),B)$ given by Lemma \ref{lem:nearly-continuity}. Now define 
	\begin{align*}
		\tilde{J}_\rho(K,X) & =\left\{X\cap Y_n \; | \; -\max \left\{N_-, \frac{K+k_-}{l_\rho(X)} \right\}-C \leq n \leq \max \left\{N_+,\frac{K+k_+}{l_{\rho}(X)} \right\}+C  \right\} \\
		\text{and } \qquad \qquad T & = \underset{X \in \Omega_\rho(K)}{\bigcup} \tilde{J}_\rho(K,X).
	\end{align*}
	It is immediate that $T$ is a finite tree containing $T_\rho(K)$. Now consider its circular set $C(T)$. This set is finite since $T$ is. Let $\vec{e} \in C(T)$ and let us prove that $e \notin T_{\rho'}(K)$, for $\rho'$ sufficiently close to $\rho$. Denote $\Omega^0(e)=\{X,Y\}$. There are three cases :
	\begin{itemize}
		\item If $l_\rho(X) >K$ and $l_\rho(Y)>K$, then in a small neighborhood of $\rho$, we still have $l_{\rho'}(X)>K$ and $l_{\rho'}(Y)>K$, so $e \notin T_{\rho'}(K)$. 
		\item If $l_\rho(X)\leq K$ and $l_\rho(Y) \leq K$, then this means that $e \in T_\rho(K)$, which is impossible since $e\in C(T)$. 
		\item If $l_\rho(X)\leq K$ and $l_\rho(Y)>K$, then after small deformation, we still have $l_{\rho'}(Y)>K$. The edge $e $ belong to the boundary $\partial X$ of $X$ so there exists $n \in \Z$ such that $Y=Y_n$. Since $e \notin T$, then $e \notin \tilde{J}_\rho(K,X)$ and as a consequence, $n>\max \{N_+,\frac{K+k_+}{l_\rho(X)}\}+C$ if $n \geq 0$ and $n< \max\{N_-, \frac{K+k_-}{l_\rho(X)}\}-C$ if $n \leq 0$. Assume that $n \geq 0$ (the case $n \leq 0$ is similar). Lemma \ref{lem:nearly-continuity} implies that if $\rho'$ belong to a sufficiently small neighborhood of $\rho$, then $A'=\rho'(P(X))$ is hyperbolic, $B'(A'^+)\neq A'^-$, with $B'=\rho'(P(Y_X))$, and  $n>\max\{N'_+,\frac{K+k'_+}{l_{\rho'}(X)}\}$, with $N'_+$ and $k'_+$ the constant given by Proposition \ref{prop:AnB} applied to $A'$ and $B'$. Therefore, $e \notin J_{\rho'}(K,X)$ and thus $e \notin T_{\rho'}(K)$. 
	\end{itemize}
  Then in any case, we proved that if $\vec{e} \in C(T)$, $e \notin T_{\rho'}(K)$. Since $T_{\rho'}(K) \cap T \neq \emptyset$ and $T_{\rho'}(K)$ is connected, this implies that $T_{\rho'}(K) \subset T$.
\end{proof}

\begin{Corollary}
	The set of Bowditch representation is open.
\end{Corollary}
\begin{proof}
      This is immediate from Proposition \ref{prop:eq finite tree} and Proposition \ref{prop:neighborhood finite tree}.
\end{proof}

Recall that the action of the outer automorphism group $\mathrm{Out}(\F_2)$ on $\chi(\F_2,\mathrm{Isom}(\mathcal{K}))$ is given by $[\Phi].[\rho]=[\rho \circ \Phi^{-1}]$, where $[\Phi] \in \mathrm{Out}(\F_2)$ and $[\rho] \in \chi(\F_2,\mathrm{Isom}(\mathcal{X})$.
\begin{Proposition}
	The group $\mathrm{Out}(\F_2)$ acts properly discontinuously on the set of Bowditch representations. 
\end{Proposition}

\begin{proof} Let $\rho : \F_2 \to \mathrm{Isom}(\mathcal{X})$ be a representation which satisfies \ref{def:BQa}. Then Proposition \ref{prop:neighborhood finite tree} implies the existence of finite tree $T$ and a neighborhood of $\rho$ consisting of representations $\rho'$ with $T_{\rho'}(C_\delta) \subset T$. Therefore we deduce, using the proof of Propositions \ref{lem:growth on big edge} and \ref{lem:growth when l(X) petit} (and as in the proof of Theorem 2.3 in \cite{tan_generalized_2008}) that the constant $C$ in the characterization \ref{def:BQd} of $\rho'$ can be defined using only a finite number of regions in $\Omega$, therefore may be chosen independently of $\rho'$ in the neighborhood. As a consequence, it follows that the same constant in characterization \ref{def:BQd} can be chosen on compact sets of Bowditch representations. \\
Let $\mathcal{K} \subset \mathrm{BQ}(\F_2,\mathrm{Isom}(\mathcal{X}))$ be a compact set and consider $E_{\Phi}=\{[\Phi] \in \mathrm{Out}(\F_2) \; | \; [\Phi](\mathcal{K}) \cap \mathcal{K} \neq \emptyset\}$. By the previous discussion, there exists a constant $C>0$ such that for all $[\Phi] \in E_\mathcal{K}$ and $[\rho] \in \mathcal{K}$ such that $[\Phi].[\rho] \in \mathcal{K}$, and for all $[\gamma] \in \mathcal{P}(\F_2)$, $l_S(\Phi.\rho(\gamma)) \geq \frac{1}{C} \Vert \gamma \Vert $. Therefore,
\begin{equation}\label{eq:prop-disc lower bound l_S}
	l_S(\rho(\gamma)) = l_S(\rho(\Phi^{-1}\Phi \gamma))=l_S(\Phi.\rho(\Phi\gamma)) \geq \frac{1}{C} \Vert \Phi \gamma \Vert.
\end{equation}
In addition, note that by the triangle inequality, there exists a constant $C'>0$ such that $l_S(\rho'(\gamma)) \leq C'\Vert \gamma \Vert $ for all $\gamma \in \mathcal{P}(\F_2)$, and that this constant can be made independent of $\rho'$ in a small neighborhood of $\rho$, and thus also uniform on compact set. Therefore, for all $[\Phi] \in E_\mathcal{K}$, $[\rho] \in \mathcal{K}$ and $[\gamma] \in \mathcal{P}(\F_2)$,
\begin{equation}\label{eq:prop-disc upper bound l_S}
	l_S(\rho(\gamma)) \leq C'\Vert \gamma \Vert .
\end{equation}
From \eqref{eq:prop-disc lower bound l_S} and \eqref{eq:prop-disc upper bound l_S}, we deduce that for all $\Phi \in E_\mathcal{K}$, for all $\gamma \in \mathcal{P}(\F_2)$, $ \Vert \Phi\gamma \Vert \leq CC' \Vert \gamma \Vert$. Since only a finite number of outer automorphism $\Phi$ can satisfying this inequality, we deduce that $E_\mathcal{K}$ is finite, hence the action of $\mathrm{Out}(\F_2)$ is properly discontinuous. \\
\end{proof}

\subsection{Recognition of Bowditch representations}~\label{subsec:recognition} \\
\indent We are now going to explain how our previous results and the tree $T_\rho(K)$ defined in section \ref{subsec:treeJ} allows us to write an algorithm to recognize Bowditch representations in this setting. Note that in contrast with the previous works in $\mathrm{SL}(2,\C)$ of Bowditch \cite{bowditch_markoff_1998}, further generalized by Tan-Wong-Zhang \cite{tan_generalized_2008}, and in $\mathrm{SU}(2,1)$ of Lawton-Maloni-Palesi \cite{lawton_dynamics_2025}, the tree $T_\rho(K)$ here is not fully attracting, and that we don't have a description of the belonging of an edge $e$ to the tree $T_\rho(K)$ as nice and explicite as in their works, where an explicite function $H(x)$, with $x=l_\rho(X)$ and $X \in \Omega^0(e)$ was sufficient to compute in order to characterize the belonging to the tree. \\

The following algorithm, given an irreducible representation $\rho : \F_2\to \mathrm{Isom}(\mathcal{X})$ and an arbitrary edge $e$ of the tree $\Sigma$, follows the orientation $\alpha_\rho$ until it reaches a sink or a region with small length. 
Recall that the notion of a \emph{sink} has been defined just before Proposition \ref{lem:sink}.

\begin{Algorithm}[Find small region or sink.] \label{alg:small region or sink}  ~\\
	$\mathrm{INPUT} :$ $(\rho,e)$ with $\rho : \F_2\to \mathrm{Isom}(\mathcal{X})$ an irreducible representation and $e \leftrightarrow (X,Y;Z,W)$ an edge of the tree $\Sigma$. \\
 Compute $l_\rho(X)$, $l_\rho(Y)$, $l_\rho(Z)$ and $l_\rho(W)$.
	\begin{itemize}
		\item If $l_\rho(X) \leq C_\delta$, or $l_\rho(Y) \leq C_\delta$, then denote $R \in \{X,Y\}$ such that $l_\rho(R) \leq C_\delta$, and exit the procedure. $\mathrm{OUTPUT} :$ $(e,R)$. 
		\item If $l_\rho(Z) \leq C_\delta$ or $l_\rho(W)\leq C_\delta$, then denote $R \in \{Z,W\}$ such that $l_\rho(R) \leq C_\delta$,  $e=X\cap R$, and exit the procedure. $\mathrm{OUTPUT :}$ $(e,R)$.
		\item If all these four lengths are greater than $C_\delta$, then compare $l_\rho(Z)$ and $l_\rho(W)$ to compute the orientation $\alpha_\rho(e)$. Denote $v$ the head of $\alpha_\rho(e)$. Up to exchanging $Z$ and $W$, let us assume that $v \leftrightarrow (X,Y,Z)$. Then Proposition \ref{lem:sink} applies, and we deduce that $v$ is the tail of at most one edge. Compute the orientation $\alpha_\rho$ on $X \cap Z$ and $Y\cap Z$. 
		\begin{itemize}
			\item If $v$ is not the tail of an edge, it means that $v$ is a sink such that all three regions around have length greater that $C_\delta$. Then exit the procedure. \\
			$\mathrm{OUTPUT : }$ ``The representation $\rho$ is a Bowditch representation." 
			\item If $v$ is the tail of exactly one edge $e'$, then run Algorithm \ref{alg:small region or sink} with input 
			$(\rho,e')$.
		\end{itemize}
	\end{itemize}
\end{Algorithm}

\begin{Proposition} \label{prop:alg find small region or sink}
	Let $\rho : \F_2\to \mathrm{Isom}(\mathcal{X})$ be an irreducible representation. 
	\begin{itemize}
		\item If $\rho$ is a Bowditch representation, then Algorithm \ref{alg:small region or sink} ends in finite time.  
		\item If Algorithm \ref{alg:small region or sink} ends in finite time, then either it found an edge $e$ together with a region $X \in \Omega^0(e)$ such that $l_\rho(X) \leq C_\delta$ or it found a sink $v \leftrightarrow (X,Y,Z)$ with $l_\rho(X)>C_\delta$, $l_\rho(Y)>C_\delta$ and $l_\rho(Z)>C_\delta$, and in this case, $\rho$ is a Bowditch representation. 
	\end{itemize}
\end{Proposition}

\begin{proof}
	\begin{itemize}
		\item Assume that $\rho$ is a Bowditch representation. If $\Omega_\rho(C_\delta)=\emptyset$, by Proposition \ref{lem:sink} and Lemma~\ref{lem:growth on big edge}, there exists a unique sink $v$ and all the edges of the tree are oriented toward this sink. Therefore Algorithm \ref{alg:small region or sink} will eventually find this sink and ends. If $\Omega_\rho(C_\delta) \neq \emptyset$, then $T_\rho(C_\delta) \neq \emptyset$ and Proposition \ref{prop:orientation outiside tree} ensures that Algorithm \ref{alg:small region or sink} will eventually find an edge $e$ such that the head of $\alpha_\rho(e)$ lies in the boundary $\partial X$ of some region $X$ in $\Omega_\rho(C_\delta)$. Therefore the algorithm stops. 
		\item It is clear that if Algorithm \ref{alg:small region or sink} ends, it either finds a couple $(e,X)$ with $e$ an edge and $X \in \Omega^0(e)$ or it finds a sink. If it finds a sink, then Proposition \ref{lem:sink} ensures that $\rho$ is a Bowditch representation. 
	\end{itemize}
\end{proof}

Now assuming that an edge $e$ is given together with $X \in \Omega^0(e) \cap \Omega_\rho(C_\delta)$, Algorithm \ref{alg:explore tree} explores recursively the set of regions $\Omega_\rho(C_\delta)$. The list $L$ represents the set of regions which have already been explored so far. 
\begin{Algorithm} \label{alg:explore tree}
	$\mathrm{INPUT} : (\rho,e,X,L)$ with $\rho : \F_2\to \mathrm{Isom}(\mathcal{X})$ an irreducible representation, $e \leftrightarrow (X,Y;Z,W) \in E(\Sigma)$ an edge of the tree $\Sigma$, $X \in \Omega^0(e)$ and $L \subset \Omega$ a list of regions.\\
	\begin{itemize}
	\item Compute $l_S(\rho(X))$.
	\item If $l_S(\rho(X)) \leq C_\delta$, then do :
	\begin{enumerate}[label=\arabic*.]
		\item \label{alg:test hyperbolicity} If $l_\rho(X)=0$, then exit the procedure. \\
			$\mathrm{OUTPUT} :$ ``The representation $\rho$ is not a Bowditch representation."
		\item \label{alg:test B(A^+)} Denote $A=\rho(P(X))$ and $B=\rho(P(Y))$. \\
		If $B(A^+)=A^-$ or $B(A^-) =A^+$, then exit the procedure. \\
		$\mathrm{OUTPUT} :$ ``The representation $\rho$ is not a Bowditch representation.`"
		\item \label{alg:compute N k} Compute $N_+$ and $k_+$ the constants given by Proposition \ref{prop:AnB} for $A$ and $B$ and $N_-$ and $k_-$ the constants given by Proposition \ref{prop:AnB} for $A^{-1}$ and $B$. Let $n_+ = \max\{N_+,\frac{C_\delta+k_+}{l_S(A)}\}$ and $n_- = \max\{N_-,\frac{C_\delta+k_-}{l_S(A)}\}$. Denote $(Y_n)_{n\in \N}$ the sequence of regions meeting $X$ such that $P(Y_n)=P(X)^nP(Y)$, for all $n \in \Z$. Let $L:=L\cup \{X\}$. \\
		For $-n_- \leq n \leq n_+$, if $Y_n \notin L$, then 
		denote $e':=X\cap Y_n$, $X':=Y_n$ and run algorithm \ref{alg:explore tree} with  new input $(\rho, e',X',L)$. 
	\end{enumerate}
	\end{itemize}
\end{Algorithm}

\begin{Proposition} \label{prop:alg explore tree}
	Let $\rho : \F_2\to \mathrm{Isom}(\mathcal{X})$ be an irreducible representation, $e \in E(\Sigma)$ be an edge and $X \in \Omega$ be a region such that $l_\rho(X) \leq C_\delta$. 
	\begin{itemize}
		\item If $\rho$ is a Bowditch representation, then Algorithm \ref{alg:explore tree} with input $(\rho, e,X,\emptyset)$ ends in finite time.
		\item If Algorithm \ref{alg:explore tree} with input $(\rho, e,X,\emptyset)$ such that $l_\rho(X) \leq C_\delta$ ends in finite time, then either $\rho$ is not a Bowditch representation and the output says it, or $\rho$ is a Bowditch representation and there is no output. 
	\end{itemize}
\end{Proposition}

\begin{proof}
	\begin{itemize}
		\item Assume that Algorithm \ref{alg:explore tree} with input $(\rho,e,X,[])$ runs forever, then it means that Algorithm \ref{alg:explore tree} will pass an infinite number of time through step \ref{alg:compute N k}, and therefore that there exists an infinite number of regions $R$ satisfying $l_\rho(R) \leq C_\delta$. Hence $\rho$ does not satisfy \ref{def:BQa}.
		\item Assume that Algorithm \ref{alg:explore tree} ends in finite time. 
		\begin{itemize}
			\item If Algorithm \ref{alg:explore tree} ends at step \ref{alg:test hyperbolicity}, then it found a region $X$ such that $\rho(P(X))$ is not hyperbolic, which prevents $\rho$ from satisfying \ref{def:BQa}. If is consistent with the output.
			\item If Algorithm \ref{alg:explore tree} ends at step \ref{alg:test B(A^+)}, then it found a basis $\{a,b\}$ of $\F_2$ so that, denoting $A=\rho(a)$ and $B=\rho(b)$, we have either $B(A^+)=A^-$ or $B(A^-)=A^+$ which again prevents $\rho$ from satisfying \ref{def:BQa} by Proposition \ref{prop:B(A+)neqA-explicite}. It is consistent with the output.
			\item If Algorithm \ref{alg:explore tree} ends after step \ref{alg:compute N k}, then it means that it has explored a finite number of region $X$ such that $l_\rho(X)\leq C_\delta$. Proposition \ref{prop:AnB} ensures that for all $n>n_+$ and all $n<n_-$, $l_\rho(Y_n) >C_\delta$ (with the notation of step \ref{alg:compute N k}). 
			Combined with the connectivity result of Lemma \ref{lem:connected}, we deduce all the regions $R$ such that $l_\rho(R) \leq C_\delta$ have been explored by the algorithm, and thus that $\Omega_\rho(C_\delta)$ is finite. Since the algorithm also checked that for all these regions $R$, $\rho(P(X))$ is hyperbolic, this ensures that $\rho$ is a Bowditch representation. In this case there is no output. 
		\end{itemize}
	\end{itemize}
\end{proof}

We can now combine Algorithm \ref{alg:small region or sink} and Algorithm \ref{alg:explore tree} to give the main algorithm which certifies that a given representation is a Bowditch representation. 
\begin{Algorithm}[Recognition of Bowditch representations]\label{alg:main} ~\\
	$\mathrm{INPUT : }$ $\rho : \F_2 \to \mathrm{Isom}(\mathcal{X})$ a representation. 
	\begin{enumerate}[label=\arabic*.]
		\item Pick $e \leftrightarrow (X,Y;Z,W)$ any edge of the tree $\Sigma$.
		\item If $l_\rho(X)=0$ or $l_\rho(Y)=0$, then exit the procedure. \\ $\mathrm{OUTPUT : }$ ``The representation $\rho$ is not a Bowditch representation."
		\item \label{alg:step-irreducibility} Denote $A=\rho(P(X))$, $B=\rho(P(Y))$. \\
		 If $\{A^+,A^-\}\cap \{B^+,B^-\} \neq \emptyset$, then exit the procedure. \\
			$\mathrm{OUTPUT :}$ ``The representation $\rho$ is not irreducible, hence not a Bowditch representation."
		\item Run algorithm \ref{alg:small region or sink} with input $(\rho,e)$.
		\begin{itemize}
			\item If Algorithm \ref{alg:small region or sink} ends and finds a sink $v$, then exit the procedure. \\
			$\mathrm{OUTPUT :}$ ``The representation $\rho$ is a Bowditch representation." 
			\item If Algorithm \ref{alg:small region or sink} ends and finds an edge $e$ together with a region $X \in \Omega^0(e)$ satisfying $l_\rho(X) \leq C_\delta$, continue.
		\end{itemize} 
		\item Run Algorithm \ref{alg:explore tree} with input $(\rho,e,X,[])$. 
		\begin{itemize}
		\item If Algorithm \ref{alg:explore tree} ends and output that $\rho$ is not a Bowditch representation, then\\
		$\mathrm{OUTPUT :}$ ``The representation $\rho$ is not a Bowditch representation."
		\item If Algorithm \ref{alg:explore tree} ends with no output, then\\
		$\mathrm{OUTPUT :}$ ``The representation $\rho$ is a Bowditch representation.".
		\end{itemize}
	\end{enumerate}
\end{Algorithm}

\begin{Theorem} Let $\rho : \F_2\to \mathrm{Isom}(\mathcal{X})$ be a representation. 
	\begin{itemize}
	\item If $\rho$ is a Bowditch representation, then Algorithm \ref{alg:main} ends in finite time.
	\item If Algorithm \ref{alg:main} ends in finite time, then it decides whether $\rho$ is a Bowditch representation. 
	\end{itemize}
\end{Theorem}

\begin{proof}
	This is just a combination of Proposition \ref{prop:alg find small region or sink} and Proposition \ref{prop:alg explore tree}. 
\end{proof}

Lastly, we end with a remark and a small complement on the algorithms. Note that the procedure requires to be able to test if two points in the boundary of the hyperbolic space $\mathcal{X}$ are equal. Indeed, step~\ref{alg:step-irreducibility} in Algorithm~\ref{alg:main} checks irreducibility by checking the equivalent condition $\{A^+,A^-\}\cap \{B^+,B^-\} =\emptyset$ and step \ref{alg:test B(A^+)} in Algorithm~\ref{alg:explore tree} asks if $B(A^+)\neq A^-$ and $B(A^-)\neq B^+$. 

Alternatively, we can give  algorithms to certify these conditions by doing only computation in the space~$\mathcal{X}$. \\

The following algorithm, given two parametrized geodesic, checks if their endpoints at $+\infty$ coïncide.
\begin{Algorithm}[Irreducibility] \label{alg:irreducibility}
	$\mathrm{INPUT :}$ $(l_1,l_2)$ with $l_1 : \R \to \mathcal{X}$ and $l_2 : \R \to \mathcal{X}$ two parametrized geodesics. 
	\begin{enumerate}[label=\arabic*.]
	\item Compute $N=\max\{6\delta,d(l_1(0),l_2(0))\}+d(l_1(0),l_2(0))+\delta$ and set $n=N$.  
	\item \label{alg:step;compute distance} Compute $d(l_1(n),l_2)$. 
	\begin{itemize}
		\item If $d(l_1(n),l_2) >2\delta$, then exit the procedure. $\mathrm{OUTPUT :} $ $l_1(+\infty) \neq l_2(+\infty)$. 
		\item If $d(l_1(n),l_2) \leq 2\delta$, then $n:=n+1$ and restart step \ref{alg:step;compute distance}. 
	\end{itemize} 
	\end{enumerate}
\end{Algorithm}

\begin{Proposition}
 Let $l_1 : \R \to \mathcal{X}$ and $l_2 : \R \to \mathcal{X}$ be two parametrized geodesics. \\
Algorithm \ref{alg:irreducibility} ends in finite time if and only if $l_1(+\infty)\neq l_2(+\infty)$.
\end{Proposition}
\begin{proof}
	\begin{itemize}
		\item If $l_1(+\infty)\neq l_2(+\infty)$, then there exists $n \geq N$ sufficiently large such that $d(l_1(n),l_2)>2\delta$, hence the algorithm ends in finite time.
		\item If $l_1(+\infty)=l_2(+\infty)$, by applying Lemma \ref{lem:pos-proj} we obtain that for all $n\geq N$, $d(l_1(n),l_2) \leq 2\delta$ and hence the algorithm runs forever. 
	\end{itemize}

\end{proof}
\begin{Remark} In $\mathrm{SL}(2,\C)$, the trace of the commutator characterizes irreducibility. Indeed, a representation $\rho : \F_2 \to \mathrm{SL}(2,\C)$ is irreducible if and only if $\mathrm{Tr}(\rho([a,b]))=2$, for any free generating set $\{a,b\}$ of $\F_2$. Thus, this condition can be computed easily. Irreducibility is an $\mathrm{Out}(\F_2)$-invariant condition which in this case corresponds to the slice $\mathrm{Tr}(\rho([a,b]))=2$ of the $\mathrm{SL}(2,\C)$-character variety.
\end{Remark}

Let us also remark that our work allows us to certify the condition $B(A^+)\neq A^-$ by computing only the lengths $l_S(A^nB)$.
\begin{Algorithm}[$B(A^+)\neq A^-$] \label{alg:B(A^+)neqA^-} 
	$\mathrm{INPUT :}$ $A$ a hyperbolic isometry and $B$ an isometry. 
	\begin{enumerate}[label=\arabic*.]
		\item Compute the constant $N$ given by Proposition \ref{prop:B(A+)neqA-explicite} and set $n=N$. 
		\item \label{step:compute-lS(AnB)} Compute $l_S(A^nB)$.
		\begin{itemize}
			\item If $l_S(A^nB) >30\delta$, then exit the procedure. $\mathrm{OUTPUT :}$ $B(A^+)\neq A^-$.
			\item If $l_S(A^nB) \leq 30\delta$, then $n:=n+1$ and restart step \ref{step:compute-lS(AnB)}. 
		\end{itemize}
	\end{enumerate}
\end{Algorithm}
\begin{Proposition}Let $A$ be a hyperbolic isometry and $B$ be an isometry. \\
	Algorithm \ref{alg:B(A^+)neqA^-} ends in finite time if and only if $B(A^+)\neq A^-$.
\end{Proposition}

\begin{proof}
	If $B(A^+) \neq A^-$, then Proposition \ref{prop:AnB} ensures the existence of an integer $n \geq N$ such that $l_S(A^nB) >30\delta$ and if $B(A^+)=A^-$, Proposition \ref{prop:B(A+)neqA-explicite} asserts that for all $n \geq N$, $l_S(A^nB) \leq 30\delta$. 
\end{proof}

\begin{Remark} In $\mathrm{SL}(2,\C)$, the condition $B(A^+)\neq A^-$ and $B(A^-)\neq A^+$ can also be certified easily using the trace of the commutator. Indeed, one can check that $B(A^+)\neq A^-$ and $B(A^-)\neq A^+$ if and only if $\mathrm{Tr}(A)^2 \neq \mathrm{Tr}([A,B])+2$. In particular, consider the $\mathrm{Out}(\F_2)$-invariant slices of the character variety $\chi(\F_2,\mathrm{SL}(2,\C))$ defined by $\{ \mathrm{Tr}(\rho[a,b])=\mu\}$, for $\mu \in \C$. Then for every representation $\rho : \F_2 \to \mathrm{SL}(2,\C)$, the condition $B(A^+)\neq A^-$ and $B(A^-)\neq A^+$ depends only on the trace of $\rho(a)$ and on which slice the representation $\rho$ lies in.  
\end{Remark}

\printbibliography
\end{document}